\documentclass[11pt]{article}

\usepackage{latexsym}
\usepackage{a4wide}
\usepackage{amscd}
\usepackage{graphics}
\usepackage{amsmath}
\usepackage{amssymb}
\usepackage{esint}
\usepackage{mathrsfs}
\usepackage{amsthm}
\usepackage{bbm}
\usepackage{color}
\usepackage{accents}
\usepackage{enumerate}
\usepackage{mathtools}
\usepackage{wasysym}
\usepackage[pdftex]{hyperref}
\usepackage{tikz}
\usetikzlibrary{matrix,arrows}

\newcommand{\N}{\mathbb{N}}                     
\newcommand{\Z}{\mathbb{Z}}                     
\newcommand{\Q}{\mathbb{Q}}                     
\newcommand{\R}{\mathbb{R}}                     
\newcommand{\C}{\mathbb{C}}                     
\newcommand{\T}{\mathbb{T}}                     
\newcommand{\D}{\mathbb{D}}                     
\newcommand{\Ker}{\mathrm{Ker\,}}               
\newcommand{\coker}{\mathrm{coker\,}}           
\newcommand{\ind}{\mathop\mathrm{ind}}          
\newcommand{\Det}{\mathrm{Det}}                 

\newcommand{\MA}{\mathcal{A}}
\newcommand{\MB}{\mathcal{B}}
\newcommand{\MD}{\mathcal{D}}
\newcommand{\ME}{\mathcal{E}}
\newcommand{\MM}{\mathcal{M}}

\newcommand{\MP}{\mathcal{P}}
\newcommand{\MS}{\mathcal{S}}
\newcommand{\MT}{\mathcal{T}}
\newcommand{\MU}{\mathcal{U}}

\newtheorem{mainthm}{\sc Theorem}           
\newtheorem{maincor}{\sc Corollary}         
\newtheorem{thm}{\sc Theorem}[section]      
\newtheorem*{thm*}{\sc Theorem}             
\newtheorem{cor}[thm]{\sc Corollary}        
\newtheorem*{cor*}{\sc Corollary}           
\newtheorem{lem}[thm]{\sc Lemma}            
\newtheorem{prop}[thm]{\sc Proposition}     
\newtheorem{defn}[thm]{\sc Definition}      
\newtheorem{rem}[thm]{\sc Remark}           
\newtheorem{claim}[thm]{\sc Claim}

\title{On closed characteristics of minimal action on a convex three-sphere}
\author{A.\ Abbondandolo, O.\ Edtmair, J.\ Kang}
\date{}

\begin{document}

\maketitle

\begin{abstract}
We prove that every closed characteristic of minimal action on the boundary of a uniformly convex domain in $\R^4$ bounds a disk-like global surface of section. A corollary is that the cylindrical symplectic capacity of a convex body in $\R^4$ coincides with the minimal action of a closed generalized characteristic on its boundary.
\end{abstract}

\section*{Introduction}

\paragraph{The main theorem.} Consider the vector space $\R^{2n}$ with the standard symplectic form
\[
\omega_0 \coloneqq \sum_{j=1}^n dx_j \wedge dy_j.
\]
Let $X\subset \R^{2n}$ be a smooth bounded domain\footnote{In this paper, a \textit{domain} is the closure of an open set.}. We call $X$ \textit{uniformly starshaped} if it is starshaped with respect to a point and the boundary $\partial X$ is transverse to all lines through this point. If in addition the sectional curvatures of $\partial X$ are strictly positive everywhere, we say that $X$ is \textit{uniformly convex}.

In \cite{rab78}, Rabinowitz proved that the boundary of a uniformly starshaped domain $X$ always admits a closed characteristic, i.e.\ a closed integral curve of the characteristic line bundle $\mathcal{L}_{\partial X}$, which is defined as the kernel of the restriction of $\omega_0$ to the tangent spaces of $\partial X$. For uniformly convex domains, the same result was independently obtained by Weinstein \cite{wei78} and Clarke \cite{cla79}.

The  action of a closed characteristic $\gamma$ is the quantity
\[
\mathcal{A}(\gamma) \coloneqq \int_{\gamma} \lambda_0,
\]
where $\lambda_0$ is any primitive of $\omega_0$ and $\gamma$ is oriented according to the natural orientation of $\mathcal{L}_{\partial X}$. This orientation is determined by any nowhere vanishing section $w$ of $\mathcal{L}_{\partial X}$ such that $\omega_0(\nu,w)>0$, where the vector $\nu$ is pointing out of $X$. If $X$ is uniformly starshaped, then the action of any closed characteristic on $\partial X$ is positive and the infimum over all actions of closed characteristics is achieved by some closed characteristic. We call any such closed characteristic of minimal action a \textit{systole}.

As observed by Weinstein in \cite{wei79}, closed characteristics on the boundary of a uniformly starshaped domain $X$ can also be seen as periodic orbits of a Reeb flow on $\partial X$. Indeed, up to a translation we can assume that the center of starshapedness of $X$ is the origin, so that the restriction $\lambda:= \lambda_0|_{\partial X}$ of the radial primitive 
\[
\lambda_0 := \frac{1}{2} \sum_{j=1}^n ( x_j \, dy_j - y_j\, dx_j)
\]
of $\omega_0$ is a contact form on $\partial X$. Its Reeb vector field $R_{\lambda}$ spans the characteristic line bundle and is positively oriented. The action $\mathcal{A}(\gamma)$ coincides with the period of $\gamma$ as periodic orbit of $R_{\lambda}$.

In their seminal paper \cite{hwz98}, Hofer, Wysocki, and Zehnder proved that the Reeb flow on the boundary of a uniformly convex domain $X\subset \R^4$ admits disk-like global surface of section, i.e.\ an embedded disk $\Sigma \subset \partial X$ whose boundary is a closed characteristic $\gamma$, whose interior is transverse to the Reeb flow, and which has the property that the Reeb orbit $\phi^t(p)$ of any point $p$ in the complement of $\gamma$ meets $\Sigma$ for $t>0$ and $t<0$. In the same paper, they asked whether there exists a disk-like global surface of section whose boundary is a systole. Our main result gives a positive answer to this question.

\begin{mainthm}
\label{main}
Let $X\subset \R^4$ be a uniformly convex domain. Then every systole of $\partial X$ bounds a disk-like global surface of section.
\end{mainthm}

\paragraph{Consequences concerning symplectic capacities.} Let $\mathrm{Symp}(\R^{2n},\omega_0)$ denote the group of symplectomorphisms of $(\R^{2n},\omega_0)$. A \textit{normalized symplectic capacity} on $(\R^{2n},\omega_0)$ is a function $c$ which associates an element of $[0,+\infty]$ to every subset of $\R^{2n}$ and satisfies the following properties:

\medskip

\noindent (\textit{monotonicity}) If there is a $\phi\in \mathrm{Symp}(\R^{2n},\omega_0)$ such that $\phi(X)\subset Y$, then $c(X)\leq c(Y)$.

\medskip

\noindent (\textit{conformality}) $c(rX)=r^2 c(X)$ for every $r>0$.

\medskip

\noindent (\textit{normalization}) Identify $\R^{2n} \cong \C^n$ and define the ball and cylinder of symplectic width $a>0$ by
\begin{equation*}
B(a) \coloneqq \left\{ z\in \C^n \mid \pi|z|^2 \leq a \right\} \quad \text{and} \quad Z(a) \coloneqq \left\{ (z_1,\dots,z_n)\in \C^n \mid \pi|z_1|^2\leq a \right\}.
\end{equation*}
Then we have $c(B(a)) = c(Z(a)) = a$ for all $a>0$.

\medskip

Actually, since we are requiring the monotonicity property only for globally defined symplectomorphisms, the above axioms define a so-called \textit{extrinsic capacity}, see the discussion in \cite[Chapter 12]{ms17}. Being equivalent to Gromov's non-squeezing theorem \cite{gro85},  the existence of a normalized symplectic capacity is a nontrivial fact. Two normalized capacities which are easy to define but hard to compute are the following ball capacity $c_B$ (or Gromov width) and cylindrical capacity $c_Z$:
\[
\begin{split}
c_B(X) &:= \sup \{ a \mid \exists \phi\in \mathrm{Symp}(\R^{2n},\omega_0) \mbox{ such that } \phi(B(a)) \subset X \}, \\
c_Z(X) &:= \inf \{ a \mid \exists \phi\in \mathrm{Symp}(\R^{2n},\omega_0) \mbox{ such that } \phi(X) \subset Z(a) \}.
\end{split}
\]
It is easy to show that the above functions are the extremal capacities: any normalized symplectic capacity $c$ satisfies
\begin{equation}
\label{cBccZ}
c_B \leq c \leq c_Z.
\end{equation}
Given a bounded convex domain $X\subset \R^{2n}$ with smooth boundary, we let $\mathcal{A}_{\min}(X)$ denote the minimal action among all closed characteristics on $\partial X$. This function turns out to be continuous on the space of bounded convex domains endowed with the Hausdorff metric and continuously extends to the space of \textit{convex bodies} in $\R^{2n}$, i.e.\ compact convex sets with non-empty interior. Actually, there is a good notion of generalized characteristic on the boundary of a convex body $X$, without any smoothness assumption, and $\mathcal{A}_{\min}(X)$ is the minimal action among all closed generalized characteristics on $\partial X$, see \cite{cla81}. 

For many normalized symplectic capacities $c$, it is known that
\begin{equation}
\label{c=sys}
c(X) = \mathcal{A}_{\min}(X) \qquad \mbox{for all convex bodies $X$}.
\end{equation}
Indeed, this is the case for the Ekeland--Hofer capacity \cite{vit89}, for the Hofer--Zehnder capacity \cite{hz90}, for the capacity from symplectic homology \cite{ak22,iri22}, which by Hermann's work \cite{her04} coincides also with Viterbo's generating functions capacity from \cite{vit92}, and for Hutchings' first ``alternative'' capacity from \cite{hut22a}, see \cite{hhr24}. Note that $c_B$ and $c_Z$ do not belong to this list.

A long-standing question in symplectic topology was whether all normalized symplectic capacities agree on convex bodies, see e.g.\ \cite[Conjecture 1.9]{her98} and \cite[Section 14.9, Problem 53]{ms17}. This question has recently received a negative answer: in \cite{ho24}, Haim-Kislev and Ostrover constructed convex bodies in $\R^{2n}$, $n\geq 2$, for which $c_B < \mathcal{A}_{\min}$. Therefore, the ball capacity of a convex body can be strictly smaller than any of the normalized symplectic capacities listed above.

In \cite[Theorem 1.3]{edt24}, the second author proved that if $X\subset \R^4$ is a uniformly starshaped domain in $\R^4$ and $\gamma$ is a closed characteristic on $\partial X$ which bounds a $\partial$-strong\footnote{We refer to \cite[Definition 2.3]{edt24} for the notion of a $\partial$-strong disk-like global surface of section. We point out that by work of Florio-Hryniewicz \cite{fh21}, every closed characteristic $\gamma$ on the boundary of a uniformly convex domain which bounds a disk-like global surface of section also bounds a $\partial$-strong disk-like global surface of section.} disk-like global surface of section, then there exists a symplectomorphism $\phi$ of $(\R^4,\omega_0)$ such that $\phi(X)$ is contained in the cylinder $Z(\mathcal{A}(\gamma))$ of symplectic width $\mathcal{A}(\gamma)$ and $\phi(\partial X)$ touches the boundary of $Z(\mathcal{A}(\gamma))$ precisely at $\phi(\gamma)$. Thanks to this result, Theorem \ref{main} has the following immediate consequence.

\begin{maincor}
If $X$ is a uniformly convex domain in $\R^4$ and $\gamma$ is a systole of $\partial X$, then there exists a symplectomorphism $\phi$ of $(\R^4,\omega_0)$ such that $\phi(X)$ is contained in the cylinder $Z(\mathcal{A}_{\min}(X))$ and $\phi(\partial X)$ touches the boundary of this cylinder precisely at $\phi(\gamma)$. Therefore, $c_Z(X)=\mathcal{A}_{\min}(X)$ for every convex body $X\subset \R^4$.
\end{maincor}

Together with \eqref{cBccZ}, we deduce that, for every normalized symplectic capacity $c$ on $(\R^4,\omega_0)$ and every convex body $X\subset \R^4$, we have $c(X)\leq \mathcal{A}_{\min}(X)$. In dimension 4, we can therefore add the following to the list of normalized symplectic capacities for which \eqref{c=sys} holds:
\begin{enumerate}[(i)]
\item Hofer's displacement energy from \cite{hof90};
\item Hutchings's first ECH capacity from \cite{hut11};
\item the Hopf capacity from \cite{hhr24}.
\end{enumerate}
Indeed, if $c$ is one of the above three capacities and $X$ is a  convex body, then the inequality $c(X)\geq \mathcal{A}_{\min}(X)$ holds true. For (i), this follows from the fact that the displacement energy is not smaller than the Ekeland-Hofer capacity, see \cite[Theorem 1.6]{hof90a}, which as recalled above satisfies \eqref{c=sys}. For (ii), this follows from the fact that on uniformly starshaped domains, the first ECH capacity is a linear combination with positive integer coefficients of the actions of some closed characteristics. For (iii), this follows from the fact that, on uniformly convex domains, the Hopf capacity equals the action of some closed characteristic.

\paragraph{Discussion of previous results.}

The boundary of a uniformly starshaped domain $X\subset \R^4$ is diffeomorphic to the three-sphere and we can regard any closed characteristic $\gamma$ as a knot. It is actually a transverse knot, meaning that it is transverse to the contact structure $\ker \lambda$. As such, a closed characteristic $\gamma$ has a self linking number $\mathrm{sl}(\gamma)\in \Z$, see \cite[Section 3.5.2]{gei08}. Adopting the terminology of \cite{hhr24}, we call a closed characteristic which is unknotted and has self linking number $-1$ a \textit{Hopf orbit}. This term is motivated by the fact that a transverse knot in the standard contact three-sphere has the transverse knot type of a fiber of the Hopf fibration $S^3\rightarrow \C\mathrm{P}^1$ if and only if it is unknotted and has self linking number $-1$ \cite{eli93}. It is not difficult to show that the boundary of a disk-like global surface of section must be a Hopf orbit, see \cite[Proposition 2.1]{hry12}.

Now assume that $X\subset \R^4$ is uniformly convex. As proved by Hryniewicz \cite[Theorem 1.7]{hry14}, every Hopf orbit on $\partial X$ bounds a disk-like global surface of section. See \cite{hs11} and \cite{hls15} for other assumptions guaranteeing that a Hopf orbit bounds a disk-like global surface of section. In view of Hryniewicz's result, our main theorem is equivalent to the assertion that every systole of $\partial X$ is a Hopf orbit. In his PhD thesis \cite{hai07}, Hainz showed that if the principal curvatures $0<k_1\leq k_2 \leq k_3$ of $\partial X$ satisfy the pointwise condition $k_3\leq k_1+k_2$, then every closed characteristic of Conley-Zehnder index 3 is a Hopf orbit (see also \cite{hh11} for related results). It follows from Ekeland's work \cite{eke90} that every systole on a uniformly convex domain has Conley-Zehnder index 3. Therefore, combining \cite{hry14} and \cite{hai07} gives a positive answer to Hofer, Wysocki and Zehnder's question under the above curvature condition.

\paragraph{The uniform convexity assumption.}

Theorem \ref{main} fails if one replaces the assumption of uniform convexity by the slightly weaker assumption of convexity. Indeed, let $0<a<b$ be positive real numbers and consider the polydisk
\begin{equation*}
P(a,b) \coloneqq \left\{ (z_1,z_2)\in \C^2 \mid \pi|z_1|^2 \leq a , \pi|z_2|^2 \leq b \right\}.
\end{equation*}
Note that the systoles of $\partial P(a,b)$ foliate the solid torus $\left\{ \pi|z_1|^2 = a, \pi|z_2|^2 \leq b \right\}$ and are mutually unlinked. After an appropriate smoothing of $P(a,b)$ near the singular set 
\[
\left\{ (z_1,z_2)\in \C^2 \mid \pi|z_1|^2 = a, \pi|z_2|^2 = b \right\},
\]
 we obtain a smooth convex domain $X$ which continues to have the property that its systoles foliate a solid torus and are mutually unlinked. Since the boundary of a disk-like global surface of section must link every other periodic orbit positively, we see that no systole of $\partial X$ is the boundary of a disk-like global surface of section. While Theorem \ref{main} fails for convex domains, we still obtain the following weaker assertion, which we prove in Section \ref{sec:proof_main_theorem}.

\begin{maincor}
\label{cor:hopf_systole_on_convex_domain}
Let $X\subset \R^4$ be a smooth convex domain. Then there exists a systole of $\partial X$ which is a Hopf orbit.
\end{maincor}

For uniformly starshaped domains, even this weaker statement fails. In fact, it is possible to realize any transverse knot type as the unique closed characteristic of minimal action on the boundary of some uniformly starshaped domain. Let us also mention the paper \cite{vko20}, which constructs uniformly starshaped domains in $\R^4$ which do not admit any disk-like global surfaces of section at all.

Hofer, Wysocki and Zehnder's result \cite[Theorem 1.3]{hwz98} asserting the existence of a disk-like global surface of section on the boundary of a uniformly convex domain $X$ actually only uses the following consequence of the uniform convexity of $X$, which they named \textit{dynamical convexity}: all the closed characteristics on $\partial X$ have Conley-Zehnder index at least 3. As proven in \cite{ce22}, dynamical convexity is a strictly weaker condition than uniform convexity, see also \cite{ce,dgz24,dgrz23} for related results. Unlike \cite{hwz98}, our proof of Theorem \ref{main} uses uniform convexity in an essential way, and we do not know whether the conclusion of Theorem \ref{main} continues to hold in the dynamically convex case.

\paragraph{Sketch of the proof of Theorem \ref{main}.} The proof of Theorem \ref{main} combines three different theories from different decades: Clarke's duality \cite{cla79}, Hamiltonian Floer homology \cite{flo88a,flo88d}, and the theory of pseudoholomorphic curves in symplectizations, as initiated by Hofer in \cite{hof93b} and further developed together with Wysocki and Zehnder in a series of papers starting with \cite{hwz95,hwz96}. In this last subsection of the introduction, we assume familiarity with these theories, but in the main body of the paper we shall give the necessary prerequisites and references to the literature.

Most of the argument works for a uniformly convex domain $X$ in $\R^{2n}$ with $n$ arbitrary, and only in the last step we require that $n=2$. Assume that there is just one closed characteristic $\gamma$ of minimal action on $\partial X$ and that $\gamma$ is nondegenerate. By the uniform convexity of $X$, we can construct a uniformly convex autonomous Hamiltonian $\tilde{K}$ on $\R^{2n}$ such that the corresponding Hamiltonian flow has just the following 1-periodic orbits: the constant one at the minimizer $z_0$ of $\tilde{K}$ and an $S^1$-family of non-constant orbits, all of which are reparametrizations of $\gamma$. The 1-periodic orbits in this $S^1$-family are clearly degenerate, but by applying a small time-periodic perturbation near $\gamma$ we can replace the $S^1$-family by two nondegenerate 1-periodic orbits $\gamma_-$ and $\gamma_+$. The Hamiltonian system which is associated to this time-dependent Hamiltonian $K$ has exactly three 1-periodic orbits: $z_0$, $\gamma_-$ and $\gamma_+$, of Conley-Zehnder index $n$, $n+1$ and $n+2$, respectively.

Since $K$ is still uniformly convex, these three 1-periodic orbits are, up to suitable translations which we here neglect, the critical points of the Clarke dual action functional $\Psi_{K^*}$ associated to the Fenchel conjugate of $K$. They have Morse index $0$, $1$ and $2$, respectively. Due to the low regularity of $\Psi_{K^*}$, the Morse theory for the Clarke dual action functional requires considering a suitable finite dimensional reduction, but we shall ignore this issue here. The functional $\Psi_{K^*}$ has a local minimizer at $z_0$, is unbounded from below and $\gamma_-$ behaves as a mountain pass critical point: one branch of the one-dimensional unstable manifold of $\gamma_-$ converges to $z_0$, while on the other one $\Psi_{K^*}$ tends to $-\infty$. 

In \cite{ak22}, the first and third author constructed an isomorphism between the Morse complex of the Clarke dual action functional and the Hamiltonian Floer complex. This isomorphism preserves the action filtrations and shifts the grading by $n$. In that paper, $\Z/2$ coefficients were considered. Here, it is important to have an isomorphism over $\Z$ and in Appendices \ref{appA} and \ref{appB} of this paper we show how to deal with orientations and extend the main theorem of \cite{ak22} to integer coefficients. When we apply this isomorphism to the Hamiltonian $K$, we obtain that the integer valued algebraic count of Floer cylinders from $\gamma_-$ to $z_0$ is 1.

We now identify $\R^{2n}$ with the completion $\hat{X}$ of the Liouville domain $(X,\lambda_0)$. The next step is to prove that, for a generic almost complex structure $J$ on $\hat{X}$ which on the positive cylindrical end $\hat{X}\setminus X$ is symplectization admissible, the algebraic count of the finite energy $J$-holomorphic planes in $\hat{X}$ which are positively asymptotic to $\gamma$ and pass through a given generic point in $\hat{X}$ is also equal to 1. This is deduced from the above count of Floer cylinders by a neck stretching argument which is similar to the arguments used by Bourgeois and Oancea in \cite{bo09a} in order to relate symplectic homology to contact homology.

We finally focus on the case $n=2$ and consider the space $\mathcal{M}_1=\mathcal{M}_1(\hat{X},\gamma,J)$ of all unparametrized finite energy $J$-holomorphic planes in $\hat{X}$ with one marked point which are positively asymptotic to $\gamma$ and the associated evaluation map $\mathrm{ev}: \mathcal{M}_1 \rightarrow \R^4$ at the marked point, which turns out to be a proper map. For $J$ generic, the space $\mathcal{M}_1$ is a smooth 4-dimensional manifold, has a canonical orientation, and by the above result the degree of the map $\mathrm{ev}$ is 1. We shall show that the latter fact implies that $\gamma$ is a Hopf orbit. Indeed, using ideas from \cite{hss22,hhr24} we shall prove that in general, for a dynamically convex nondegenerate domain $X\subset \R^4$ and a periodic Reeb orbit $\gamma$ on $\partial X$ with Conley-Zehnder index 3, the degree of the evaluation map $\mathrm{ev}: \mathcal{M}_1(\hat{X},\gamma,J) \rightarrow \R^4$ is non-negative, strictly positive if and only if $\gamma$ bounds a positively immersed symplectic disk in $X$, and that the following statements are equivalent:
\begin{enumerate}[(i)]
\item $\deg (\mathrm{ev} ) =1$;
\item $\gamma$ bounds an embedded symplectic disk in $X$;
\item $\gamma$ is a Hopf orbit;
\item $\gamma$ bounds a disk-like global surface of section in $\partial X$.
\end{enumerate}
See Theorem \ref{thm:degree_characterizes_hopf} below for a more general statement, in which we allow $\gamma$ to have an arbitrary Conley-Zehnder index by working with fast pseudoholomorphic planes, as introduced by Hryniewicz in \cite{hry12}. Note that some implications in the above equivalence of statements are not new. For example, as already mentioned earlier, the equivalence of (iii) and (iv) was proved by Hryniewicz in \cite{hry14}. The main novelty of our work is that statement (i) implies the other three.

This proves Theorem \ref{main} under the additional assumption that $\partial X$ has just one closed characteristic $\gamma$ of minimal action and that $\gamma$ is nondegenerate. The general case now follows from a perturbation argument.

\paragraph{Structure of the paper.}

In Sections \ref{CRsec} and \ref{phcce}, we review preliminary material concerning linear Cauchy-Riemann type operators on punctured Riemann surfaces and punctured pseudoholomorphic curves in symplectizations and symplectic cobordisms. In Section \ref{sec:fast_planes}, we establish some results concerning the compactness and regularity of moduli spaces of fast pseudoholomorphic planes. Moreover, we introduce the degree of the evaluation map on the moduli space of fast planes with marked point. In Section \ref{sec:detecting_hopf}, we prove Theorem \ref{thm:degree_characterizes_hopf}, which says that this degree being equal to $1$ characterizes Hopf orbits. In Section \ref{floclasec}, we review the Floer complex of the Hamiltonian action functional and the Morse complex of the Clarke dual action functional. Moreover, we state Theorem \ref{isom}, which says that there exists a chain complex isomorphism between these two complexes over the integers. In Section \ref{conv-count-sec}, we use this isomorphism to prove that the count of Floer cylinders from $\gamma_-$ to $z_0$ mentioned above is equal to $1$. This is the content of Theorem \ref{convex-count}. In Section \ref{ns-stat}, we use a series of neck stretching arguments to show that the count of pseudoholomorphic planes asymptotic to the systole of a uniformly convex domain and passing through a prescribed point is equal to $1$. This is Theorem \ref{count-planes}. In Section \ref{sec:proof_main_theorem}, we prove Theorem \ref{main}, the main result of our paper, and also Corollary \ref{cor:hopf_systole_on_convex_domain}. Appendices \ref{appA} and \ref{appB} are devoted to the proof of Theorem \ref{isom}.

\paragraph{Acknowledgements.} A.A.~is partially supported by the DFG under the Collaborative Research Center SFB/TRR 191 - 281071066 (Symplectic Structures in Geometry, Algebra and Dynamics). O.E.~is supported by Dr.\ Max R\"{o}ssler, the Walter Haefner Foundation, and the ETH Z\"{u}rich Foundation. J.K.~is partially supported by Samsung Science and Technology Foundation SSTF-BA1801-01 and National Research Foundation of Korea grant NRF-2020R1A5A1016126 and RS-2023-00211186.

\tableofcontents

\numberwithin{equation}{section}

\section{Cauchy-Riemann operators on punctured Riemann surfaces}
\label{CRsec}

In this section, we review some material concerning linear Cauchy-Riemann type operators on vector bundles over punctured Riemann surfaces. Some parts of this review closely follow the exposition in \cite{wen10, wen16}.

\subsection{Asymptotic operators}

Let $\T\coloneqq \R/\Z$ denote the circle. An {\it asymptotic operator} is an unbounded self-adjoint operator of the form
\begin{equation}
\label{eq:asymptotic_operator_trivialized}
A: H^1(\T,\R^{2n})\subset L^2(\T,\R^{2n}) \rightarrow L^2(\T,\R^{2n}), \qquad Af = -J_0 f' - Sf.
\end{equation}
Here $J_0$ denotes the standard complex structure on $\R^{2n} \cong \C^n$ and $S: \T \rightarrow \mathrm{Sym}(2n)$ is a continuous loop of real symmetric $2n\times 2n$ matrices. 
The spectrum $\sigma(A)\subset \R$ of an asymptotic operator is discrete and accumulates at $\pm \infty$. Each eigenvalue has multiplicity at most $2n$ because eigenfunctions satisfy a linear ODE. An asymptotic operator $A$ is called {\it nondegenerate} if $0$ is not in its spectrum. Associated to an asymptotic operator $A = -J_0\partial_t-S$, we have an arc $\Phi:[0,1]\rightarrow\operatorname{Sp}(2n)$ in the group of symplectic matrices characterized by the ODE
\begin{equation*}
\Phi(0) = \operatorname{id}, \qquad \Phi'(t)=J_0S(t)\Phi(t).
\end{equation*}
The operator $A$ is nondegenerate if and only if $1$ is not an eigenvalue of $\Phi(1)$.

Let $\operatorname{Sp}_*(2n)\subset \operatorname{Sp}(2n)$ be the open and dense subset consiting of all symplectic matrices which do not have $1$ as an eigenvalue. Moreover, let $\widetilde{\operatorname{Sp}}_*(2n)$ denote the preimage of $\operatorname{Sp}_*(2n)$ under the universal covering map $\widetilde{\operatorname{Sp}}(2n)\rightarrow \operatorname{Sp}(2n)$. The {\it Conley-Zehnder index} \cite{cz84} is a locally constant map
\begin{equation*}
\operatorname{CZ} : \widetilde{\operatorname{Sp}}_*(2n) \rightarrow \Z
\end{equation*}
which descends to a bijection between the path components of $\widetilde{\operatorname{Sp}}_*(2n)$ and the integers $\Z$. If $A$ is a nondegenerate asymptotic operator with associated arc of symplectic matrices $\Phi$, we define the Conley-Zehnder index of $A$ to be
\begin{equation*}
\operatorname{CZ}(A) \coloneqq \operatorname{CZ}([\Phi]).
\end{equation*}
There exists a unique maximal lower semicontinuous extension of the Conley-Zehnder index to all of $\widetilde{\operatorname{Sp}}(2n)$. In this paper, $\operatorname{CZ}([\Phi])$ will refer to the value of this lower semicontinuous function whenever $[\Phi]\notin \widetilde{\operatorname{Sp}}_*(2n)$. This also yields an extension of the Conley-Zehnder index to degenerate asymptotic operators. For $\delta \in \R$, it will be useful to introduce the {\it constrained Conley-Zehnder index}
\begin{equation*}
\operatorname{CZ}^\delta(A) \coloneqq \operatorname{CZ}(A-\delta).
\end{equation*}
Clearly, $\operatorname{CZ}^0$ recovers the usual Conley-Zehnder index.

In the case $n=1$, the Conley-Zehnder index admits a description in terms of winding numbers of eigenfunctions \cite{hwz95}. There is a function 
\begin{equation*}
\operatorname{wind}:\sigma(A)\rightarrow \Z
\end{equation*}
called the {\it winding number} defined as follows. Let $\nu\in \sigma(A)$ be an eigenvalue with eigenfunction $f\in C^{\infty}(\T,\R^{2})$. Since $f$ solves a first order linear ODE, it must be nowhere vanishing and we define $\operatorname{wind}(\nu)$ to be the winding number of $f:\T\rightarrow \R^{2}\setminus \left\{ 0 \right\}$. This does not depend on the choice of the eigenfunction $f$ for $\nu$. Moreover, the function $\operatorname{wind}$ is non-decreasing and attains each integer value twice if eigenvalues are counted with multiplicity. For $\delta \in \R$, we define the quantities
\begin{equation}
\label{alpha}
\begin{split}
\alpha^{<\delta}(A) &\coloneqq \max \left\{ \operatorname{wind}(\nu) \mid \nu\in \sigma(A), \nu<\delta \right\}, \\
\alpha^{\geq \delta}(A) & \coloneqq \min \left\{ \operatorname{wind}(\nu) \mid \nu\in \sigma(A), \nu\geq \delta \right\}, \\
p^\delta(A) &\coloneqq \alpha^{\geq \delta}(A) - \alpha^{<\delta}(A) \in \left\{ 0,1 \right\}.
\end{split}
\end{equation}
For $\delta\in \R$, the constrained Conley-Zehnder index is given by
\begin{equation}
\label{eq:conley_zehnder_via_winding}
\operatorname{CZ}^\delta(A) = \alpha^{<\delta}(A) +  \alpha^{\geq \delta}(A)  = 2\alpha^{<\delta}(A) + p^\delta(A).
\end{equation}

Let $E\rightarrow \T$ be a Hermitian vector bundle. An operator
\begin{equation*}
A:H^1(E)\subset L^2(E)\rightarrow L^2(E)
\end{equation*}
is called an asymptotic operator if it is of the form \eqref{eq:asymptotic_operator_trivialized} in some (and hence any) unitary trivialization of $E$. Equivalently, asymptotic operators on $E$ are precisely the operators of the form $-J\nabla_t$ where $\nabla$ is a symplectic connection on $E$. For every choice of unitary trivialization $\tau$ of $E$ and every $\delta\in \R$, we have a Conley-Zehnder index $\operatorname{CZ}_\tau(A)$ and a constrained Conley-Zehnder index $\operatorname{CZ}_\tau^\delta(A)$. Clearly, this only depends on the homotopy class of $\tau$. When $E$ has complex rank 1, the Conley-Zehnder index $\operatorname{CZ}_\tau^\delta(A)$ is determined by the winding number $\operatorname{wind}_{\tau}: \sigma(A) \rightarrow \Z$ through the integers $\alpha_{\tau}^{<\delta}(A)$ and $\alpha_{\tau}^{\geq \delta}(A)$ as in \eqref{eq:conley_zehnder_via_winding}. These functions are obtained from the ones defined above using the trivialization $\tau$, but they actually depend only on the homotopy class of $\tau$. The number $p^{\delta}(A):= \alpha_{\tau}^{\geq \delta}(A) - \alpha_{\tau}^{<\delta}(A)$ is instead readily seen to be independent of $\tau$.

\subsection{Real linear Cauchy-Riemann type operators}

Let $(\Sigma,j)$ be a closed Riemann surface and let $\Gamma \subset \Sigma$ be a finite set of punctures. Assume that $\Gamma=\Gamma^+ \cup \Gamma^-$ is partitioned into positive and negative punctures. Set $\dot{\Sigma}\coloneqq \Sigma\setminus \Gamma$ and let $E\rightarrow \dot{\Sigma}$ be a smooth complex vector bundle of complex rank $n$. A {\it real-linear Cauchy-Riemann type operator} on $E$ is a first-order differential operator
\begin{equation*}
D:\Gamma(E)\rightarrow \Gamma(\Lambda^{0,1}T^*\dot{\Sigma}\otimes E)
\end{equation*}
satisfying
\begin{equation*}
D(fX) = fDX + (\overline{\partial}f) X \qquad \text{for all $f\in C^\infty(\dot{\Sigma},\R)$ and $X\in \Gamma(E)$,}
\end{equation*}
where $\Gamma(E)$ denotes the space of smooth sections of $E$. For each puncture $z\in \Gamma^\pm$, let us choose a biholomorphic identification of a punctured neighbourhood $\dot{\MU}_z$ of $z$ with the half cylinder $Z^{\pm}$ where $Z^+ = [0,\infty)\times \T$ and $Z^- = (-\infty,0]\times \T$. Moreover, let us fix a Hermitian vector bundle $E_z\rightarrow \T$ and a complex vector bundle isomorphism $E|_{\dot{\MU}_z} \cong \operatorname{pr}^*E_z$ where $\operatorname{pr}:Z^{\pm} \rightarrow \T$ is the natural projection. We call such a choice of identifications an {\it asymptotically Hermitian structure} on $E$. With an asymptotically Hermitian structure fixed, any unitary trivialization $\tau: E_z\cong \R^{2n}$ induces a trivialization $\tau:E|_{\dot{\MU}_z}\cong Z^{\pm}\times \R^{2n}$.

Let $D$ be a Cauchy-Riemann type operator on $E$, let $z\in \Gamma^\pm$ be a puncture, and let $A_z$ be an asymptotic operator on $E_z$. Using a trivialization $\tau$ of $E_z$ and indicating the coordinates on $Z^\pm$ by $(s,t)$, we can write
\begin{equation*}
D = \partial_s + J_0\partial_t + S(s,t) \qquad \text{and} \qquad A_z = -J_0\partial_t - S(t).
\end{equation*}
Here $S(s,t)$ takes values in the space of $2n\times 2n$ real matrices and $S(t)$ takes values in the subspace of symmetric matrices. We say that $D$ is {\it asymptotic} to $A_z$ at the puncture $z$ if $S(s,t)$ converges to $S(t)$ uniformly in $t$ as $s$ tends to $\pm \infty$.

Now suppose that $D$ is asymptotic to an asymptotic operator $A_z$ at each puncture $z\in \Gamma$. If each $A_z$ is nondegenerate, then for every $p\in (1,+\infty)$ the operator
\begin{equation*}
D : W^{1,p}(E) \rightarrow L^p(\Lambda^{0,1}T^*\dot{\Sigma}\otimes E)
\end{equation*}
is Fredholm with Fredholm index given by the formula
\begin{equation}
\label{eq:index_formula_CR_operator}
\operatorname{ind}(D) = n \chi(\dot{\Sigma}) + 2 c_1^\tau(E) + \sum\limits_{z\in \Gamma^+} \operatorname{CZ}_\tau(A_z) - \sum\limits_{z\in \Gamma^-} \operatorname{CZ}_\tau(A_z).
\end{equation}
In this formula, $\chi(\dot\Sigma)$ is the Euler characteristic of $\dot{\Sigma}$ and $c_1^\tau(E)$ denotes the {\it relative first Chern number} of $E$ with respect to the trivialization $\tau$. If $E$ is a line bundle, then $c_1^\tau(E)$ is the algebraic count of zeros of a generic section of $E$ which, near the punctures $\Gamma$, is non-vanishing and constant with respect to the trivialization $\tau$. For vector bundles $E$ of higher rank, $c_1^\tau(E)$ can be defined using the direct sum property.

\subsection{Orientation lines}

Given a real vector space $V$ of finite dimension $n\geq 0$, let us define the {\it orientation line} of $V$ to be the free $\Z$-module of rank $1$
\begin{equation*}
\mathfrak{o}_V \coloneqq H_n(V,V\setminus\left\{ 0 \right\};\Z).
\end{equation*}
A choice of generator of $\mathfrak{o}_V$ corresponds to a choice of orientation of $V$. If $T:X\rightarrow Y$ is a Fredholm operator between the real Banach spaces $X$ and $Y$, we define the orientation line of $T$ to be
\begin{equation*}
\mathfrak{o}_T \coloneqq \mathfrak{o}_{\operatorname{ker} T} \otimes \mathfrak{o}_{\operatorname{coker}T }^*.
\end{equation*}
The orientation lines form a local system over the space $\Phi(X,Y)$ of Fredholm operators from $X$ to $Y$, endowed with the operator norm topology. In fact, the orientation line of the Fredholm operator $T$ is canonically isomorphic to the orientation line of the determinant line of $T$, and the latter lines are the fibers of a vector bundle over $\Phi(X,Y)$, see Appendix \ref{appA} below.

If the vector space $V$ is complex linear, then the complex orientation of $V$ induces a preferred isomorphism $\mathfrak{o}_V\cong \Z$. Similarly, if the Banach spaces $X$ and $Y$ have complex structures and the operator $T\in \Phi(X,Y)$ is complex linear, we obtain a preferred isomorphism $\mathfrak{o}_T\cong \Z$.

Consider a complex vector bundle $E\rightarrow \dot{\Sigma}$ over a punctured Riemann surface $\dot{\Sigma}=\Sigma\setminus\Gamma$. Fix an asymptotically Hermitian structure at the punctures $\Gamma$. Let $\MD$ denote the space of all real linear Cauchy-Riemann type operators on $E$ with nondegenerate asymptotic operators at the punctures. The restriction  to $\MD$ of the local system given by the orientation lines on the space of Fredholm operators from $W^{1,p}(E)$ to $L^p(\Lambda^{0,1}T^* \dot\Sigma\otimes E)$ defines a local system $\mathfrak{o}$ on $\MD$. For each puncture $z\in \Gamma$, let $\MA_z$ denote the space of all nondegenerate asymptotic operators $A_z$ on the asymptotic bundle $E_z$. Moreover, set $\MA_\Gamma \coloneqq \prod_{z\in\Gamma}\MA_z$. There is a natural map $\MD\rightarrow \MA_\Gamma$ assigning to a Cauchy-Riemann type operator its asymptotic operators. The fibers of this map are contractible. The local system $\mathfrak{o}$ therefore descends from $\MD$ to $\MA_\Gamma$.

For $j\in \left\{ 0,1 \right\}$, consider complex vector bundles $E_j\rightarrow \dot{\Sigma}_j$ of the same rank over punctured Riemann surfaces $\dot{\Sigma}_j = \Sigma_j\setminus \Gamma_j$. Fix asymptotically Hermitian structures. Let $\Gamma_j^{\pm}$ be partitions into positive and negative punctures. Set $\Gamma^+\coloneqq \Gamma_0^+$ and $\Gamma^-\coloneqq \Gamma_1^-$. Suppose that there is an identification $\Gamma_0^- \cong \Gamma_1^+\eqqcolon \Gamma^0$. For $z\in \Gamma^0$, fix a unitary bundle isomorphism $(E_0)_z \cong (E_1)_z$. Using these data, we can form a vector bundle $E$ over a glued surface $\dot{\Sigma}$ with punctures $\Gamma \coloneqq \Gamma^+ \cup \Gamma^-$. Let $D_j$ be Cauchy-Riemann type operators on $E_j$ with nondegenerate asymptotic operators. Assume that the asymptotic operators of $D_0$ at the punctures $\Gamma^0$ match those of $D_1$. Let $D$ be a Cauchy-Riemann type operator obtained by gluing $D_0$ and $D_1$. Then the Floer-Hofer kernel gluing operation \cite{fh93} induces an isomorphism
\begin{equation}
\label{eq:floer_hofer_kernel_gluing}
\mathfrak{o}_{D_0} \otimes \mathfrak{o}_{D_1} \cong \mathfrak{o}_D.
\end{equation}
This isomorphism is natural in the following sense. Let $\mathfrak{o}_j$ and $\mathfrak{o}$ denote the orientation local systems over $\MA_{\Gamma_j}$ and $\MA_\Gamma$, respectively. The product $\MA_{\Gamma^+}\times \MA_{\Gamma^0} \times \MA_{\Gamma^-}$ admits natural projections $\operatorname{pr}_j$ to $\MA_{\Gamma_j}$ and $\operatorname{pr}$ to $\MA_{\Gamma}$. Isomorphism \eqref{eq:floer_hofer_kernel_gluing} induces an isomorphism of local systems over $\MA_{\Gamma^+} \times \MA_{\Gamma^0} \times \MA_{\Gamma^-}$
\begin{equation}
\label{eq:floer_hofer_kernel_gluing_naturality}
\operatorname{pr}_0^*\mathfrak{o}_0 \otimes \operatorname{pr}_1^*\mathfrak{o}_1 \cong \operatorname{pr}^*\mathfrak{o}.
\end{equation}

The following lemma describes the orientation local system $\mathfrak{o}$ in more detail in the case of a complex line bundle.

\begin{lem}
\label{lem:orientation_local_system}
Let $E\rightarrow \dot{\Sigma}$ be a complex line bundle.
\begin{enumerate}[(i)]
\item Consider a puncture $z\in \Gamma$. Path components of $\MA_z$ are classified by the Conley-Zehnder index. Path components of odd Conley-Zehnder index are contractible and path components of even Conley-Zehnder index are homotopy equivalent to the circle $\T$.
\item Consider a loop in $\MA_\Gamma$ which is non-constant only in one single factor $\MA_z$, where it represents a generator of the fundamental group of a path component of even Conley-Zehnder index. Then the monodromy of $\mathfrak{o}$ around this loop is $-1$.
\end{enumerate}
\end{lem}

\begin{rem}
\label{rem:canonical_orientation_plane}
{\rm Let us describe the special case of Lemma \ref{lem:orientation_local_system} for $\dot{\Sigma} = \C$. For each integer $k$, let $\MD_k$ denote the space of Cauchy-Riemann type operators on $E\rightarrow \C$ which are positively asymptotic to some nondegenerate asymptotic operator of Conley-Zehnder index $k$. For even $k$, the space $\MD_k$ is homotopy equivalent to $\T$ and the restriction $\mathfrak{o}|_{\MD_k}$ is the unique (up to isomorphism) non-trivial local system of free Abelian groups of rank $1$. If $k$ is odd, then $\MD_k$ is contractible and $\mathfrak{o}|_{\MD_k}$ is necessarily trivial. In this case, the set of complex linear Cauchy-Riemann type operators in $\MD_k$ is non-empty and connected. Thus there exists a preferred trivialization $\mathfrak{o}|_{\MD_k}\cong \MD_k \times \Z$ which agrees with the preferred isomorphism $\mathfrak{o}_D\cong \Z$ whenever $D\in \MD_k$ is complex linear. The same is true in case of negative asymptotics.}
\end{rem}

\begin{proof}
The space of asymptotic operators $\MA_z$ is homotopy equivalent to the space of arcs
\begin{equation*}
\MS\MP_*(2) \coloneqq \left\{ \Phi:[0,1]\rightarrow \operatorname{Sp}(2) \mid \Phi(0)=\operatorname{id}, \enspace \Phi(1)\in \operatorname{Sp}_*(2) \right\}.
\end{equation*}
We claim that the natural projection map $p:\MS\MP_*(2) \rightarrow \widetilde{\operatorname{Sp}}_*(2)$ is a homotopy equivalence. Indeed, $p$ is a fibration whose fiber is homotopy equivalent to a connected component of the based loop space $\Omega\operatorname{Sp}(2)$. Since $\operatorname{Sp}(2)$ is homotopy equivalent to $\operatorname{U}(1)$, it is not hard to see that the fibers of $p$ are contractible, implying that $p$ is a homotopy equivalence. Let us abbreviate
\begin{align*}
\operatorname{Sp}_+(2) & \coloneqq \left\{ B \in \operatorname{Sp}(2) \mid \det(\operatorname{id} - B) > 0 \right\} \\
\operatorname{Sp}_-(2) & \coloneqq \left\{ B \in \operatorname{Sp}(2) \mid \det(\operatorname{id} - B) < 0 \right\}.
\end{align*}
Clearly, $\operatorname{Sp}_*(2)$ is the disjoint union of $\operatorname{Sp}_+(2)$ and $\operatorname{Sp}_-(2)$. Using that the inclusions $\operatorname{Sp}_\pm(2)\subset \operatorname{Sp}(2)$ induce trivial homomorphisms of fundamental groups, we see that the projection $\widetilde{\operatorname{Sp}}_*(2)\rightarrow \operatorname{Sp}_*(2)$ maps path components of $\widetilde{\operatorname{Sp}}_*(2)$ homeomorphically onto $\operatorname{Sp}_\pm(2)$. Path components of odd Conley-Zehnder index are mapped to $\operatorname{Sp}_+(2)$ and path components of even Conley-Zehnder index are mapped to $\operatorname{Sp}_-(2)$. The first assertion of the lemma follows from the fact that $\operatorname{Sp}_+(2)$ is contractible and that $\operatorname{Sp}_-(2)$ is homotopy equivalent to $\T$.

In view of the naturality of the Floer-Hofer kernel gluing operation \eqref{eq:floer_hofer_kernel_gluing_naturality}, it suffices to prove the second assertion of the lemma in the special case of a trivial line bundle $E$ over $\C$. Let $D$ denote an arbitrary Cauchy-Riemann type operator with nondegenerate asymptotic operator with even Conley-Zehnder index $k$. By \eqref{eq:index_formula_CR_operator}, $D$ has odd Fredholm index $1 \pm k$, the alternative for the sign depending on whether the puncture at $\infty$ is positive or negative. Consider the loop of Cauchy-Riemann type operators $(D_\lambda \coloneqq e^{i\pi \lambda}De^{-i\pi \lambda})_{\lambda\in \T}$. The image of $(D_\lambda)_{\lambda\in\T}$ in the space of nondegenerate asymptotic operators $\mathcal{A}_\infty$ represents a generator of the fundamental group of the path component it is contained in. The kernel and cokernel of $D_\lambda$ are given by
\begin{equation*}
\operatorname{ker}D_\lambda = e^{i\pi\lambda}\operatorname{ker}D\qquad \text{and} \qquad \operatorname{coker}D_\lambda = e^{i\pi\lambda}\operatorname{coker}D.
\end{equation*}
Since $e^{i\pi}=-1$, we see that the vector bundle $(\operatorname{ker}D_\lambda)_{\lambda\in\T}$ over $\T$ is trivial if and only if the dimension of $\operatorname{ker}D$ is even. The analogous statement holds for $(\operatorname{coker}D_\lambda)_{\lambda\in\T}$. Since the index of $D$ is odd, exactly one of the vector spaces $\operatorname{ker}D$ and $\operatorname{coker}D$ has odd dimension. This implies that the bundle of orientation lines $(\mathfrak{o}_{D_\lambda})_{\lambda\in \T}$ is non-trivial.
\end{proof}

\subsection{Exponential weights}

We also need to consider Sobolev spaces of sections of $E$ with {\it exponential weights} at the punctures. For each puncture $z\in \Gamma$, let us fix an exponential weight $\delta_z\in \R$ and let $\delta = (\delta_z)_{z\in \Gamma}$ denote the collection of all weights. We define $W^{1,p,\delta}(E)$ to be the space of all sections $X$ of $E$ of class $W^{1,p}_{\operatorname{loc}}$ such that, on the punctured neighbourhood $\dot{\MU}_z\cong Z^{\pm}$ of $z\in \Gamma^\pm$, the section $e^{\pm \delta_z s}X(s,t)$ is of class $W^{1,p}$. The space $L^{p,\delta}(\Lambda^{0,1}T^*\dot{\Sigma}\otimes E)$ is defined similarly. Now consider again a Cauchy-Riemann type operator $D$ on $E$ with asymptotic operators $A_z$. Assume that $\mp\delta_z\notin \sigma(A_z)$ at every puncture $z\in \Gamma^{\pm}$. Then the operator
\begin{equation*}
D : W^{1,p,\delta}(E) \rightarrow L^{p,\delta}(\Lambda^{0,1}T^*\dot{\Sigma}\otimes E)
\end{equation*}
is Fredholm with Fredholm index given by
\begin{equation}
\label{eq:index_formula_CR_operator_exponential_weight}
\operatorname{ind}^\delta(D) = n \chi(\dot{\Sigma}) + 2 c_1^\tau(E) + \sum\limits_{z\in \Gamma^+} \operatorname{CZ}_\tau^{-\delta_z}(A_z) - \sum\limits_{z\in \Gamma^-} \operatorname{CZ}_\tau^{\delta_z}(A_z).
\end{equation}
This can be seen as follows. Let $\Phi:E\rightarrow E$ be a bundle automorphism which, near every puncture $z\in \Gamma^\pm$, is given by multiplication by $e^{\pm\delta_z s}$. Then the composition $\Phi_* D \Phi_*^{-1}$ is a Cauchy-Riemann type operator with asymptotic operators $A_z \pm \delta_z$. Since $\Phi$ induces Banach space isomorphisms $\Phi_*:W^{1,p,\delta} \rightarrow W^{1,p}$ and $\Phi_*:L^{p,\delta}\rightarrow L^p$ between Sobolev and Lebesgue spaces with and without exponential weight, the index formula with exponential weights \eqref{eq:index_formula_CR_operator_exponential_weight} is a direct consequence of the index formula without exponential weights \eqref{eq:index_formula_CR_operator}.

\subsection{Automatic regularity}

Now consider the case that $E$ is a complex line bundle. Let $D$ be a Cauchy-Riemann type operator on $E$ and let $\delta$ be a choice of exponential weights such that $\mp \delta_z \notin \sigma(A_z)$ for every $z\in \Gamma^\pm$. Suppose that $X\in \operatorname{ker}D$ is not identically equal to zero. It follows from the analysis in \cite{hwz96} and \cite{sie08} that, for every puncture $z\in \Gamma^{\pm}$, there exists an eigenfunction $e$ of $A_z$ with eigenvalue $\nu$ satisfying $\pm \nu < -\delta_z$ such that, for all $\pm s \gg 0$,
\begin{equation}
\label{eq:asymptotic_formula_linear_CR_operator}
X(s,t) = e^{\nu s}(e(t) + r(s,t))
\end{equation}
where the error term $r(s,t)$ tends to zero as $s\rightarrow \pm \infty$ uniformly in $t$. This asymptotic formula implies that $X$ has a well-defined winding number $\operatorname{wind}^\tau(X,z)$ at every puncture $z\in \Gamma^\pm$. Since $\pm \nu < -\delta_z$ and the winding number is non-decreasing on the spectrum $\sigma(A_z)$, we have
\begin{equation}
\label{eq:winding_inequality_linear_CR_operator}
\operatorname{wind}^\tau(X,z) \leq \alpha_\tau^{<-\delta_z}(A_z) \quad \text{if $z\in \Gamma^+$} \qquad \text{and} \qquad \operatorname{wind}^\tau(X,z) \geq \alpha_\tau^{\geq \delta_z}(A_z)\quad \text{if $z\in \Gamma^-$.}
\end{equation}
Another consequence of the asymptotic formula \eqref{eq:asymptotic_formula_linear_CR_operator} is that the zeros of $X$ cannot accumulate at the punctures $\Gamma$. Since the zeros of $X$ are isolated in $\dot{\Sigma}$ and have positive multiplicities, this implies that the algebraic count of zeros of $X$ is non-negative and finite. This fact in combination with estimate \eqref{eq:winding_inequality_linear_CR_operator} can be used to deduce Proposition \ref{prop:automatic_regularity_linear_CR_operator} below. The first assertion of Proposition \ref{prop:automatic_regularity_linear_CR_operator} is a special case of the more general result \cite[Proposition 2.2]{wen10} and the second assertion follows easily from its proof. For convenience of the reader, we repeat the argument.

\begin{prop}
\label{prop:automatic_regularity_linear_CR_operator}
Let $E\rightarrow \C$ be a complex line bundle. Let $D$ be a Cauchy-Riemann type operator on $E$ with asymptotic operator $A$ at the positive puncture $\infty$. Let $\delta \in \R$ be an exponential weight such that $-\delta\notin \sigma(A)$. Then the following statements hold:
\begin{enumerate}[(i)]
\item \label{item:automatic_regularity_linear_CR_operator} The operator $D:W^{1,p,\delta}(E)\rightarrow L^{p,\delta}(\Lambda^{0,1}T^*\C \otimes E)$ is injective if and only if $\operatorname{ind}^\delta(D)\leq 0$ and surjective if and only if $\operatorname{ind}^\delta(D)\geq 0$.
\item \label{item:automatic_regularity_linear_CR_operator_index_2} Assume that $\operatorname{ind}^\delta(D) = 2$. Let $X_1,X_2 \in \operatorname{ker}(D)$ be a basis inducing the preferred complex orientation of the orientation line $\mathfrak{o}_D$ (see Remark \ref{rem:canonical_orientation_plane}). Then, for all $p\in \C$, the vectors $X_1(p), X_2(p)$ form a basis of the fiber $E_p$ inducing the complex orientation of $E_p$.
\end{enumerate}
\end{prop}

\begin{proof}
Let $\tau: E\cong \C$ be a global trivialization of $E$. Assume first that $\operatorname{ind}^\delta(D)\leq 0$. It follows from the index formula \eqref{eq:index_formula_CR_operator_exponential_weight} that $\operatorname{CZ}_\tau^{-\delta}(A) < 0$, which in turn implies that $\alpha_\tau^{<-\delta}(A) < 0$. Assume by contradiction that $\operatorname{ker}(D)$ contains a non-zero element. Using \eqref{eq:winding_inequality_linear_CR_operator}, we obtain $\operatorname{wind}^\tau(X,\infty) < 0$. This yields a contradiction because $\operatorname{wind}^\tau(X,\infty)$ is also equal to the algebraic count of zeros of $X$, which is always non-negative.

Next, assume that $\operatorname{ind}^\delta(D)\geq 0$. Let us first consider the case that $\operatorname{ind}^\delta(D) = 2m$ is even. By \eqref{eq:index_formula_CR_operator_exponential_weight}, $\operatorname{CZ}_{\tau}^{-\delta}(A)$ is odd and hence by \eqref{eq:conley_zehnder_via_winding} $p^{-\delta}(A)=1$. Pick distinct points $p_1,\dots,p_m\in \dot{\Sigma}$ and consider the evaluation map
\begin{equation*}
\operatorname{ev}:\operatorname{ker}(D) \rightarrow \bigoplus\limits_{k=1}^m E_{p_k} \qquad X \mapsto (X(p_k))_k.
\end{equation*}
Assume by contradiction that $X$ is a non-zero element of the kernel of $\operatorname{ev}$. This implies that the algebraic count of zeros of $X$ is at least $m$. Since this count of zeros agrees with $\operatorname{wind}^\tau(X,\infty)$, using \eqref{eq:winding_inequality_linear_CR_operator} we obtain $m\leq \alpha_\tau^{<-\delta}(A)$. We can therefore estimate 
\begin{equation*}
2m = \operatorname{ind}^\delta(D) = 1 + \operatorname{CZ}_\tau^{-\delta}(A) = 1 + 2 \alpha_\tau^{<-\delta}(A) + p^{-\delta}(A) \geq 2m + 2.
\end{equation*}
This is a contradiction, which means that the evaluation map $\operatorname{ev}$ is injective. Thus we have
\begin{equation*}
2m = \operatorname{ind}^\delta(D) \leq \operatorname{dim}\operatorname{ker}(D) \leq 2m,
\end{equation*}
which implies that $\operatorname{ev}$ is an isomorphism and that $D$ is surjective. In the case that $\operatorname{ind}^\delta(D)=2m+1$ is odd, we modify the above argument as follows. In this case, $\operatorname{CZ}_{\tau}^{-\delta}(A)$ is even and hence by \eqref{eq:conley_zehnder_via_winding}
\begin{equation}
\label{p=0}
\alpha^{\geq -\delta}(A) - \alpha^{<-\delta}(A)=0.
\end{equation} 
Let $\nu$ denote the largest eigenvalue of $A$ which is strictly less than $-\delta$. By \eqref{p=0} and the fact that the function $\rm{wind}$ attains each integral value twice, the corresponding eigenspace $E(A,\nu)$ has real dimension $1$. Any $X\in \operatorname{ker}(D)$ admits an asymptotic expansion for $s \gg 0$
\begin{equation*}
X(s,t) = e^{\nu s}(e(t) + r(s,t))\qquad \text{with $r(s,\cdot)\rightarrow 0$}
\end{equation*}
for a unique $e\in E(A,\nu)$. We define a modified evaluation map by
\begin{equation*}
\operatorname{ev} : \operatorname{ker}(D) \rightarrow E(A,\nu) \oplus \bigoplus\limits_{k=1}^m E_{p_k} \qquad X \mapsto (e,(X(p_k))_k).
\end{equation*}
Arguing similarly as above, one can verify that this modified evaluation map is an isomorphism, which again implies that $D$ is surjective. This concludes the proof of assertion \eqref{item:automatic_regularity_linear_CR_operator}. In the case $\operatorname{ind}^\delta(D)=2$, the evaluation map $\operatorname{ev}:\operatorname{ker}(D)\rightarrow E_p$ is an isomorphism for all points $p\in \C$. Clearly, $\operatorname{ev}$ is complex linear if $D$ is complex linear. This shows assertion \eqref{item:automatic_regularity_linear_CR_operator_index_2}.
\end{proof}

\section{Pseudoholomorphic curves with cylindrical ends}
\label{phcce}

Punctured pseudoholomorphic curves in symplectic manifolds with cylindrical ends were first introduced by Hofer in \cite{hof93b}. In the following, we review some fundamental properties of punctured pseudoholomorphic curves. This review is based on material from numerous foundational works, such as \cite{hwz95, hwz96, hwz99, hwz02,BEHWZ03,sie08,wen10,Wen10b}.

\subsection{Preliminaries}
\label{subsec:setup_pseudoholomorphic_curves}

Let $(Y,\xi)$ be a contact manifold and let $\lambda$ be a contact form defining the contact structure $\xi = \operatorname{ker}\lambda$. The contact form induces a Reeb vector field $R_\lambda$ characterized by
\begin{equation*}
\iota_{R_\lambda} d\lambda = 0 \quad \text{and} \quad \iota_{R_\lambda}\lambda = 1.
\end{equation*}
A periodic orbit of $R_\lambda$ of period $T>0$ is a loop $\gamma : \R/T\Z\rightarrow Y$ everywhere tangent to $R_\lambda$. The orbit is called simple if $\gamma$ is injective and multiply covered otherwise. We consider two periodic orbits to be equivalent if they are related by reparametrization. For every periodic orbit $\gamma$ of $R_\lambda$, let us fix a base point $b_\gamma \in \operatorname{im}(\gamma)$. Requiring that $\gamma(0)=b_\gamma$ induces a preferred parametrization of every periodic orbit. We will always use these preferred parametrizations throughout this paper. A periodic orbit is called nondegenerate if the linearized return map of the Reeb flow around the orbit does not have $1$ as an eigenvalue. If every periodic Reeb orbit is nondegenerate, then the contact form $\lambda$ is called nondegenerate. This is a $C^\infty$ generic condition.

An almost complex structure $J$ on the symplectization $\R\times Y$ is called {\it symplectization admissible} if it satisfies the following properties:
\begin{enumerate}
\item $J$ is invariant under translations along the $\R$-factor of $\R\times Y$.
\item The contact structure $\xi$ is invariant under $J$ and the restriction $J|_\xi$ is $d\lambda$-compatible.
\item Let $a$ denote the coordinate of the $\R$-factor of $\R\times Y$. Then $J\partial_a = R_\lambda$. 
\end{enumerate}

Consider a strong symplectic cobordism $(X,\omega)$ with convex end $(Y^+,\lambda^+)$ and concave end $(Y^-,\lambda^-)$. This means that $X$ is a compact manifold, has boundary $Y^+ \cup Y^-$ and that there exists a Liouville vector field $Z$ near the boundary which is outward pointing at $Y^+$, inward pointing at $Y^-$, and satisfies $\iota_{Z}\omega|_{Y^\pm}= \lambda^{\pm}$. Via the flow of $Z$, we identify a small collar neighbourhood of $Y^+$ with $(-\varepsilon,0]\times Y^+$. Using this identification, we attach a positive cylindrical end $X^+=[0,\infty)\times Y^+$ along the boundary component $Y^+$. Similarly, we identify a collar neighbourhood of $Y^-$ with $[0,\varepsilon)\times Y^-$ and attach a negative cylindrical end $X^-=(-\infty,0]\times Y^-$ along $Y^-$. The resulting manifold $\hat{X}$ is called the {\it symplectic completion} of $X$. An almost complex structure $J$ on $\hat{X}$ is called {\it admissible} if its restriction to $X$ is $\omega$-compatible and if its restrictions to the cylindrical ends $X^\pm$ agree with some admissible almost complex structures on the symplectizations of $(Y^\pm,\lambda^\pm)$.

Since a symplectization can be regarded as the symplectic completion of a trivial cobordism, we will phrase most statements in our review of puncutred pseudoholomorphic curves in the more general setting of symplectic completions and only restrict to symplectizations when necessary.

\subsection{Finite energy curves}
\label{subsec:finite_energy_curves}

Let $(X,\omega)$ be a strong symplectic cobordism with convex/concave ends $(Y^\pm,\lambda^\pm)$ and set $\xi^{\pm} := \ker \lambda^{\pm}$. Equip the symplectic completion $\hat{X}$ with an admissible almost complex structure $J$. Let $(\Sigma,j)$ be a closed Riemann surface. Let $\Gamma \subset \Sigma$ be a finite set of punctures and set $\dot{\Sigma} \coloneqq \Sigma\setminus \Gamma$. The {\it Hofer energy} of a pseudoholomorphic curve $u:(\dot{\Sigma},j)\rightarrow (\hat{X},J)$ is defined by
\begin{equation*}
E(u)\coloneqq \int_{u^{-1}(X)}u^*\omega \enspace +\enspace \sup_{\varphi^+} \int_{u^{-1}(X^{+})} u^* d(\varphi^+ \lambda^+) \enspace +\enspace \sup_{\varphi^-} \int_{u^{-1}(X^-)} u^* d(\varphi^- \lambda^-).
\end{equation*}
Here, the first supremum is taken over all smooth functions $\varphi^+:[0,\infty)\rightarrow [0,1]$ satisfying $(\varphi^+)'\geq 0$. Similarly, the second supremum is taken over all smooth functions $\varphi^-:(-\infty,0]\rightarrow [0,1]$ satisfying $(\varphi^-)'\geq 0$. The curve $u$ is said to have {\it finite energy} if $E(u) < \infty$.

Consider a puncture $z\in \Gamma$ of $u$. The puncture is called {\it removable} if $u$ maps some neighbourhood of $z$ into a compact subset of $\hat{X}$. In this case, the map $u$ smoothly extends over $z$. The puncture is called {\it positive}/{\it negative} if the image of some neighbourhood of $z$ is contained in $X^{\pm}$ and $a\circ u(w)$ tends to $\pm \infty$ as $w$ tends to $z$. Here $a$ denotes the first coordinate of the cylindrical ends $X^\pm$. It follows from the analysis in \cite{hof93b} and \cite{hwz96} that if $u$ is a finite energy curve, then every puncture is removable, positive or negative. After filling in all removable punctures, we obtain a partition $\Gamma = \Gamma^+ \cup \Gamma^-$ into positive and negative punctures.

\subsection{Asymptotics at the punctures}
\label{subsec:asymptotics_at_punctures}

The asymptotic behaviour of pseudoholomorphic curves at punctures was first described in detail by Hofer-Wysocki-Zehnder in \cite{hwz96}. This description was later refined by Siefring in \cite{sie08}.

Assume that the contact forms ${\lambda^{\pm}}$ are nondegenerate. Suppose that $u:(\dot{\Sigma},j)\rightarrow (\hat{X},J)$ is a finite energy curve and let $z$ be a puncture of sign $\varepsilon = \pm 1$. Equip a neighbourhood of $z$ with a cylindrical holomorphic coordinate $(s,t)\in Z^{\pm}$. Then there exist a possibly multiply covered periodic orbit $\gamma$ of $R_{\lambda^{\pm}}$ of some period $T>0$ and a not necessarily holomorphic reparametrization $\varphi$ of $Z^\pm$ such that, for all $\pm s\gg 0$, we have
\begin{equation*}
u\circ \varphi(s,t) = \operatorname{exp}_{u_\gamma(s,t)} f(s,t)
\end{equation*}
where $u_\gamma(s,t)\coloneqq (Ts,\gamma(Tt))\in X^\pm$ and $f\in \Gamma(u_\gamma^*\xi^\pm)$ is a section with exponential decay. Here exponential decay means that there exists $r>0$ such that
\begin{equation*}
\lim_{s\rightarrow \pm \infty} \sup_{t\in \T} e^{\pm r s} |\nabla^\alpha f(s,t)| = 0
\end{equation*}
for all multi-indices $\alpha$. The section $f$ is called the {\it asymptotic representative} of $u$.

The asymptotic representative admits a more detailed description in terms of eigenvalues and eigenfunctions of the {\it asymptotic operator} $A_\gamma$ of the periodic Reeb orbit $\gamma$. Let us abbreviate $\tilde{\gamma}(t)\coloneqq \gamma(Tt)$ for $t\in \T$. The asymptotic operator is defined to be the unbounded self-adjoint operator
\begin{equation*}
A_\gamma : H^1(\tilde{\gamma}^*\xi^\pm) \subset L^2(\tilde{\gamma}^*\xi^\pm) \rightarrow L^2(\tilde{\gamma}^*\xi^\pm) \qquad A_\gamma \coloneqq -J\nabla_t^{R_{\lambda^\pm}}
\end{equation*}
where the Hermitian structure on the vector bundle $\tilde{\gamma}^* \xi^{\pm}$ is given by $d\lambda^{\pm}(\cdot,J \cdot)$ and $\nabla^{R_{\lambda^\pm}}$ denotes the symplectic connection on $\tilde{\gamma}^*\xi^\pm$ induced by the linearized Reeb flow. By the non-degeneracy assumption on $\lambda^{\pm}$, $A_{\gamma}$ is nondegenerate. The following alternative holds for the asymptotic representative $f$: either $f$ vanishes identically for all $\pm s\gg 0$, or there exists an eigenfunction $e\in \Gamma(\tilde{\gamma}^*\xi^\pm)$ of $A_\gamma$ with eigenvalue $\nu$ satisfying $\pm \nu < 0$ such that, for all $\pm s \gg 0$,
\begin{equation*}
f(s,t) = e^{\nu s}(e(t) + r(s,t))
\end{equation*}
where the error term $r\in \Gamma(u_\gamma^*\xi^\pm)$ has exponential decay. In the latter case, we say that $u$ has {\it non-trival asymptotic behaviour} at the puncture $z$. We call $\nu$ and $e$ the {\it asymptotic eigenvalue} and \textit{eigenfunction} of $u$ at $z$. Moreover, we call $|\nu|$ the {\it transversal convergence rate}. In the former case, we say that $u$ has {\it trivial asymptotic behaviour} at $z$ and that the transversal convergence rate is $+\infty$.

There is also a description of the relative asymptotic behaviour of two cylindrical ends of pseudoholomorphic curves asymptotic to the same orbit due to Siefring \cite{sie08}. Let $u$ and $v$ be two punctured pseudoholomorphic curves with a puncture of the same sign $\varepsilon = \pm 1$ asymptotic to the same periodic orbit $\gamma$. Let $f_u$ and $f_v$ denote the asymptotic representatives of $u$ and $v$. Either, $f_u$ and $f_v$ agree identically for all $\pm s\gg 0$, or there exists an eigenfunction $e$ of $A_\gamma$ with eigenvalue $\nu$ satisfying $\pm \nu <0$ such that, for all $\pm s\gg 0$,
\begin{equation*}
f_u(s,t) - f_v(s,t) = e^{\nu s}(e(t) + r(s,t))
\end{equation*}
with an error term $r(s,t)$ of exponential decay. We call $e$ the {\it relative asymptotic eigenfunction} and $\nu$ the {\it relative asymptotic eigenvalue}. Moreover, we call $|\nu|$ the {\it relative transversal convergence rate}. Suppose now that $f_u$ and $f_v$ do not agree identically. Let $e_u$ and $e_v$ be the asymptotic eigenfunctions and $\nu_u$ and $\nu_v$ the asymptotic eigenvalues of $u$ and $v$, respectively. If $e_u \neq e_v$, then the relative transversal convergence rate $|\nu|$ is given by $\min\left\{ |\nu_u|,|\nu_v| \right\}$. If $e_u = e_v$, then the relative transversal convergence rate $|\nu|$ is strictly bigger than $|\nu_u| = |\nu_v|$.

\subsection{Algebraic invariants}

In this subsection, we review some algebraic invariants associated to punctured pseudoholomorphic curves in dimension $4$ introduced by Hofer-Wysocki-Zehnder in \cite{hwz95}. Let $X^4$ be a $4$-dimensional strong symplectic cobordism. Let $z$ be a positive/negative puncture of a finite energy curve $u:(\dot{\Sigma},j)\rightarrow (\hat{X},J)$ and let $\gamma$ be the asymptotic Reeb orbit of $u$ at $z$. Let $\tau$ be a symplectic trivialization of the bundle $(\gamma^*\xi^{\pm},d\lambda^\pm)$. If $u$ has non-trivial asymptotic behaviour at $z$, we define the asymptotic winding number $\operatorname{wind}_\infty^\tau(u,z)$ to be the winding number of the asymptotic eigenvector of $u$ at $z$ with respect to the trivialization $\tau$. If $u$ has trivial asymptotic behaviour at $z$, we define $\operatorname{wind}_\infty^\tau(u,z)$ to be $-\infty$ if $z$ is a positive puncture and $+\infty$ if $z$ is a negative puncture.

Let us choose a trivialization $\tau$ for every puncture of $u$. Since we have a splitting $TX^\pm = \langle \partial_a,R_{\lambda^\pm}\rangle \oplus \xi^\pm$, such a choice of trivializations induces a symplectic trivialization of the restriction of $u^*T\hat{X}$ to a neighbourhood of the punctures. It is possible to choose the trivializations $\tau$ in such a way that the trivialization of $u^*T\hat{X}$ near the punctures extends to a global trivialization on all of $\dot{\Sigma}$. For such a choice of $\tau$, we define
\begin{equation*}
\operatorname{wind}_\infty(u) \coloneqq \sum\limits_{z\in \Gamma^+} \operatorname{wind}_\infty^\tau(u,z) - \sum\limits_{z\in \Gamma^-} \operatorname{wind}_\infty^\tau(u,z) \; \in \Z \cup \{-\infty\}.
\end{equation*}
This quantity does not depend on the choice of $\tau$ subject to the condition explained above.

Let us now further specialize the discussion to the case of a symplectization of a $3$-dimensional contact manifold $(Y^3,\xi =\ker \lambda)$. Let $\pi_\xi : T(\R\times Y)\rightarrow \xi$ denote the projection induced by the splitting $T(\R\times Y) = \langle \partial_a, R_\lambda \rangle\oplus \xi$. The section $\pi_\xi\circ du$ of $\Lambda^{1,0}T^*\dot{\Sigma}\otimes \xi$ either vanishes identically or has isolated zeros with positive multiplicities. In the former case, $u$ is a cover of a trivial cylinder over a periodic Reeb orbit. In the latter case, the quantity $\operatorname{wind}_\pi(u)\geq 0$ is defined to be the count of zeros of $\pi_\xi\circ du$ with multiplicities. The quantities $\operatorname{wind}_\pi(u)$ and $\operatorname{wind}_\infty(u)$ are related by the formula
\begin{equation*}
\operatorname{wind}_\pi(u) = \operatorname{wind}_\infty(u) - \chi(\dot{\Sigma}).
\end{equation*}

\subsection{Fredholm theory}
\label{subsec:fredholm_theory}

In this section, we review some aspects of the Fredholm theory for moduli spaces of punctured pseudoholomorphic curves first developed by Hofer-Wysocki-Zehnder in \cite{hwz99}. Our exposition borrows heavily from Wendl's works \cite{wen10,Wen10b}.

Let us return to the setting of a strong symplectic cobordism $X^{2n}$ of arbitrary dimension $2n$ with convex/concave ends $(Y^\pm,\lambda^\pm)$ and equipped with an admissible almost complex structure $J$. Fix a closed connected oriented surface $\Sigma$. Let $\Gamma = \Gamma^+ \cup \Gamma^-\subset \Sigma$ be a finite set of punctures. For every $z\in \Gamma^\pm$, fix a periodic orbit $\gamma_z$ in $Y^\pm$ and an asymptotic constraint $c_z\geq 0$ such that $\mp c_z \notin \sigma(A_{\gamma_z})$. We consider triples $(j,u,\ell)$ where $j$ is a complex structure on $\Sigma$ compatible with the orientation, $u$ is a finite energy pseudoholomorphic curve $u:(\dot{\Sigma},j)\rightarrow (\hat{X},J)$ which is asymptotic to the orbit $\gamma_z$ with transversal convergence rate strictly bigger than $c_z$ at each puncture $z$, and $\ell$ is a collection of asymptotic markers $\ell_z \in (T_z\Sigma\setminus\left\{ 0 \right\})/\R_{>0}$, one for each puncture $z$, mapping to the marked point $b_{\gamma_z}$ under the asymptotic map $(T_z\Sigma\setminus\left\{ 0 \right\})/\R_{>0}\rightarrow \operatorname{im}(\gamma_z)$ induced by $u$. Note that if $\gamma_z$ is a $k$-fold cover of a simple periodic orbit, then, given $u$, there are $k$ possible choices for the asymptotic marker $\ell_z$. In particular, if $\gamma_z$ is simple, then $\ell_z$ is already determined by $u$. Two triples $(j,u,\ell)$ and $(j',u',\ell')$ are considered equivalent if there exists an orientation and puncture preserving diffeomorphism $\varphi\in \operatorname{Diff}_+(\Sigma,\Gamma)$ such that $\varphi^*j' = j$, $u'\circ \varphi = u$ and $\varphi^*\ell' = \ell$. Let $\MM$ denote the moduli space of all such equivalence classes.

A neighbourhood of a point $[j,u,\ell]$ in $\MM$ can be described as the zero set of a Fredholm section of a Banach bundle. More precisely, the non-linear Cauchy-Riemann operator
\begin{equation*}
\overline{\partial}_J(j,u,\ell) \coloneqq du + J \circ du \circ j
\end{equation*}
can be regarded as a Fredholm section $\overline{\partial}_J:\MT\times \MB \rightarrow \ME$ of an appropriate Banach bundle $\ME$. Here, $\MT$ is a Teichm\"{u}ller slice at $j$ invariant under the action of the group $\operatorname{Aut}(\dot{\Sigma},j)$ of biholomorphic automorphisms of $(\Sigma,j)$ fixing $\Gamma$. The space $\MB$ is a Banach manifold of tuples $(u,\ell)$ where $u:\dot{\Sigma}\rightarrow \hat{X}$ is a Sobolev map asymptotic to the periodic orbits $\gamma_z$ at the punctures with suitable exponential decay and $\ell$ is a collection of asymptotic markers. A neighbourhood of $[j,u,\ell]$ in $\MM$ is in bijection with $\overline{\partial}_J^{-1}(0)/\operatorname{Aut}(\dot{\Sigma},j)$. Note that because of the presence of asymptotic markers, the action of $\operatorname{Aut}(\dot{\Sigma},j)$ on $\overline{\partial}_J^{-1}(0)$ is free even if $u$ is a multiply covered curve.

By slight abuse of notation, let us in the following abbreviate an entire triple $(j,u,\ell)\in \overline{\partial}_J^{-1}(0)$ simply by $u$ whenever convenient. The linearization $D\overline{\partial}_J(u)$ is a Fredholm operator of index
\begin{equation*}
\operatorname{ind}(D\overline{\partial}_J(u)) =  \operatorname{dim}\operatorname{Aut}(\dot{\Sigma},j) + (n-3)\chi(\dot{\Sigma}) + 2c_1^\tau(u^*T\hat{X}) + \sum\limits_{z\in \Gamma^+} \operatorname{CZ}_\tau^{-c_z}(\gamma_z) - \sum\limits_{z\in \Gamma^-}\operatorname{CZ}_\tau^{c_z}(\gamma_z)
\end{equation*}
where $\tau$ is a choice of symplectic trivialization of $\gamma_z^*\xi^{\pm}$ for every asymptotic Reeb orbit. The virtual dimension of $\mathcal{M}$ at $[u]$ is then the number
\[
\operatorname{ind}^c(u) := \operatorname{ind}(D\overline{\partial}_J(u)) - \operatorname{dim}\operatorname{Aut}(\dot{\Sigma},j).
\]
Let us split the linearized operator $D\overline{\partial}_J(u)$ into two components
\begin{equation}
\label{eq:D_partial_J}
D\overline{\partial}_J(u): T_j\MT \oplus T_u\MB \rightarrow \ME_u \quad (y,v) \mapsto G_uy + D_uv.
\end{equation}
The operator $D_u$ is a linear Cauchy-Riemann type operator
\begin{equation*}
D_u : T_u\MB = V_\Gamma \oplus W^{1,p,(\delta,c)}(u^*T\hat{X}) \rightarrow \ME_{u} = L^{p,(\delta,c)}(\Lambda^{0,1}T^*\dot{\Sigma}\otimes u^*T\hat{X}). \end{equation*}
Here, $\delta>0$ is a sufficiently small positive number and $W^{1,p,(\delta,c)}(u^*T\hat{X})$ is a Sobolev space of sections of $u^*T\hat{X}$ with anisotropic exponential weight $(\delta,c)$. More precisely, $W^{1,p,(\delta,c)}(u^*T\hat{X})$ consists of sections $f$ of $u^*T\hat{X}$ of class $W^{1,p}_{\operatorname{loc}}$ such that, near a puncture $z\in \Gamma^\pm$, the projection $\pi_{\xi^\pm}(f)$ is of class $W^{1,p,c_z}$ and the projection $\pi_{\langle \partial_a,R_{\lambda^\pm}\rangle}(f)$ is of class $W^{1,p,\delta}$. Here $W^{1,p,c_z}$ and $W^{1,p,\delta}$ are the usual Sobolev spaces with exponential  weight. This means that if $(s,t)\in Z^{\pm}$ is a cylindrical holomorphic coordinate near $z$, then a section $f(s,t)$ is in $W^{1,p,\varepsilon}$ for some $\varepsilon\in \R$ if and only if $e^{\pm \varepsilon s}f(s,t)$ is in $W^{1,p}$. The space $L^{p,(\delta,c)}(\Lambda^{0,1}T^*\dot{\Sigma}\otimes u^*T\hat{X})$ admits an analogous description. The vector space $V_\Gamma$ has dimension $2\#\Gamma$. For each puncture $z\in \Gamma^\pm$, it contains two linearly independent sections of $u^*T\hat{X}$ supported near $z$ which take values in first summand of the splitting $TX^{\pm} = \langle \partial_a, R_{\lambda^\pm}\rangle \oplus \xi^\pm$ and are constant with respect to the natural trivialization $\langle \partial_a,R_{\lambda^\pm}\rangle\cong\C$ in some neighbourhood of $z$.

If the pseudoholomorphic curve $u$ is immersed, there clearly exists a splitting $u^*T\hat{X} = T_u \oplus N_u$ into a tangent subbundle $T_u$ satisfying $(T_u)_z = \operatorname{im}du(z)$ and a complementary normal subbundle $N_u$. Even if $u$ is not immersed, there still exists such a splitting which satisfies $(T_u)_z = \operatorname{im}du(z)$ at all immersed points $z$ of $u$ by the work of Ivashkovich-Shevchishin \cite{IS99} and Wendl \cite{wen10}. The complex linear part of $D_u$ is a complex linear Cauchy-Riemann operator. We equip the complex vector bundle $u^*T\hat{X}$ with the holomorphic structure induced by this operator. Then the section $du$ of the bundle $T^*\dot{\Sigma}\otimes u^*T\hat{X}$ turns out to be holomorphic \cite[Lemma 1.3.1]{IS99}. If $u$ is not constant, the zeros of $du$ are isolated. In a local holomorphic chart near a zero $z_0$, the section $du$ is given by $(z-z_0)^kF(z)$ where $k > 0$ is the order of the zero and $F(z)$ is holomorphic and satisfies $F(0)\neq 0$. This shows that there exists a unique holomorphic subbundle $T_u$ of $u^*T\hat{X}$ such that $(T_u)_z = \operatorname{im}du(z)$ whenever $du(z)\neq 0$. We call $T_u$ and $N_u$ the {\it generalized tangent} and {\it normal} bundles of $u$. Moreover, we define the quantity
\begin{equation*}
Z(du) \coloneqq \sum\limits_{z\in (du)^{-1}(0)} \operatorname{ord}(du,z) \geq 0.
\end{equation*}
Here $\operatorname{ord}(du,z)$ denotes the order of $z$ as a zero of $du$ as discussed above. Clearly, the curve $u$ is immersed if and only if $Z(du) = 0$.

Using the splitting $u^*T\hat{X} = T_u \oplus N_u$, we can split the domain and codomain of $D_u$ as follows:
\begin{align*}
V_\Gamma \oplus W^{1,p,(\delta,c)}(u^*T\hat{X}) &= (V_\Gamma^T \oplus W^{1,p,\delta}(T_u)) \oplus W^{1,p,c}(N_u) \\
L^{p,(\delta,c)}(\Lambda^{0,1}T^*\dot{\Sigma}\otimes u^*T\hat{X}) &= L^{p,\delta}(\Lambda^{0,1}T^*\dot{\Sigma}\otimes T_u) \oplus L^{p,c}(\Lambda^{0,1}T^*\dot{\Sigma}\otimes N_u).
\end{align*}
Here $V_\Gamma^T$ consists of sections of $T_u$ which are supported near the punctures and asymptotically constant. The operator $D_u$ can be written in block matrix form with respect to these splittings. This block matrix turns out to be upper triangular and takes the following form:
\begin{equation}
\label{eq:block_matrix_form_linearization}
D_u = \left( \begin{matrix}
\overline{\partial}_{T_u} & * \\
0 & D_u^N
\end{matrix} \right).
\end{equation}
Here $\overline{\partial}_{T_u}$ is the Cauchy-Riemann operator on $T_u$ with respect to the holomorphic structure discussed above. The operator
\begin{equation*}
D_u^N : W^{1,p,c}(N_u) \rightarrow L^{p,c}(\Lambda^{0,1}T^*\dot{\Sigma}\otimes N_u)
\end{equation*}
is called the {\it normal operator} of $u$ and has Fredholm index
\begin{equation*}
\operatorname{ind}^c(D_u^N) = \operatorname{ind}^c(u) - 2Z(du).
\end{equation*}
It turns out that the operator
\begin{equation*}
G_u+\overline{\partial}_{T_u} : T_j\MT \oplus V_\Gamma^T \oplus W^{1,p,\delta}(T_u) \rightarrow L^{p,\delta}(\Lambda^{0,1}T^*\dot{\Sigma}\otimes T_u)
\end{equation*}
has index $\operatorname{dim}\operatorname{Aut}(\dot{\Sigma},j) + 2Z(du)$, is complex linear and always surjective \cite{wen10}. In particular, we see that the full deformation operator $D\overline{\partial}_J(u)$ is surjective if and only if the normal operator $D_u^N$ is surjective. Recall that if the almost complex structure $J$ is chosen generically, then $D\overline{\partial}_J(u)$ is surjective for all curves $u$ which are somewhere injective.

We conclude this section with a review of orientations of moduli spaces following \cite[\S 2.13]{par19} (see also \cite{bm04}). The orientation lines $\mathfrak{o}_{D\overline{\partial}_J(u)}$ form a local system on $\overline{\partial}_J^{-1}(0)$, which we call the orientation local system of $\overline{\partial}_J^{-1}(0)$. Whenever $\overline{\partial}_J^{-1}(0)$ is cut out transversely, this orientation local system agrees with the usual orientation local system $\mathfrak{o}_{T_u\overline{\partial}_J^{-1}(0)}$. The action of $\operatorname{Aut}(\dot{\Sigma},j)$ on $\overline{\partial}_J^{-1}(0)$ has a natural lift to an action on the orientation local system. Since the action on $\overline{\partial}_J^{-1}(0)$ is free, we obtain a local system on the quotient $\MM$ whose stalk at $[u]\in \MM$ is canonically isomorphic to $\mathfrak{o}_{D\overline{\partial}_J(u)}$. We can therefore regard $\mathfrak{o}_{D\overline{\partial}_J(u)}\otimes \mathfrak{o}_{T_{\operatorname{id}}\operatorname{Aut}(\dot{\Sigma},j)}^*$ as a local system on $\MM$ and call it the orientation local system of $\MM$. If $\MM$ is cut out transversely, then this is canonically isomorphic to $\mathfrak{o}_{T_{[u]}\MM}$. Since $\operatorname{Aut}(\dot{\Sigma},j)$ is equipped with a natural complex orientation, the orientation local system of $\MM$ is canonically isomorphic to $\mathfrak{o}_{D\overline{\partial}_J(u)}$.

There is a canonical isomorphism of orientation lines
\begin{equation}
\label{eq:orientation_lines_full_and_normal_deformation_operator}
\mathfrak{o}_{D\overline{\partial}_J(u)} \cong \mathfrak{o}_{D_u^N}
\end{equation}
constructed as follows. Because of the block matrix form \eqref{eq:block_matrix_form_linearization} and surjectivity of $G_u + \overline{\partial}_{T_u}$, we have a short exact sequence of kernels
\begin{equation}
\label{eq:short_exact_sequence_of_kernels}
0 \rightarrow \operatorname{ker}(G_u+\overline{\partial}_{T_u}) \rightarrow \operatorname{ker}D\overline{\partial}_J(u) \rightarrow \operatorname{ker}D_u^N \rightarrow 0
\end{equation}
and an isomorphism of cokernels
\begin{equation*}
\operatorname{coker}D\overline{\partial}_J(u)\cong \operatorname{coker}D_u^N.
\end{equation*}
This yields isomorphisms
\begin{equation*}
\mathfrak{o}_{\operatorname{ker}D\overline{\partial}_J(u)} \cong \mathfrak{o}_{\operatorname{ker}(G_u+\overline{\partial}_{T_u})} \otimes \mathfrak{o}_{\operatorname{ker}D_u^N} \quad \text{and}\quad
\mathfrak{o}_{\operatorname{coker}D\overline{\partial}_J(u)} \cong \mathfrak{o}_{\operatorname{coker}D_u^N}.
\end{equation*}
The isomorphism \eqref{eq:orientation_lines_full_and_normal_deformation_operator} is obtained by combining the above isomorphisms and using that $\mathfrak{o}_{\operatorname{ker}(G_u+\overline{\partial}_{T_u})}$ has a natural orientation because $G_u+\overline{\partial}_{T_u}$ is complex linear.

Note that
\begin{equation*}
T_{\operatorname{id}}\operatorname{Aut}(\dot{\Sigma},j) \rightarrow \operatorname{ker}(G_u + \overline{\partial}_{T_u}) \quad X \mapsto du(X)
\end{equation*}
is a complex linear injection. Whenever $u$ is an immersed curve, this injection is an isomorphism. In view of the short exact sequence \eqref{eq:short_exact_sequence_of_kernels}, we obtain an isomorphism
\begin{equation}
\label{eq:kernel_normal_operator}
T_{[u]}\MM \cong \operatorname{ker}D_u^N
\end{equation}
if $u$ is immersed and regular. In this situation, an orientation of $\mathfrak{o}_{D_u^N}$ induces an orientation of $T_{[u]}\MM$ in two different ways: via isomorphism \eqref{eq:orientation_lines_full_and_normal_deformation_operator} and via isomorphism \eqref{eq:kernel_normal_operator}. These two orientations clearly agree.

The orientation local system $\mathfrak{o}_{D\overline{\partial}_J(u)}$ on $\MM$ is canonically isomorphic to a suitable tensor product of the orientation lines of the asymptotic Reeb orbits of $u$ as we now explain. For simplicity, let us assume that all asymptotic weights $c_z$ are equal to zero. The orientation line $\mathfrak{o}_\gamma$ of a nondegenerate periodic Reeb orbit $\gamma$ of a strict contact manifold $(Y,\lambda)$ is defined as follows (cf. \cite[Definition 2.46]{par19}). Consider the vector bundle $V\coloneqq \C \oplus \tilde{\gamma}^*\xi$ on $[0,\infty)\times \T$ where $\tilde{\gamma}(t) = \gamma(T t)$ and $T$ is the period of $\gamma$. Define a real linear Cauchy-Riemann type operator $D$ on $V$ by
\begin{equation*}
D \coloneqq (\partial_s + i\partial_t)\oplus (\partial_s - A_{\gamma}).
\end{equation*}
We regard $[0,\infty)\times \T$ as a subset of $\C$ via $(s,t)\mapsto e^{2\pi(s+it)}$. We arbitrarily extend $V$ to a Hermitian vector bundle over $\C$ and $D$ to a real linear Cauchy-Riemann type operator on $V$. Then $D$ defines a Fredholm operator
\begin{equation*}
D : W^{1,p,\delta}(V) \rightarrow L^{p,\delta}(V)
\end{equation*}
for $\delta>0$ sufficiently small and we define the orientation line $\mathfrak{o}_{\gamma}$ of $\gamma$ to be
\begin{equation*}
\mathfrak{o}_{\gamma} \coloneqq \mathfrak{o}_{D}.
\end{equation*}
It is verified in \cite[Lemma 2.47]{par19} that this is well-defined up to canonical isomorphism. Note that in contrast to \cite{par19}, our basepoint $b_\gamma$ is fixed and suppressed from the notation. Similarly to \cite[Lemma 2.51]{par19}, the Floer-Hofer kernel gluing operation yields a canonical isomorphism
\begin{equation}
\label{eq:isomorphism_orientation_lines_moduli_space_and_orbits}
\mathfrak{o}_{D\overline{\partial}_J(u)}\cong \bigotimes\limits_{z\in \Gamma^+} \mathfrak{o}_{\gamma_z} \otimes \bigotimes\limits_{z\in \Gamma^-} \mathfrak{o}_{\gamma_z}^*.
\end{equation}
Consider the special case that $\MM$ is regular, $0$-dimensional and compact. In this case, a choice of orientation of the right hand side of \eqref{eq:isomorphism_orientation_lines_moduli_space_and_orbits} gives rise to an orientation of $\MM$ and therefore to an integer valued signed count of $\MM$. In other words, we can canonically regard the algebraic count of $\MM$ as an element of the dual of the right hand side of \eqref{eq:isomorphism_orientation_lines_moduli_space_and_orbits}, i.e.
\begin{equation*}
\# \MM \in \bigotimes\limits_{z\in \Gamma^+} \mathfrak{o}_{\gamma_z}^* \otimes \bigotimes\limits_{z\in \Gamma^-} \mathfrak{o}_{\gamma_z}.
\end{equation*}
Note that we have a well-defined absolute value $|\#\MM|\in \Z_{\geq 0}$ independent of choices.

\section{Fast pseudoholomorphic planes}
\label{sec:fast_planes}

The goal of this section is to establish some fundamental properties of moduli spaces of {\it fast pseudoholomorphic planes}. The notion of a fast plane was first introduced by Hryniewicz in \cite{hry12} in order to construct disk-like global surfaces of section bounded by periodic orbits of arbitrarily high Conley-Zehnder index of dynamically convex Reeb flows on $S^3$. Fast planes have played an important role in many further works on $3$-dimensional Reeb dynamics (cf. \cite{hs11, hry14, fk16, hsw23, hss22, hhr24}). The following definition of fast planes in symplectic cobordisms is taken from \cite{hhr24}.

\begin{defn}
\label{def:fast_plane}
Let $X^4$ be a $4$-dimensional strong symplectic cobordism with convex/concave ends $(Y^\pm,\lambda^\pm)$. Let $\hat{J}$ be an admissible almost complex structure on the completion $\hat{X}$. Let $u:(\C,i) \rightarrow (\hat{X},\hat{J})$  be a finite energy plane which is positively asymptotic to a periodic Reeb orbit of $Y^+$ at the puncture at $\infty$. We call $u$ a {\it fast plane} if $\operatorname{wind}_{\infty}(u) \leq 1$.
\end{defn}

\begin{rem}
\label{rem:fast_plane}
{\rm Suppose that $\hat{X}=\R\times Y$ is actually the symplectization of a $3$-dimensional contact manifold $(Y,\lambda)$ and that $\hat{J}$ is symplectization admissible. If $u$ is a fast plane, then we have
\begin{equation*}
0 \leq \operatorname{wind}_\pi(u) = \operatorname{wind}_\infty(u) - \chi(\C) \leq 0.
\end{equation*}
This implies that $\operatorname{wind}_\infty(u) = 1$ and $\operatorname{wind}_\pi(u) = 0$. In particular, we see that $u$ must be immersed (and therefore somewhere injective) and that the projection of $u$ to $Y$ is transverse to the Reeb flow on $Y$. If $\hat{X}$ is not a symplectization, then $\operatorname{wind}_\pi(u)$ is not defined. The plane $u$ might be non-immersed or multiply covered.}
\end{rem}

Throughout the remainder of this section, let $X\subset \R^4$ be a uniformly starshaped domain containing the origin in its interior. We endow the boundary $Y\coloneqq \partial X$ with the contact form $\lambda$ given by the restriction of the one-form
\[
\lambda_0 := \frac{1}{2} \sum_{j=1}^2 (x_j \, dy_j - y_j \, dx_j).
\]
The corresponding contact structure $\xi:= \ker \lambda$ admits a global trivialization, and Conley-Zehnder indices of periodic orbits of the Reeb flow of $\lambda$ refer to such a trivialization.
Let $\gamma$ be a simple periodic orbit of this Reeb flow on $Y$. Assume that the Reeb flow on $Y$ is nondegenerate and dynamically convex up to action $\MA(\gamma)$. This means that all the periodic orbits of action not larger than $\MA(\gamma)$ are nondegenerate and have Conley-Zehnder index at least 3.

Let $J$ be a symplectization admissible almost complex structure on $\R\times Y$. Moreover, let $\hat{J}$ be an admissible almost complex structure on $\hat{X}$ whose restriction to the positive cylindrical end $X^+=[0,\infty)\times Y$ agrees with the restriction of $J$. We let
\begin{equation*}
\MM^{\operatorname{fast}}(\R\times Y,\gamma,J) \qquad \text{and} \qquad \MM^{\operatorname{fast}}(\hat{X},\gamma,\hat{J})
\end{equation*}
denote the moduli spaces of unparametrized fast pseudoholomorphic planes positively asymptotic to $\gamma$ in $\R\times Y$ and $\hat{X}$, respectively. Elements of $\MM^{\operatorname{fast}}(\R\times Y,\gamma,J)$ are represented by parametrized fast pseudoholomorphic planes $u:(\C,i)\rightarrow (\R\times Y,J)$ asymptotic to $\gamma$ and two such maps $u_1$ and $u_2$ are declared equivalent if there exists a biholomorphism $\varphi$ of $\C$ such that $u_2 = u_1\circ \varphi$. Elements of $\MM^{\operatorname{fast}}(\hat{X},\gamma,\hat{J})$ admit an analogous description. Note that we have a natural $\R$ action on $\MM^{\operatorname{fast}}(\R\times Y,\gamma,J)$ given by translation along the $\R$ factor of $\R\times Y$.

We remark that because $\gamma$ is assumed to be a simple orbit, all planes asymptotic to $\gamma$ are necessarily somewhere injective. For fast planes in the symplectization $\R\times Y$, this would also be true without the assumption that $\gamma$ is simple (see Remark \ref{rem:fast_plane}). However, for fast planes in $\hat{X}$ we need the assumption that $\gamma$ is simple to deduce somewhere injectivity.

In the following, we establish some important properties of the moduli spaces of fast planes. We refer to \cite{BEHWZ03} for the notion of SFT convergence and the corresponding compactness theorem.

\begin{prop}[Compactness {\cite[Cor. 4.7]{hss22}, \cite[\S 2.5]{hhr24}}]
\label{prop:fast_planes_compactness}
The following statements are true.
\begin{enumerate}[(i)]
\item The quotient $\MM^{\operatorname{fast}}(\R\times Y,\gamma,J)/\R$ by the natural $\R$-action is compact.
\item Let $(u_k)_k$ be a sequence in $\MM^{\operatorname{fast}}(\hat{X},\gamma,\hat{J})$. Then $(u_k)_k$ has a subsequence which converges to an element of $\MM^{\operatorname{fast}}(\hat{X},\gamma,\hat{J})$ or an element of $\MM^{\operatorname{fast}}(\R\times Y,\gamma,J)$ in the sense of SFT.
\end{enumerate}
\end{prop}

\begin{proof} 
The first assertion is immediate from \cite[Cor. 4.7]{hss22}. The second assertion can be proved using similar arguments (cf. \cite[\S 2.5]{hhr24}). For convenience of the reader, we explain the proof.

Let $(u_k)_k$ be a sequence of fast planes in $\MM^{\operatorname{fast}}(\hat{X},\gamma,\hat{J})$. By the SFT compactness theorem, up to a subsequence $(u_k)_k$ converges to a pseudoholomorphic building. The top level $v$ of this building could be contained either in $\hat{X}$, in which case the building has only one layer, or in $\R\times Y$. Let us first assume that $v$ is contained in $\R\times Y$. In this case, $v$ is a punctured finite energy sphere $v:(\C\setminus \Gamma,i)\rightarrow (\R\times Y,J)$ with exactly one positive puncture $\infty$, at which it is asymptotic to the orbit $\gamma$, and some set of negative punctures $\Gamma$. Our goal is to show that $v\in \MM^{\operatorname{fast}}(\R\times Y,\gamma,J)$, i.e. that  $\Gamma = \emptyset$ and $\operatorname{wind}_{\infty}(v) = 1$.

First observe that since $\gamma$ is a simple orbit, $v$ cannot be a cover of the trivial cylinder $\R\times \gamma$. In particular it must have nontrivial asymptotic behaviour at its positive puncture. Let us parametrize $\C\setminus \left\{ 0 \right\}$ by $\R\times \T \ni (s,t) \mapsto e^{2\pi(s+it)}$. Let $\tau$ be a global trivialization of the contact structure $\xi$ on $Y$. Via $\tau$, we can regard $\pi_\xi(\partial_sv(s,\cdot))$ as a function in $C^\infty(\T,\R^2)$ for all $s$ such that $\left\{ s \right\}\times \T$ does not contain a puncture. For all $s\gg 0$ sufficiently large, this function is nowhere vanishing, i.e. takes values in $\R^2 \setminus \{0\}$, and we have $\operatorname{wind}_\infty^\tau(v,\infty) = \operatorname{wind}(\pi_\xi(\partial_sv(s,\cdot)))$. Let us fix such a value $s_0\gg 0$. After a suitable reparametrization of the planes $u_k$, we can assume that $u_k$ converges to $v$ in $C^\infty_{\operatorname{loc}}(\C\setminus \Gamma)$, up to translation along the $\R$-factor of the symplectization $\R\times Y$. In particular, we have convergence $\pi_\xi(\partial_su_k(s_0,\cdot))\rightarrow \pi_\xi(\partial_sv(s_0,\cdot))$ in $C^\infty(\T)$. For all sufficiently large $k\gg 0$, the function $\pi_\xi(\partial_su_k(s_0,\cdot))$ must therefore be nowhere vanishing and $\operatorname{wind}(\pi_\xi(\partial_su_k(s_0,\cdot))) = \operatorname{wind}(\pi_\xi(\partial_sv(s_0,\cdot)))$. Since the section $\pi_\xi(\partial_su_k)$ is in the kernel of some Cauchy-Riemann type operator, its zeros are isolated and have positive multiplicities. Moreover, $\pi_\xi(\partial_su_k(s,\cdot))$ is nowhere vanishing for all $s\gg 0$ and since the planes $u_k$ are fast, we have $\operatorname{wind}(\pi_\xi(\partial_su_k(s,\cdot)))\leq 1$. For $s>s_0$, the difference $\operatorname{wind}(\pi_\xi(\partial_su_k(s,\cdot))) - \operatorname{wind}(\pi_\xi(\partial_su_k(s_0,\cdot)))$ agrees with the count of zeros of $\pi_\xi(\partial_su_k)$ in the region $[s_0,s]\times \T$, which is non-negative. We conclude that
\begin{equation*}
\operatorname{wind}_\infty^\tau(v,\infty) = \operatorname{wind}(\pi_\xi(\partial_sv(s_0,\cdot))) = \operatorname{wind}(\pi_\xi(\partial_su_k(s_0,\cdot)))\leq 1.
\end{equation*}
We recall that
\begin{equation*}
0\leq \operatorname{wind}_\pi(v) = \operatorname{wind}_\infty^\tau(v,\infty) - \sum\limits_{z\in \Gamma}\operatorname{wind}_\infty^\tau(v,z) - \chi(\C\setminus \Gamma).
\end{equation*}
Dynamical convexity up to action $\MA(\gamma)$ implies that $\operatorname{wind}_\infty^\tau(v,z)\geq 2$ for all $z\in \Gamma$. Therefore, the right hand side of the above equality is bounded from above by $1-2\#\Gamma - (1-\#\Gamma) = -\#\Gamma$. We conclude that $v$ does not have any negative punctures and that $\operatorname{wind}_\infty(v) = 1$. Thus $v\in \MM^{\operatorname{fast}}(\R\times Y,\gamma,J)$.

It remains to consider the case that the top layer $v$ is contained in $\hat{X}$. Clearly, $v$ must be a pseudoholomorphic plane positively asymptotic to $\gamma$. We need to verify that $v$ is fast. This is the case if $v$ has trivial asymptotic behaviour by definition. So assume that $v$ has non-trival asymptotic behaviour. Arguing exactly as before, one can show that $\operatorname{wind}_\infty(v)\leq 1$. This shows that $v$ is fast and concludes the proof of the second assertion.
\end{proof}

In addition, it will be useful to introduce moduli spaces of unparametrized fast pseudoholomorphic planes with one marked point
\begin{equation*}
\MM^{\operatorname{fast}}_1(\R\times Y,\gamma,J) \qquad \text{and} \qquad \MM^{\operatorname{fast}}_1(\hat{X},\gamma,\hat{J}).
\end{equation*}
Elements of $\MM^{\operatorname{fast}}_1(\R\times Y,\gamma,J)$ are represented by pairs $(u,z)$ consisting of a parametrized fast pseudoholomorphic plane $u:(\C,i)\rightarrow (\R\times Y,J)$ positively asymptotic to $\gamma$ and a point $z\in \C$. Two such pairs $(u_1,z_1)$ and $(u_2,z_2)$ are considered to be equivalent if there exists a biholomorphism $\varphi$ of $\C$ such that $u_2 = u_1\circ \varphi$ and $\varphi(z_2)=z_1$. Elements of $\MM^{\operatorname{fast}}_1(\hat{X},\gamma,J)$ have an analogous description. There is a natural forgetful map $\MM^{\operatorname{fast}}_1(\R\times Y,\gamma,J)\rightarrow \MM^{\operatorname{fast}}(\R\times Y,\gamma,J)$ forgetting the marked point. As observed above, all pseudoholomorphic planes asymptotic to $\gamma$ are somewhere injective. Thus the fibres of this forgetful map are simply given by $\C$. The same is true for the forgetful map on $\MM^{\operatorname{fast}}_1(\hat{X},\gamma,\hat{J})$. We have well-defined evaluation maps
\begin{align*}
\operatorname{ev}_{\R\times Y}: \MM^{\operatorname{fast}}_1(\R\times Y,\gamma,J) \rightarrow \R\times Y \qquad (u,z) &\mapsto u(z) \\
\operatorname{ev}_{\hat{X}}: \MM^{\operatorname{fast}}_1(\hat{X},\gamma,\hat{J}) \rightarrow \hat{X} \qquad (u,z) &\mapsto u(z).
\end{align*}
Note that $\operatorname{ev}_{\R\times Y}$ is equivariant with respect to the natural $\R$-actions on $\MM^{\operatorname{fast}}_1(\R\times Y,\gamma,J)$ and $\R\times Y$ by translation.

\begin{prop}[Properness {\cite[Lemma 4.9]{hss22}}]
\label{prop:fast_planes_properness}
The following statements are true.
\begin{enumerate}[(i)]
\item Let $K\subset \R\times (Y\setminus \gamma)$ be a compact subset. Then the preimage of $K$ under the evaluation map $\operatorname{ev}_{\R\times Y}$ is compact.
\item The evaluation map $\operatorname{ev}_{\hat{X}}$ is proper.
\end{enumerate}
\end{prop}

\begin{proof}
Both assertions are a straightforward application of Proposition \ref{prop:fast_planes_compactness}. The first assertion also follows from \cite[Lemma 4.9]{hss22}. We present a proof of the second assertion.

Let $K\subset \hat{X}$ be a compact subset. Let $(u_k,z_k)\in \operatorname{ev}_{\hat{X}}^{-1}(K)$ be a sequence of fast planes with marked point. Let us temporarily ignore the marked points $z_k$ and apply Proposition \ref{prop:fast_planes_compactness} to the underlying sequence of fast planes $u_k$. This allows us to pass to a subsequence which converges to a fast plane $u$ positively asymptotic to $\gamma$, either in the symplectization $\R\times Y$ or in the completion $\hat{X}$. Since each $u_k$ passes through the compact set $K$, the same is true for the limit $u$, which therefore is contained in $\hat{X}$. Next, we apply the SFT compactness theorem to the sequence $(u_k,z_k)$ with the marked points. After passing to a further subsequence, this sequence converges to a pseudoholomorphic building with a marked point. Since we already know that the sequence $u_k$ converges to a single layer building consisting of a fast plane $u$ positively asymptotic to $\gamma$ in $\hat{X}$, there are only two possibilities for the limiting building of $(u_k,z_k)$: a single layer building consisting of a fast plane with marked point $(u,z)$ in $\hat{X}$; or a two layer building consisting of a fast plane $u$ in $\hat{X}$ and the trivial cylinder over $\gamma$ with marked point $(u_\gamma,z)$ in $\R\times Y$. We can immediately rule out the second case because all $u_k(z_k)$ are contained in the compact set $K$. In the first case, we must clearly have $u(z)\in K$. Therefore, $(u,z) \in \operatorname{ev}_{\hat{X}}^{-1}(K)$ is the desired limit of a subsequence of $(u_k,z_k)$. This shows compactness of $\operatorname{ev}_{\hat{X}}^{-1}(K)$ and therefore properness of $\operatorname{ev}_{\hat{X}}$.
\end{proof}

Let $A_\gamma$ denote the asymptotic operator of the orbit $\gamma$. Since $Y$ is dynamically convex up to action $\MA(\gamma)$, we have $\operatorname{CZ}_\tau(\gamma) \geq 3$ where $\tau$ is a global trivialization of the contact structure $\xi$. Equivalently, we have $\alpha_\tau^{<0}(A_\gamma) \geq 1$ and $\alpha_\tau^{\geq 0}(A_\gamma)\geq 2$. This means that we can choose a non-negative constant $c\geq 0$ such that $-c\notin \sigma(A_\gamma)$ and $\alpha_\tau^{<-c}(A_\gamma) = 1$ and $\alpha_\tau^{\geq -c}(A_\gamma) = 2$. A pseudoholomorphic plane in $\R\times Y$ or $\hat{X}$ positively asymptotic to $\gamma$ is fast if and only if its transversal asymptotic convergence rate is strictly bigger than $c$. Thus the Fredholm theory for pseudoholomorphic curves with asymptotic constraints which we summarized in Section \ref{phcce} applies to the moduli spaces of fast pseudoholomorphic planes. Note that since $\gamma$ is a simple periodic orbit, there is a unique choice of asymptotic marker for every plane asymptotic to $\gamma$ and we can therefore omit asymptotic markers from our discussion.

The constrained Fredholm index of a fast plane $u$ is given by
\begin{equation}
\label{eq:fast_planes_constrained_Fredholm_index}
\operatorname{ind}^{c}(u) = -\chi(\C) + \operatorname{CZ}_{\tau}^{-c}(\gamma) = 2.
\end{equation}
Recall that there is a normal deformation operator
\begin{equation*}
D_u^N : W^{1,p,c}(N_u) \rightarrow L^{p,c}(\Lambda^{0,1}T^*\C\otimes N_u)
\end{equation*}
of Fredholm index
\begin{equation}
\label{eq:fast_planes_Fredholm_index_normal_operator}
\operatorname{ind}^c(D_u^N) = \operatorname{ind}^c(u) - 2 Z(du)
\end{equation}
with the property that $u$ is regular if and only if $D_u^N$ is surjective. Note that the index $\operatorname{ind}^c(D_u^N)$ is always even. Therefore, the orientation line $\mathfrak{o}_{D_u^N}$ admits a canonical complex orientation by Remark \ref{rem:canonical_orientation_plane}. This induces a canonical orientation of our moduli spaces of fast planes, as explained at the end of Section \ref{subsec:fredholm_theory}. In the following, we equip the moduli spaces with this orientation.

\begin{prop}[Automatic regularity]
\label{prop:fast_planes_automatic_regularity}
The following statements are true.
\begin{enumerate}[(i)]
\item \label{item:fast_planes_automatic_regularity_symplectization} Every fast plane in $\R\times Y$ is automatically regular. In particular, the moduli spaces $\MM^{\operatorname{fast}}(\R\times Y,\gamma,J)$ and $\MM^{\operatorname{fast}}_1(\R\times Y,\gamma,J)$ are smooth manifolds of dimensions $2$ and $4$, respectively.
\item \label{item:fast_planes_automatic_regularity_completion} Every fast plane $u$ in $\hat{X}$ satisfying $Z(du)\in \left\{ 0,1 \right\}$ is automatically regular. In particular, open neighbourhoods of the subsets of $\MM^{\operatorname{fast}}(\hat{X},\gamma,\hat{J})$ and $\MM^{\operatorname{fast}}_1(\hat{X},\gamma,\hat{J})$ consisting of fast planes satisfying $Z(du)\in \left\{ 0,1 \right\}$ are smooth manifolds of dimensions $2$ and $4$, respectively.
\item \label{item:fast_planes_automatic_regularity_evaluation_map_symplectization} The evaluation map $\operatorname{ev}_{\R\times Y}$ is an orientation preserving local diffeomorphisms.
\item \label{item:fast_planes_automatic_regularity_evaluation_map_completion} The restriction of the evaluation map $\operatorname{ev}_{\hat{X}}$ to the open subset consisting of all immersed fast planes with marked point is an orientation preserving local diffeomorphism.
\item \label{item:fast_planes_automatic_regularity_evaluation_map_non_immersed} Let $(u,z)\in \MM^{\operatorname{fast}}_1(\hat{X},\gamma,\hat{J})$ and assume that $Z(du) = 1$. Then 
\begin{equation*}
\operatorname{im}d\operatorname{ev}_{\hat{X}}(u,z) = (T_u)_z \subset T_{u(z)}\hat{X}.
\end{equation*}
\end{enumerate}
\end{prop}

\begin{proof}
Let $u$ be a fast plane in $W$ where $W\in \{\R\times Y, \hat{X}\}$. By Proposition \ref{prop:automatic_regularity_linear_CR_operator}, the operator $D_u^N$ is surjective if and only if $\operatorname{ind}^c(D_u^N)\geq 0$. It is immediate from equations \eqref{eq:fast_planes_constrained_Fredholm_index} and \eqref{eq:fast_planes_Fredholm_index_normal_operator} that this is the case if and only if $Z(du) \in \left\{ 0,1 \right\}$, i.e. if and only if $u$ is immersed or has exactly one non-immersed point of order $1$. This shows statement \eqref{item:fast_planes_automatic_regularity_completion} and also statement \eqref{item:fast_planes_automatic_regularity_symplectization} because fast planes in the symplectization are automatically immersed.

Next, consider an immersed fast plane with marked point $(u,z)\in \MM^{\operatorname{fast}}_1(W,\gamma)$. By equations \eqref{eq:fast_planes_constrained_Fredholm_index} and \eqref{eq:fast_planes_Fredholm_index_normal_operator}, we have $\operatorname{ind}^c(D_u^N) = 2$. We may split the tangent space
\begin{equation}
\label{eq:fast_planes_automatic_regularity_proof_splitting_tangent_spaces}
T_{(u,z)}\MM_1^{\operatorname{fast}}(W,\gamma) = T_z\C \oplus \operatorname{ker}(D_u^N)
\end{equation}
into a component $T_z\C$ tangent to the fiber of the forgetful map $\MM_1^{\operatorname{fast}}(W,\gamma)\rightarrow \MM^{\operatorname{fast}}(W,\gamma)$ and a component $\operatorname{ker}(D_u^N)$ transverse to the fiber. The linearization of the evaluation map $\operatorname{ev}_{W}$ at $(u,z)$ is given by
\begin{equation}
\label{eq:fast_planes_automatic_regularity_proof_splitting_linearization}
d\operatorname{ev}_{W}(u,z) : T_z\C \oplus \operatorname{ker}(D_u^N) \rightarrow (T_u)_z \oplus (N_u)_z = T_{u(z)}W \qquad (w,X) \mapsto (du(z)w, X(z)).
\end{equation}
The first component of this linear map is complex linear and bijective because $u$ is immersed by assumption. The second component is bijective and orientation preserving by Proposition \ref{prop:automatic_regularity_linear_CR_operator}. Therefore, $d\operatorname{ev}_{W}(u,z)$ is invertible and orientation preserving. This implies statements \eqref{item:fast_planes_automatic_regularity_evaluation_map_symplectization} and \eqref{item:fast_planes_automatic_regularity_evaluation_map_completion}.

Assume now that $(u,z) \in \MM^{\operatorname{fast}}_1(\hat{X},\gamma,\hat{J})$ satisfies $Z(du)=1$. In this case, the normal operator $D_u^N$ has index $0$ and is therefore bijective. The tangent part of the deformation operator $D_u$ is the Cauchy-Riemann operator $\overline{\partial}_{T_u}$ induced by the holomorphic structure of $T_u$
\begin{equation*}
\overline{\partial}_{T_u}: V_\gamma^T \oplus W^{1,p,\delta}(T_u) \rightarrow L^{p,\delta}(\Lambda^{0,1}T^*\C\otimes T_u).
\end{equation*}
It is surjective and has index $6$. We can explicitly identify its kernel. Let $z_0\in \C$ denote the unique critical point of $du$. Then $\operatorname{ker}\overline{\partial}_{T_u}$ consists of all sections of $T_u$ of the form
\begin{equation}
\label{eq:fast_planes_automatic_regularity_proof_kernel}
\C \ni z \mapsto du(z) (a_1(z-z_0) + a_0 + a_{-1}(z-z_0)^{-1}) \in (T_u)_z
\end{equation}
where $a_1,a_0,a_{-1}\in \C$. Note that \eqref{eq:fast_planes_automatic_regularity_proof_kernel} is a well-defined section of $T_u$ because in a holomorphic chart near $z_0$, the derivative of $u$ can be written as $du(z) = (z-z_0)F(z)$ where $F$ is holomorphic and satisfies $F(z_0)\neq 0$. For every fixed $z\in \C$, the map
\begin{equation}
\label{eq:fast_planes_automatic_regularity_proof_kernel_eval_z}
\C^3 \ni (a_1,a_0,a_{-1}) \mapsto du(z) (a_1(z-z_0) + a_0 + a_{-1}(z-z_0)^{-1}) \in (T_u)_z
\end{equation}
is surjective. If $z$ is not a zero of $du$, this is clear because $(T_u)_z = \operatorname{im}du(z)$. If $z=z_0$ is the unique zero of $du$, write $du(z) = (z-z_0)F(z)$ as above and recall that $(T_u)_{z_0} = \C \cdot F(z_0)$ by definition. Surjectivity of the map \eqref{eq:fast_planes_automatic_regularity_proof_kernel_eval_z} implies that the image of $d\operatorname{ev}_W(u,z)$ is equal to $(T_u)_z$.
\end{proof}

\begin{prop}[Generic regularity]
\label{prop:fast_planes_generic_regularity}
Suppose that the almost complex structure $\hat{J}$ on $\hat{X}$ is generic or a member of a generic $1$-parameter family. Then every $u\in \MM^{\operatorname{fast}}(\hat{X},\gamma,\hat{J})$ satisfies $Z(du)\in \left\{ 0,1 \right\}$. Moreover, the set of $u$ satisfying $Z(du)=1$ is finite.
\end{prop}

\begin{rem}
{\rm Note that any fast plane in $\hat{X}$ which does not intersect the interior of $X$ is automatically immersed by Remark \ref{rem:fast_plane}. In Proposition \ref{prop:fast_planes_generic_regularity} it is therefore possible to fix the almost complex structure $\hat{J}$ on the positive cylindrical end $X^+=[0,\infty)\times Y$ and only make it generic in the interior of $X$.}
\end{rem}

\begin{proof}
We assume by contradiction that $u\in\MM^{\operatorname{fast}}(\hat{X},\gamma,\hat{J})$ satisfies $Z(du)\geq 2$. By equations \eqref{eq:fast_planes_constrained_Fredholm_index} and \eqref{eq:fast_planes_Fredholm_index_normal_operator}, the normal operator $D_u^N$ has index at most $-2$. Hence the image of $D_u^N$ has codimension at least $2$ and the same must be true for the image of the full deformation operator $D_u$. Recall that $u$ is somewhere injective because $\gamma$ is a simple orbit. Thus it follows from the assumption that $\hat{J}$ is contained in a generic $1$-parameter family that the image of $D_u$ has codimension at most $1$. This is a contradiction and we conclude that $u$ cannot exist.

It remains to argue that the set of fast planes $u$ satisfying $Z(du) = 1$ is finite. Consider the moduli space
\begin{equation*}
\MM^{\operatorname{fast}}_{\operatorname{crit}}(\hat{X},\gamma,\hat{J}) \coloneqq \left\{ (u,z)\in \MM^{\operatorname{fast}}_1(\hat{X},\gamma,\hat{J}) \mid du(z) = 0 \right\}.
\end{equation*}
We claim that this moduli space is compact. Indeed, let $(u_k,z_k)$ be a sequence in $\MM^{\operatorname{fast}}_{\operatorname{crit}}(\hat{X},\gamma,\hat{J})$. After passing to a subsequence, we can assume that $(u_k,z_k)$ converges in the sense of SFT to a limiting pseudoholomorphic building with one marked point. By Proposition \ref{prop:fast_planes_compactness}, the sequence $u_k$ without the marked point converges to a fast plane $u$ either in $\hat{X}$ or in $\R\times Y$. This results in the following possible options for the SFT limit of the sequence $(u_k,z_k)$ with marked point: a single level building consisting of a fast plane with marked point $(u,z)$ asymptotic to $\gamma$ in $\hat{X}$ or in $\R\times Y$; or a two level building with bottom level a fast plane $u$ asymptotic to $\gamma$ in $\hat{X}$ or $\R\times Y$ and top level the trivial cylinder with marked point $(u_\gamma,z)$ over $\gamma$ in $\R\times Y$. Since $du_\gamma(z)\neq 0$, the sequence $(u_k,z_k)$, which satisfies $du_k(z_k)=0$, clearly cannot converge to such a two level building. Similarly, since all fast planes in symplectizations are immersed, we would have $du(z)\neq 0$ in the case of a single level limiting building in the symplectization $\R\times Y$, also ruling out this case. Hence the limiting building must be a fast plane $(u,z)$ in $\hat{X}$ satisfying $du(z)= 0$, i.e. an element of $\MM^{\operatorname{fast}}_{\operatorname{crit}}(\hat{X},\gamma,\hat{J})$. This implies that $\MM^{\operatorname{fast}}_{\operatorname{crit}}(\hat{X},\gamma,\hat{J})$ is compact. Our goal is to show that it is also regular and $0$-dimensional. Finiteness is then an immediate consequence.

Let $(u,z_0)\in \MM^{\operatorname{fast}}_{\operatorname{crit}}(\hat{X},\gamma,\hat{J})$. We recall from \cite[\S 4.7]{wen16} that one can set up a Fredholm theory for the moduli space $\MM^{\operatorname{fast}}_{\operatorname{crit}}(\hat{X},\gamma,\hat{J})$ such that the deformation operator $D_u^{\operatorname{crit}}$ is simply given by the restriction of the deformation operator $D_u$ to the subspace $V_\gamma \oplus W^{2,p,(\delta,c)}_{\operatorname{crit}}(u^*T\hat{X})$ where
\begin{equation*}
W^{2,p,(\delta,c)}_{\operatorname{crit}}(u^*T\hat{X}) \coloneqq \left\{ \eta \in W^{2,p,(\delta,c)}(u^*T\hat{X}) \mid \nabla^{1,0} \eta(z_0)= 0 \right\}.
\end{equation*}
Note that $W^{2,p,(\delta,c)}_{\operatorname{crit}}(u^*T\hat{X})$ has codimension $4$ in $W^{2,p,(\delta,c)}(u^*T\hat{X})$. Since $D_u$ has index $6$, we conclude that $D_u^{\operatorname{crit}}$ has index $2$. Recall from the proof of Proposition \ref{prop:fast_planes_automatic_regularity} that the kernel of $D_u$ consists of all sections of $T_u$ of the form
\begin{equation*}
\C \in z \mapsto du(z) (a_1 (z-z_0) + a_0 + a_{-1}(z-z_0)^{-1}) \in (T_u)_z
\end{equation*}
where $a_j \in \C$. In a local holomorphic chart around $z_0$ we can write $du(z) = (z-z_0)F(z)$ where $F$ is holomorphic. A simple computation shows that an element of the kernel of $D_u$ is contained in $V_\gamma \oplus W^{2,p,(\delta,c)}_{\operatorname{crit}}$ if and only if
\begin{equation*}
a_0 F(z_0) + a_{-1} F'(z_0) = 0.
\end{equation*}
Thus the dimension of $\operatorname{ker}D_u^{\operatorname{crit}}$ is $2$ if $F(z_0)$ and $F'(z_0)$ are linearly independent over $\C$ and $4$ otherwise. Since the index of $D_u^{\operatorname{crit}}$ is $2$, this means that $D_u^{\operatorname{crit}}$ is surjective in the former case and has cokernel of dimension $2$ in the latter case. Since $\hat{J}$ is contained in a generic $1$-parameter family, the latter case cannot happen. We conclude that the moduli space of parametrized fast planes with vanishing derivative at $z_0$ is regular and $2$-dimensional. Note that the group of biholomorphisms of $\C$ fixing $z_0$ is $2$-dimensional. Hence the moduli space of unparametrized fast planes with vanishing derivative at a marked point is $0$-dimensional.
\end{proof}

By Proposition \ref{prop:fast_planes_automatic_regularity} the evaluation map
\begin{equation*}
\operatorname{ev}_{\R\times Y} : \MM^{\operatorname{fast}}_1(\R\times Y,\gamma,J) \rightarrow \R\times Y
\end{equation*}
is an orientation preserving local diffeomorphism between smooth $4$-manifolds for every symplectization admissible almost complex structure $J$. Whenever the moduli space $\MM^{\operatorname{fast}}_1(\R\times Y,\gamma,J)$ is non-empty, the evaluation map $\operatorname{ev}_{\R\times Y}$ restricts to a covering map
\begin{equation*}
\operatorname{ev}_{\R\times Y}|_{\operatorname{ev}_{\R\times Y}^{-1}(\R\times (Y\setminus \gamma))} : \operatorname{ev}_{\R\times Y}^{-1}(\R\times (Y\setminus \gamma)) \rightarrow \R\times (Y\setminus \gamma)
\end{equation*}
of finite positive degree by Proposition \ref{prop:fast_planes_properness}. We let $\operatorname{deg}(\operatorname{ev}_{\R\times Y})>0$ denote the degree of this covering map. In the case that $\MM^{\operatorname{fast}}_1(\R\times Y,\gamma,J)$ is empty, we set $\operatorname{deg}(\operatorname{ev}_{\R\times Y}) \coloneqq 0$. Note that for every point $z\in \R\times (Y\setminus \gamma)$, the degree $\operatorname{deg}(\operatorname{ev}_{\R\times Y})$ agrees with both the algebraic and the geometric count of $\operatorname{ev}_{\R\times Y}^{-1}(z)$.

If $\hat{J}$ is a generic admissible almost complex structure on $\hat{X}$, then
\begin{equation*}
\operatorname{ev}_{\hat{X}} : \MM^{\operatorname{fast}}_1(\hat{X},\gamma,\hat{J}) \rightarrow \hat{X}
\end{equation*}
is a proper smooth map between smooth oriented $4$-manifolds by Propositions \ref{prop:fast_planes_properness}, \ref{prop:fast_planes_automatic_regularity} and \ref{prop:fast_planes_generic_regularity} and as such has a well-defined degree $\operatorname{deg}(\operatorname{ev}_{\hat{X}})$. It is an orientation preserving local diffeomorphism on the complement of the locus of non-immersed fast planes, which consists of finitely many fibers of the forgetful map $\MM^{\operatorname{fast}}_1(\hat{X},\gamma,\hat{J})\rightarrow \MM^{\operatorname{fast}}(\hat{X},\gamma,\hat{J})$. Therefore, the degree $\operatorname{deg}(\operatorname{ev}_{\hat{X}})$ is non-negative and zero exactly if $\MM^{\operatorname{fast}}_{1}(\hat{X},\gamma,\hat{J})$ is empty. If $z\in \hat{X}$ is a point not contained in the union of the images of all non-immersed planes, which is a set of codimension $2$, then the algebraic and geometric counts of $\operatorname{ev}_{\hat{X}}^{-1}(z)$ both agree with $\operatorname{deg}(\operatorname{ev}_{\hat{X}})$.

\begin{prop}[Degree]
\label{prop:degree_invariance}
The degrees $\operatorname{deg}(\operatorname{ev}_{\R\times Y})$ and $\operatorname{deg}(\operatorname{ev}_{\hat{X}})$ agree and are independent of the choice of the almost complex structure.
\end{prop}

\begin{proof}
Let us begin by showing that $\operatorname{deg}(\operatorname{ev}_{\R\times Y})$ is independent of the choice of $J$. To this end, consider two symplectization admissible almost complex structures $J_0$ and $J_1$. For $k\in \left\{ 0,1 \right\}$, let $\operatorname{ev}_{\R\times Y}^k$ denote the resulting evaluation maps. Connect $J_0$ to $J_1$ via a smooth path $(J_t)_{t\in [0,1]}$. Let $\MM^{\operatorname{fast}}_1(\R\times Y,\gamma,(J_t)_{t\in [0,1]})$ denote the moduli space consisting of all tuples $(t,(u,z))$ where $t\in [0,1]$ and $[u,z]\in \MM^{\operatorname{fast}}_1(\R\times Y,\gamma,J_t)$. We have an evaluation map
\begin{equation*}
\operatorname{ev}_{[0,1]\times \R\times Y}:\MM^{\operatorname{fast}}_1(\R\times Y,\gamma,(J_t)_{t\in [0,1]}) \rightarrow [0,1]\times \R\times Y \qquad (t,(u,z)) \mapsto (t,u(z)).
\end{equation*}
If follows exactly as in the proofs of Propositions \ref{prop:fast_planes_compactness}, \ref{prop:fast_planes_properness} and \ref{prop:fast_planes_automatic_regularity} that the restriction of $\operatorname{ev}_{[0,1]\times \R\times Y}$ to the preimage of $[0,1]\times \R\times (Y\setminus \gamma)$ is a finite covering map. Over $\left\{ 0,1 \right\}\times \R\times (Y\setminus \gamma)$ this covering map restricts to the covering maps induced by $\operatorname{ev}_{\R\times Y}^0$ and $\operatorname{ev}_{\R\times Y}^1$, which therefore have the same degree.

Next, consider a generic admissible almost complex structure $\hat{J}$ on $\hat{X}$. Let $J$ be the symplectization admissible almost complex structure on $\R\times Y$ whose restriction to $X^+$ agrees with $\hat{J}$. Fix a point $q\in Y\setminus \gamma$. We claim that, for every sufficiently large $a \gg 0$, every fast plane in $\R\times Y$ or $\hat{X}$ asymptotic to $\gamma$ and passing through $(a,q)$ is contained in $X^+$. In the case of $\R\times Y$, this easily follows from Propsitions \ref{prop:fast_planes_properness} and \ref{prop:fast_planes_automatic_regularity}. Indeed, there are only finitely many fast planes $u_i$ passing through $(0,q)$. For every $a$, the fast planes passing through $(a,q)$ are precisely the translates of $u_i$ along the $\R$-factor of $\R\times Y$ by amount $a$. For $a\gg 0$, these finitely many translates are clearly contained in $X^+$. Next, we consider the case of $\hat{X}$. Assume by contradiction that the claim fails. Then we may pick, for every $k\geq 0$, a fast plane with marked point $(u_k,z_k)$ which intsects $X$ and satisfies $u_k(z_k) = (k,q)\in X^+$. By Proposition \ref{prop:fast_planes_compactness} applied to the sequence $u_k$ without the marked points, we may pass to a subsequence which converges to a fast plane $u$ in either $\hat{X}$ or the symplectization $\R\times Y$. Since each $u_k$ intersects $X$, the same must be true for $u$, which is therefore contained in $\hat{X}$. We apply the SFT compactness theorem to the sequence $(u_k,z_k)$ with marked points and pass to a further subsequence which converges to a pseudoholomorphic building with marked point. As in the proof of Proposition \ref{prop:fast_planes_properness}, there are two options for the limiting building: a single layer building consisting of a fast plane with marked point $(u,z)$ in $\hat{X}$; or a two layer building consisting of a fast plane $u$ in $\hat{X}$ and the trivial cylinder over $\gamma$ with marked point $(u_\gamma,z)$ in $\R\times Y$. In the first case, we have $u_k(z_k)\rightarrow u(z)$ as $k\rightarrow \infty$, which clearly contradicts $u_k(z_k) = (k,q)$. In the second case, we have $\operatorname{pr}_Y(u_k(z_k))\rightarrow \operatorname{pr}_Y(u_\gamma(z))$ as $k\rightarrow \infty$, where $\operatorname{pr}_Y:\R\times Y\rightarrow Y$ denotes the natural projection. This is also a contradiction because $\operatorname{pr}_Y(u_k(z_k)) = q \notin \gamma$ and $\operatorname{pr}_Y(u_\gamma(z)) \in \gamma$. This concludes the proof of the claim in the case of $\hat{X}$.
 
Let us now fix $a\gg 0$ sufficiently large as above. Then we have $\operatorname{ev}_{\R\times Y}^{-1}(a,q) = \operatorname{ev}_{\hat{X}}^{-1}(a,q)$. Since the degrees of $\operatorname{ev}_{\R\times Y}$ and $\operatorname{ev}_{\hat{X}}$ are both given by the geometric count of this set, they must agree. Finally, we observe that since we already know that $\operatorname{deg}(\operatorname{ev}_{\R\times Y})$ is independent of the almost complex structure, the same must be true for $\operatorname{deg}(\operatorname{ev}_{\hat{X}})$.
\end{proof}

\section{Detecting Hopf orbits}
\label{sec:detecting_hopf}

The main result of this section states that the degree of the evaluation map from the moduli space of fast planes with marked point discussed in Section \ref{sec:fast_planes} detects Hopf orbits. Before stating it, let us make the following definition. Let $X\subset \R^4$ be a domain which is  uniformly starshaped with respect to the origin and let $(Y,\xi = \operatorname{ker}\lambda)$ be the contact boundary of $X$.

\begin{defn} Let $K\subset (Y,\xi)$ be a transverse knot. We call an immersed disk $D\subset X$ a {\it positively immersed symplectic disk with boundary $K$} if it satisfies the following properties:
\begin{enumerate}[(i)]
\item $D$ is symplectic, i.e. the symplectic form on $X$ restricts to an area form on $D$.
\item Consider a collar neighbourhood $(-\varepsilon,0]\times Y$ of $Y$ induced by the radial Liouville vector field. For $\varepsilon>0$ sufficietly small, the intersection of $D$ with this collar neighbourhood is given by $(-\varepsilon,0]\times K$.
\item All self-intersection points of $D$ are positive transverse double points. Positivity means that if we orient $D$ via the symplectic form, then the orientations of the two branches of $D$ meeting at a double point induce the natural symplectic orientation of $X$.
\end{enumerate}
\end{defn}

\begin{rem}
{\rm The existence of such a positively immersed symplectic disk only depends on the transverse knot type of $K$. If $X'$ is another uniformly starshaped domain with contact boundary $Y'$, then the radial projection $Y\rightarrow Y'$ is a contactomorphism. If $K\subset Y$ and $K'\subset Y'$ are two transverse knots corresponding to each other via this contactomorphism, then $K$ bounds a positively immersed symplectic disk in $X$ if and only if $K'$ bounds a positively immersed symplectic disk in $X'$.}
\end{rem}

If $D$ is a positively immersed symplectic disk bounded by $K$, then the self-linking number of $K$ and the number of double points of $D$ are related by the formula
\begin{equation*}
\operatorname{sl}(K) = -1 + 2 \cdot \#(\text{double points of $D$}),
\end{equation*}
see \cite[proof of Corollary 1.9]{hwz99}. In particular, we see that if $K$ bounds one embedded symplectic disk, then all positively immersed symplectic disks bounded by $K$ are embedded.

Let us now return to the setting of Section \ref{sec:fast_planes}. Let $X\subset \R^4$ be a uniformly starshaped domain with boundary $Y\coloneqq \partial X$, let $\gamma$ be a simple periodic orbit on $Y$, and assume that the Reeb flow on $Y$ is nondegenerate and dynamically convex up to action $\MA(\gamma)$. Recall that the evaluation maps
\begin{equation*}
\operatorname{ev}_{\R\times Y} : \MM^{\operatorname{fast}}_1(\R\times Y,\gamma,J)\rightarrow \R\times Y \quad \text{and}\quad \operatorname{ev}_{\hat{X}} : \MM^{\operatorname{fast}}_1(\hat{X},\gamma,\hat{J}) \rightarrow \hat{X}
\end{equation*}
have well-defined non-negative finite degrees which moreover agree and are independent of the choice of almost complex structure. In the following, let us abbreviate this common degree by $\operatorname{deg}(\operatorname{ev})$.

\begin{thm}
\label{thm:degree_characterizes_hopf}
The degree $\operatorname{deg}(\operatorname{ev})$ is non-negative, and strictly positive if and only if $\gamma$ bounds a positively immersed symplectic disk. Moreover, the following statements are equivalent:
\begin{enumerate}[(i)]
\item \label{item:degree_characterizes_hopf_deg_1} $\operatorname{deg}(\operatorname{ev})=1$;
\item \label{item:degree_characterizes_hopf_embedded_disk} $\gamma$ bounds an embedded symplectic disk;
\item \label{item:degree_characterizes_hopf_hopf_orbit} $\gamma$ is a Hopf orbit;
\item \label{item:degree_characterizes_hopf_dgss} $\gamma$ bounds a disk-like global surface of section.
\end{enumerate}
\end{thm}

Note that Theorem \ref{thm:degree_characterizes_hopf} fully characterizes in purely topological terms when a simple periodic orbit on the boundary of a nondegenerate dynamically convex domain is the asymptotic limit of a fast pseudoholomorphic plane. This is particularly interesting in view of \cite[Theorem 1.4]{hss22}. This theorem states that if a periodic orbit on a nondegenerate dynamically convex domain is the limit of a fast plane, then it actually bounds a global surface of section for the Reeb flow. We obtain the following immediate corollary:

\begin{cor}
Let $X\subset \R^4$ be a nondegenerate dynamically convex domain and let $\gamma$ be a simple periodic orbit whose underlying transverse knot type admits a positively immersed symplectic disk. Then $\gamma$ bounds a global surface of section.
\end{cor}

\begin{proof}[Proof of Theorem \ref{thm:degree_characterizes_hopf}]
We begin by showing three claims.

\begin{claim}
\label{claim:degree_characterization_hopf_pisd_implies_deg_g_0}
Assume that $\gamma$ bounds a positively immersed symplectic disk. Then $\operatorname{deg}(\operatorname{ev})>0$.
\end{claim}

\begin{proof}
Starting from a positively immersed symplectic disk with boundary $\gamma$, one can construct a positively immersed symplectic plane in $\hat{X}$ whose intersection with the positive cylindrical end $X^+$ is given by the half cylinder $[0,\infty)\times \gamma$. We may pick an admissible almost complex structure $\hat{J}$ on $\hat{X}$ which turns this positively immersed symplectic plane into an immersed pseudoholomorphic plane $u$ positively asymptotic to $\gamma$. Note that $u$ has trivial asymptotic behaviour at $\gamma$ and is therefore a fast plane. By assertion \eqref{item:fast_planes_automatic_regularity_completion} in Proposition \ref{prop:fast_planes_automatic_regularity}, the fast plane $u$ is regular. It therefore persists under small perturbations of the almost complex structure and remains immersed. This shows that there exists a generic admissible almost complex structure $\hat{J}$, still denoted by the same symbol, and an immersed fast plane $u\in \MM^{\operatorname{fast}}(\hat{X},\gamma,\hat{J})$. By Proposition \ref{prop:fast_planes_generic_regularity} the moduli space $\MM^{\operatorname{fast}}(\hat{X},\gamma,\hat{J})$ only contains finitely many non-immersed planes. The union of their images cannot contain the image of the immersed plane $u$. Hence we can pick a point $z\in\C$ such that $u(z)$ is not contained in any non-immersed plane. As observed above, the degree $\operatorname{deg}(\operatorname{ev})$ is equal to the geometric count $\operatorname{ev}_{\hat{X}}^{-1}(u(z))$. This count is strictly positive because $(u,z)$ is contained in $\operatorname{ev}_{\hat{X}}^{-1}(u(z))$.
\end{proof}

\begin{claim}
\label{claim:degree_characterization_hopf_non_embedded_pisd_implies_deg_g_1}
Assume that $\gamma$ bounds a positively immersed symplectic disk which is not embedded. Then $\operatorname{deg}(\operatorname{ev})>1$.
\end{claim}

\begin{proof}
By Claim \ref{claim:degree_characterization_hopf_pisd_implies_deg_g_0}, the degree $\operatorname{deg}(\operatorname{ev})$ is non-negative. We can therefore find an immersed fast plane $u$ asymptotic to $\gamma$. Since $\gamma$ bounds a non-embedded positively immersed symplectic disk, the plane $u$ cannot be embedded. This means that there exist distinct points $z_1\neq z_2\in \C$ such that $u(z_1)=u(z_2) \eqqcolon p$. We observe that $(u,z_1)$ and $(u,z_2)$ represent distinct elements of the moduli space $\MM^{\operatorname{fast}}_1(\hat{X},\gamma,\hat{J})$. Indeed, since $u$ is somewhere injective, any biholomorphism $\varphi$ of $\C$ satisfying $u\circ \varphi = u$ is necessarily equal to the identity and can certainly not carry $z_2$ to $z_1$. This shows that the geometric count of $\operatorname{ev}_{\hat{X}}^{-1}(p)$ is strictly bigger than $1$. By assertion \eqref{item:fast_planes_automatic_regularity_evaluation_map_completion} in Proposition \ref{prop:fast_planes_automatic_regularity}, the evaluation map $\operatorname{ev}_{\hat{X}}$ is a local diffeomorphism near the points $(u,z_1)$ and $(u,z_2)$ in $\MM^{\operatorname{fast}}_1(\hat{X},\gamma,\hat{J})$. This implies that the geometric count of $\operatorname{ev}_{\hat{X}}^{-1}(q)$ is strictly bigger than $1$ for all $q$ sufficiently close to $p$. Since a generic such $q$ is not contained in the image of any non-immersed fast plane in $\MM^{\operatorname{fast}}_1(\hat{X},\gamma,\hat{J})$, we obtain that the degree $\operatorname{deg}(\operatorname{ev})$ is strictly bigger than $1$.
\end{proof}

\begin{claim}
\label{claim:degree_characterization_hopf_embedded_disk_implies_deg_1}
Assume that $\gamma$ bounds an embedded symplectic disk. Then $\operatorname{ev}_{\hat{X}}$ is a diffeomorphism and $\operatorname{ev}_{\R\times Y}$ is a diffeomorphism onto $\R\times (Y\setminus \gamma)$. In particular, $\operatorname{deg}(\operatorname{ev})=1$.
\end{claim}

\begin{proof}
The proof of this claim is probably well-known to experts. For convenience of the reader, we still carry it out. Since $\gamma$ bounds an embedded symplectic disk, any pseudoholomorphic plane asymptotic to $\gamma$ must be embedded. Our goal is to show that any two fast planes $u$ and $v$ asymptotic to $\gamma$ with distinct images must in fact have disjoint images. Let us fix two such fast planes $u$ and $v$. Since $u$ and $v$ have distinct images, at least one of them, say $u$, has non-trivial asymptotic behaviour. Let $e_u$ and $\nu_u$ be the asymptotic eigenfunction and eigenvalue of $u$. Moreover, let $e$ and $\nu$ be the relative asymptotic eigenfunction and eigenvalue of $u$ and $v$. Note that we have $\nu \leq \nu_u < 0$. Let $\tau$ be a symplectic trivialization of $\gamma^*\xi$ which extends to a global trivialization of the contact structure $\xi$ on $Y$. Since $u$ is a fast plane, we have $\operatorname{wind}^\tau(e_u) \leq 1$ by definition. Since the winding number is non-decreasing on the spectrum $\sigma(A_\gamma)$ and $\nu \leq \nu_u$, we obtain that $\operatorname{wind}^\tau(e)\leq 1$. Recall that $\operatorname{sl}(\gamma) = -1$ because $\gamma$ bounds an embedded symplectic disk. The self-linking number of $\gamma$ can be computed as the algebraic count of intersections between $u$ and a pushoff via a generic section of the normal bundle $N_u$ which, near infinity, is non-vanishing and constant with respect to the trivialization $\tau$. This implies that the algebraic count of intersections $u\cdot v$ between $u$ and $v$ is given by 
\begin{equation*}
u\cdot v = \operatorname{sl}(\gamma) + \operatorname{wind}^\tau(e) \leq -1 + 1 = 0.
\end{equation*}
Positivity of intersections then immediately implies that $u$ and $v$ are disjoint. Note that we could also have proved this assertion using more general intersection theoretic results for punctured pseudoholomorphic curves due to Hutchings \cite{hut02} and Siefring \cite{sie11}.

Consider now the evaluation map $\operatorname{ev}_{\hat{X}}$. The facts that all fast planes asymptotic to $\gamma$ are embedded and that any two distinct fast planes are disjoint implies that $\operatorname{ev}_{\hat{X}}$ is injective. By Propositions \ref{prop:fast_planes_properness} and \ref{prop:fast_planes_automatic_regularity} it is also proper and a local diffeomorphism. Since we have already established that $\operatorname{deg}(\operatorname{ev}_{\hat{X}})>0$, we can deduce that $\operatorname{ev}_{\hat{X}}$ is a diffeomorphism, implying that $\operatorname{deg}(\operatorname{ev}) = 1$. It remains to argue that any fast plane in the symplectization is disjoint from the trivial cylinder $\R\times \gamma$. Assume by contradiction that $u$ is a fast plane intersecting $\R\times \gamma$. Since $u$ is asymptotic to $\gamma$, this would imply that $u$ intersects the translated plane $u+c$ for all $c\gg 0$ sufficiently large. However, this is impossible because $u$ and $u+c$ are distinct and thus disjoint by the above discussion.
\end{proof}

Let us now turn to the proof of Theorem \ref{thm:degree_characterizes_hopf}. If $\operatorname{deg}(\operatorname{ev})>0$, then there exists an immersed fast plane asymptotic to $\gamma$. Starting from such a plane, it is not hard to construct a positively immersed symplectic disk with boundary $\gamma$. Conversely, if $\gamma$ bounds a positively immersed symplectic disk, then $\operatorname{deg}(\operatorname{ev})>0$ by Claim \ref{claim:degree_characterization_hopf_pisd_implies_deg_g_0}. This shows the first assertion of Theorem \ref{thm:degree_characterizes_hopf}.

Next, we verify the equivalence of statements \eqref{item:degree_characterizes_hopf_deg_1} - \eqref{item:degree_characterizes_hopf_dgss}. Suppose that $\operatorname{deg}(\operatorname{ev})= 1$. Then $\gamma$ bounds a positively immersed symplectic disk by the first assertion of Theorem \ref{thm:degree_characterizes_hopf}. By Claim \ref{claim:degree_characterization_hopf_non_embedded_pisd_implies_deg_g_1}, this disk must be embedded. Conversely, it $\gamma$ bounds an embedded symplectic disk, then $\operatorname{deg}(\operatorname{ev})=1$ by Claim \ref{claim:degree_characterization_hopf_embedded_disk_implies_deg_1}. This shows the equivalence of statements \eqref{item:degree_characterizes_hopf_deg_1} and \eqref{item:degree_characterizes_hopf_embedded_disk}.

Assume that $\gamma$ bounds an embedded symplectic disk. Then Claim \ref{claim:degree_characterization_hopf_embedded_disk_implies_deg_1} implies that the evaluation map $\operatorname{ev}_{\R\times Y}$ is a diffeomorphism onto $\R\times (Y\setminus \gamma)$. In this situation, it is well-known that the projection of each fast plane $u$ to $Y$ is an embedded disk-like global surface of section with boundary $\gamma$ (see \cite{hwz98, hry12}). The existence of a disk-like global surface of section with boundary $\gamma$ clearly implies that $\gamma$ is Hopf. Finally, every Hopf orbit must bound an embedded symplectic disk because every fiber of the Hopf fibration does. This shows the equivalence of statements \eqref{item:degree_characterizes_hopf_embedded_disk}, \eqref{item:degree_characterizes_hopf_hopf_orbit} and \eqref{item:degree_characterizes_hopf_dgss}.
\end{proof}

\section{Floer theory and Clarke duality}
\label{floclasec}

\subsection{The Floer complex of the direct action functional}
\label{flosubsec}

We recall the definition of the Floer complex which is associated to a suitable time-periodic Hamiltonian $H$ and time-periodic almost complex structure $J$ on $(\R^{2n},\omega_0)$.  Given a Hamiltonian $H\in C^{\infty}(\T\times \R^{2n})$, let $X_H$ be the time-periodic Hamiltonian vector field defined by
\[
\imath_{X_H} \omega_0 = - dH,
\]
where $d$ denotes differentiation with respect to the spatial variables. We assume that $H$ satisfies the following conditions:

\paragraph{\sc (Linear growth of the Hamiltonian vector field)} There exists $c>0$ such that $|X_H(t,z)|\leq c(1+|z|)$ for every $(t,z)\in \T\times \R^{2n}$.

\paragraph{\sc (Non-resonance at infinity)} There exist positive numbers $\epsilon>0$ and $r>0$ such that, for every smooth curve $\gamma: \T \rightarrow \R^{2n}$ satisfying
\[
\|\dot{\gamma} - X_H(\cdot,\gamma)\|_{L^2(\T)} < \epsilon,
\]
the inequality $\|\gamma\|_{L^2(\T)} \leq r$ holds.

\paragraph{\sc (Non-degeneracy)} All 1-periodic orbits $\gamma$ of $X_H$ are nondegenerate, meaning that 1 is not an eigenvalue of the differential of the time-1 flow at $\gamma(0)$.

\begin{rem}
\label{exham}
{\rm Here is a typical class of Hamiltonians which satisfy the first two conditions. Let $X\subset \R^{2n}$ be a domain which is uniformly starshaped with respect to the origin. Denote by $H_X$ the unique positively 2-homogeneous function taking the value 1 on $\partial X$ and assume that $H\in C^{\infty}(\T\times \R^{2n})$ satisfies 
\[
H(t,z) = \eta H_X(z) + \sigma, \qquad \forall t\in \T, \forall z \in \R^{2n} \mbox{ with } |z|\geq R,
\]
where $\eta>0$, $\sigma\in \R$ and $R>0$. Then the Hamltonian vector field of $H$ has linear growth. Moreover, $H$ is non-resonant at infinity provided that $\eta$ is different from all the actions of closed characteristics on $\partial X$. See \cite[Lemma 4.1]{ak22}. Another class of Hamiltonians satisfying these two conditions is described in \cite[Remark 4.2]{ak22}.}
\end{rem}

By the above assumptions, the set $\mathcal{P}(H)$ of 1-periodic orbits of $X_H$ is finite. Its elements are precisely the critical points of the direct action functional
\begin{equation}
	\label{direct_action}
	\mathcal{A}_H(\gamma) := \int_{\T} \gamma^* \lambda_0 - \int_{\T} H(t,\gamma(t))\, dt, \qquad \gamma\in C^{\infty}(\T,\R^{2n}),
\end{equation}
where $\lambda_0$ is any primitive of $\omega_0$.

Let $J=J(t,z)$ be a smooth $\omega_0$-compatible time-periodic almost complex structure on $\R^{2n}$ which is globally bounded as map into the space of endomorphisms of $\R^{2n}$. 

Let $\gamma\in \mathcal{P}(H)$. The symplectic form $\omega_0$ and the almost complex structure $J$ induce a Hermitian structure on the vector bundle $\gamma^* (T \R^{2n})$, which carries the asymptotic operator $A_{H,\gamma} := - J \nabla_t^{X_H}$, where $\nabla_t^{X_H}$ is the symplectic connection which is obtained by linearizing the flow of $X_H$ along $\gamma$. In terms of the natural trivialization of $T \R^{2n}$, we have
\[
A_{H,\gamma} = - J(\cdot,\gamma(\cdot)) \frac{d}{dt} - \nabla_J^2 H(\cdot,\gamma(\cdot)),
\]
where $\nabla_J^2 H$ denotes the spatial Hessian of $H$ with respect to the time-dependent Riemannian structure induced by $J$.  The asymptotic operator $A_{H,\gamma}$ is nondegenerate and the corresponding path in the linear symplectic group is given by the differential of the flow of $X_H$ along $\gamma$. We denote by $\mathrm{CZ}(\gamma)$ the Conley-Zehnder index of this path. The \textit{orientation line of $\gamma$} $\mathfrak{o}_{\gamma}$ is defined to be the orientation line of a Cauchy-Riemann type operator on $\C$ which is positively asymptotic to $A_{H,\gamma}$. This is well defined because the space of these Cauchy-Riemann type operators is contractible. The choice of a different $J$ produces a canonically isomorphic orientation line $\mathfrak{o}_{\gamma}$.

For every pair of 1-periodic orbits $\gamma_+$, $\gamma_-$ in $\mathcal{P}(H)$, consider the space of Floer cylinders 
\begin{equation}
\label{modfloer}
\begin{split}
\mathcal{M}_{\partial^F}^{\#}(\gamma_+,\gamma_-) := \bigl\{ u\in C^{\infty}(\R\times \T , \R^{2n}) \mid  & \; \partial_s u + J(t,u)(\partial_t u - X_H(t,u)) = 0 , \\ & \lim_{s\rightarrow \pm \infty} u(s,\cdot) = \gamma_{\pm} \bigr\}.
\end{split}
\end{equation}
The energy identity for $u\in \mathcal{M}_{\partial^F}^{\#}(\gamma_+,\gamma_-)$ reads
\[
E(u):= \int_{\R\times \T} |\partial_s u|_{J_t}^2\, ds \, dt = \mathcal{A}_H(\gamma_+) - \mathcal{A}_H(\gamma_-),
\]
where $|\cdot|_{J_t}^2 = \omega_0(\cdot,J_t\cdot)$. Together with the above assumptions of linear growth of $X_H$ and non-resonance at infinity, the uniform upper bound on the energy implies that the cylinders in $\mathcal{M}_{\partial^F}^{\#}(\gamma_+,\gamma_-)$ take values in a compact subset of $\R^{2n}$, see \cite[Proposition 3.1]{ak22}. Then a standard argument involving elliptic estimates and the fact that $\omega_0$ is exact implies that each space $\mathcal{M}_{\partial^F}^{\#}(\gamma_+,\gamma_-)$ is pre-compact in the $C^{\infty}_{\mathrm{loc}}$ topology.

The linearization of the Floer equation at a solution $u\in \mathcal{M}_{\partial^F}^{\#}(\gamma_+,\gamma_-)$ is a real Cauchy-Riemann type operator $D_u$ on the Hermitian vector bundle 
\[
u^*(T\R^{2n})\rightarrow \R\times \T
\]
which is asymptotic to the operators $A_{H,\gamma_{\pm}}$ for $s\rightarrow \pm \infty$. The operator $D_u$ is Fredholm of index
\[
\ind D_u = \mathrm{CZ}(\gamma_+) - \mathrm{CZ}(\gamma_-).
\]
If $J$ is generic, then all the spaces $\mathcal{M}_{\partial^F}^{\#}(\gamma_+,\gamma_-)$ are cut transversely and are hence smooth manifolds of dimension $\mathrm{CZ}(\gamma_+) - \mathrm{CZ}(\gamma_-)$. We fix such a generic $J$.

Now consider the case of index difference $\mathrm{CZ}(\gamma_+)- \mathrm{CZ}(\gamma_-)=1$. The quotient 
\[
\mathcal{M}_{\partial^F}(\gamma_+,\gamma_-) := \mathcal{M}_{\partial^F}^{\#}(\gamma_+,\gamma_-)/\R
\]
by the free $\R$-action $(\sigma,u)\mapsto u(\sigma + \cdot, \cdot)$ is discrete and compact, and hence a finite set. The algebraic count of $\MM_{\partial^F}(\gamma_+,\gamma_-)$ naturally takes value in $\mathfrak{o}_{\gamma_+}^*\otimes \mathfrak{o}_{\gamma_-}$. Indeed, let $D_{\gamma_-}$ be a real Cauchy-Riemann type operator on $\C$ which is positively asymptotic to $A_{H,\gamma_-}$ and consider $u\in \MM_{\partial^F}^\#(\gamma_+,\gamma_-)$. Then gluing the linearization $D_u$ with $D_{\gamma_-}$ produces a Cauchy-Riemann type operator on $\C$ which is positively asymptotic to $A_{H,\gamma_+}$. Hence the Floer-Hofer kernel gluing operation yields a canonical isomorphism
\begin{equation*}
\mathfrak{o}_{\gamma_+}\cong \mathfrak{o}_{D_u}\otimes \mathfrak{o}_{\gamma_-}.
\end{equation*}
In other words, the orientation local system of $\MM_{\partial^F}^\#(\gamma_+,\gamma_-)$ is naturally isomorphic to $\mathfrak{o}_{\gamma_+}\otimes\mathfrak{o}_{\gamma_-}^*$. Since $\MM_{\partial^F}(\gamma_+,\gamma_-)$ is obtained by taking the quotient by the above $\R$-action, its orientation local system is naturally isomorphic to $\mathfrak{o}_{\gamma_+}\otimes \mathfrak{o}_{\gamma_-}^*\otimes \mathfrak{o}_\R^* \cong \mathfrak{o}_{\gamma_+}\otimes \mathfrak{o}_{\gamma_-}^*$. Here we identify $\mathfrak{o}_\R\cong \Z$ using the obvious orientation of $\R$. A choice of generator of $\mathfrak{o}_{\gamma_+}\otimes \mathfrak{o}_{\gamma_-}^*$ thus induces on orientation of $\MM_{\partial^F}(\gamma_+,\gamma_-)$, which in turn gives rise to a $\Z$-valued count. This precisely describes an element of the dual line $(\mathfrak{o}_{\gamma_+}\otimes \mathfrak{o}_{\gamma_-}^*)^* \cong \mathfrak{o}_{\gamma_+}^*\otimes \mathfrak{o}_{\gamma_-}$, i.e. yields a canonical count
\begin{equation*}
\#\MM_{\partial^F}(\gamma_+,\gamma_-) \in \mathfrak{o}_{\gamma_+}^*\otimes \mathfrak{o}_{\gamma_-} \cong \operatorname{Hom}_\Z(\mathfrak{o}_{\gamma_+},\mathfrak{o}_{\gamma_-}).
\end{equation*}
Note that the absolute value $|\#\MM_{\partial^F}(\gamma_+,\gamma_-)|\in \Z_{\geq 0}$ is well-defined and independent of choices.

The \textit{Floer complex} $\partial^F : F_*(H) \rightarrow F_{*-1}(H)$ of $(H,J)$ is defined as
\begin{equation}
\label{floer_complex}
	F_k(H):= \bigoplus_{\substack{\gamma\in \mathcal{P}(H)\\ \mathrm{CZ}(\gamma)=k}} \mathfrak{o}_{\gamma}, \qquad \partial^F := \bigoplus_{\substack{(\gamma_+,\gamma_-)\in \mathcal{P}(H)^2\\ \mathrm{CZ}(\gamma_+) - \mathrm{CZ}(\gamma_-) = 1}} \#\mathcal{M}_{\partial^F}(\gamma_+,\gamma_-).
\end{equation}
Indeed, a standard argument proves that $\partial^F$ is a boundary operator. The resulting object is a $\Z$-graded chain complex of free Abelian groups $(F_k(H),\partial^F)$ which is filtered by the action $\mathcal{A}_H$. Changing the generic almost complex structure $J$ produces an isomorphic $\Z$-graded and filtered chain complex. 

\subsection{The Morse complex of the dual action functional}

We summarize some facts about Clarke's dual action functional and its Morse complex. We refer to \cite[Sections 5 and 6]{ak22} for the proofs. Note that we are adopting different sign conventions and are making some other minor changes.

We now assume that the Hamiltonian $H\in C^{\infty}(\T\times \R^{2n})$ is nondegenerate, non-resonant at infinity and satisfies the following condition:

\paragraph{\sc (Quadratic convexity)} There are positive numbers $\underline{h},\overline{h}$ such that $\underline{h} I \leq \nabla^2 H(t,z) \leq \overline{h} I$ for every $(t,z)\in \T\times \R^{2n}$.

\medskip

This condition implies the linear growth at infinity of the Hamiltonian vector field. A typical asymptotic behaviour of quadratically convex Hamiltonians is the one described in Remark \ref{exham}, where now the domain $X$ is assumed to be uniformly convex, meaning that its boundary has strictly positive sectional curvature. 

Denote by
\[
H^1_0(\T,\R^{2n}) := \Bigl\{ \gamma\in H^1(\T,\R^{2n}) \mid \int_{\T} \gamma(t)\, dt = 0 \Bigr\}
\]
the Hilbert space of loops in $\R^{2n}$ of Sobolev regularity $H^1$ with vanishing mean and by
\[
\mathbb{P}: H^1(\T,\R^{2n}) \rightarrow H^1_0(\T,\R^{2n}), \qquad \gamma\mapsto \gamma -  \int_{\T} \gamma(t)\, dt,
\]
the linear projection onto it along the space of constant loops. The Hilbert space $H^1_0(\T,\R^{2n})$ is the domain of Clarke's dual action functional, which takes the form
\[
\Psi_{H^*}(\gamma) :=- \int_{\T} \gamma^* \lambda_0 + \int_{\T} H^*(t,- J_0 \gamma'(t))\, dt, \qquad \gamma\in H^1_0(\T,\R^{2n}),
\]
where $H^*\in C^{\infty}(\T\times \R^{2n})$ is the Fenchel conjugate of $H$:
\[
H^*(t,z) := \max_{w\in \R^{2n}} \bigl( w\cdot z - H(t,w) \bigr).
\]
There is a one-to-one correspondence between critical points of the dual action functional and critical points of the direct action functional, i.e.\ 1-periodic orbits of $X_H$: if $\gamma\in C^{\infty}(\T,\R^{2n})$ is a critical point of $\mathcal{A}_H$ then $\mathbb{P}(\gamma)$ is a critical point of $\Psi_{H^*}$ and, conversely, if $\gamma\in H^1_0(\T,\R^{2n})$ is a critical point of $\Psi_{H^*}$ then there exists a unique vector $z\in \R^{2n}$ such that $\gamma+z$ is a critical point of $\mathcal{A}_H$. The values of the two functionals at critical points coincide:
\begin{equation}
\label{cfr0}
\Psi_{H^*}(\mathbb{P} (\gamma)) = \mathcal{A}_H(\gamma) \qquad \forall \gamma\in \mathcal{P}(H).
\end{equation}
One advantage of the dual action functional $\Psi_{H^*}$ is that its critical points have finite Morse index $i_{\mathrm{M}}$: more precisely, we have
\begin{equation}
\label{finind}
i_{\mathrm{M}}(\mathbb{P}(\gamma),\Psi_{H^*}) = \mathrm{CZ}(\gamma) - n, \qquad \forall \gamma\in \mathcal{P}(H).
\end{equation}
The functional $\Psi_{H^*}$ is continuously differentiable on $H^1_0(\T,\R^{2n})$ but in general only twice G\^ateaux differentiable. In order to associate a Morse complex to it, it is useful to restrict it to a suitable finite dimensional submanifold of $H^1_0(\T,\R^{2n})$ which contains all its critical points and has the property that the restriction of $\Psi_{H^*}$ to it is smooth. Given a natural number $N$, we consider the orthogonal splitting
\[
\begin{split}
H^1_0(\T,\R^{2n}) = \mathbb{H}_N \oplus \mathbb{H}^N, \qquad \mathbb{H}_N &:= \{ \gamma \mid \widehat{\gamma}(k) = 0 \mbox{ if } k <0 \mbox{ or } k> N\}, \\
\mathbb{H}^N &:= \{ \gamma \mid \widehat{\gamma}(k) = 0 \mbox{ if } 1\leq k \leq N\},
\end{split}
\]
where $\widehat{\gamma}(k)\in \R^{2n}$ are the Fourier coefficients of $\gamma$ in the representation
\[
\gamma(t) = \sum_{k\in \Z} e^{2\pi k t J_0} \widehat{\gamma}(k).
\]
Denote by $\mathbb{P}_N$ the projector onto the $2nN$-dimensional subspace $\mathbb{H}_N$ along $\mathbb{H}^N$.

If $N$ is large enough, then for every $x\in  \mathbb{H}_N$ there exists $Y_N(x)\in \mathbb{H}^N$ such that the restriction of $\Psi_{H^*}$ to the affine subspace $x + \mathbb{H}^N$ has a unique critical point $x+Y_N(x)$, which is a global minimizer. The map $Y_N:  \mathbb{H}_N \rightarrow  \mathbb{H}^N$ is smooth and takes values in the space of smooth loops. We denote by $M\subset H^1_0(\T,\R^{2n})$ the graph of $Y_N$ and by $\psi_{H^*}$ the restriction of $\Psi_{H^*}$ to $M$:
\[
M:= \{ x + Y_N(x) \mid x\in \mathbb{H}_N\}, \qquad \psi_{H^*}:= \Psi_{H^*}|_M.
\]
We have
\begin{equation}
\label{cfr1}
\psi_{H^*} (q) \leq \Psi_{H^*}(q+y) \qquad \forall q\in M,\; \forall y\in \mathbb{H}^N.
\end{equation} 
The function $\psi_{H^*}$ on the finite dimensional manifold $M$ is smooth, Morse and satisfies the Palais-Smale condition, once $M$ is endowed with the Riemannian structure induced by the ambient Hilbert space $H^1_0(\T,\R^{2n})$. The critical points of $\psi_{H^*}$ coincide with the critical points of $\Psi_{H^*}$ and have the same Morse index. By the above considerations, we then have
\begin{equation}
\label{indici}
i_{\mathrm{M}}(\mathbb{P}(\gamma)) :=  i_{\mathrm{M}}(\mathbb{P}(\gamma),\psi_{H^*}) = \mathrm{CZ}(\gamma) - n, \qquad \forall \gamma\in \mathcal{P}(H).
\end{equation}
By the non-degeneracy and non-resonance assumption, $\psi_{H^*}$ has finitely many critical points. Given a Riemannian metric $g$ on $M$ which is uniformly equivalent to the one induced by the ambient Hilbert space, we denote by $\phi^t_{-\nabla \psi_{H^*}}$ the negative gradient flow of $\psi_{H^*}$ with respect to $g$ and by
\[
\begin{split}
W^u(p) &:= \{ q\in M \mid \phi^t_{-\nabla \psi_{H^*}}(q) \rightarrow p \mbox{ for } t\rightarrow -\infty\}, \\
W^s(p) &:= \{ q\in  M \mid \phi^t_{-\nabla \psi_{H^*}}(q) \rightarrow p \mbox{ for } t\rightarrow +\infty\},
\end{split}
\]
the unstable and stable manifold of the critical point $p$ of $\psi_{H^*}$ with respect to this flow. The unstable manifold of $p$ has dimension $i_{\mathrm{M}}(p)$. 

When the metric $g$ is generic, unstable and stable manifolds of each pair of critical points intersect transversally and $\psi_{H^*}$ has a well-defined Morse complex $\{M_*(\psi_{H^*}),\partial^M\}$ over $\Z$, whose definition we briefly recall.

The orientation line of a critical point $p$ of $\psi_{H^*}$ is defined to be
\[
\mathfrak{o}_p := \mathfrak{o}_{T_p W^u(p)},
\]
and we consider the free Abelian group
\[
M_k(\psi_{H^*}) := \bigoplus_{\substack{p\in \mathrm{crit}(\psi_{H^*})\\ i_{\mathrm{M}}(p)=k}} \mathfrak{o}_p.
\]
Since the unstable manifold of $p$ is orientable and connected,  we obtain canonical isomorphisms
\begin{equation}
\label{ide1}
\mathfrak{o}_{T_x W^u(p)} \cong \mathfrak{o}_p \qquad \forall x\in W^u(p).
\end{equation}
Since $T_p M = T_p W^u(p) \oplus T_p W^s(p)$, we also have canonical isomorphisms
\begin{equation}
\label{ide2}
\mathfrak{o}_{\frac{T_x M}{T_x W^s(p)}} \cong \mathfrak{o}_{\frac{T_p M}{T_p W^s(p)}} \cong \mathfrak{o}_{T_{p} W^u(p)} = \mathfrak{o}_p \qquad \forall x\in W^s(p).
\end{equation}
Let $x\in W^u(p_+) \cap W^s(p_-)$. From the Morse-Smale transversality condition $T_x W^u(p_+) + T_xW^s(p_-)= T_x M$, we obtain the exact sequence
\[
0 \rightarrow T_x \bigl( W^u(p_+) \cap  W^s(p_-) \bigr) \rightarrow T_x W^u(p_+)  \rightarrow \frac{T_x M}{T_x W^s(p_-)} \rightarrow 0,
\]
where the first map is the inclusion and the second one the restriction of the quotient projection. Together with \eqref{ide1} and \eqref{ide2}, the above exact sequence induces the canonical isomorphism
\begin{equation}
\label{canMorse}
\mathfrak{o}_{p_+} \cong \mathfrak{o}_{T_x ( W^u(p_+) \cap  W^s(p_-))} \otimes \mathfrak{o}_{p_-}.
\end{equation}
This means that the orientation local system of $W^u(p_+)\cap W^s(p_-)$ is canonically isomorphic to $\mathfrak{o}_{p_+}\otimes \mathfrak{o}_{p_-}^*$. The gradient flow of $\nabla\psi_{H^*}$ induces an $\R$-action on $W^u(p_+)\cap W^s(p_-)$, which is free whenever $p_+\neq p_-$. In this case, the orientation local system of the quotient
\begin{equation*}
\MM_{\partial^M}(p_+,p_-) \coloneqq W^u(p_+)\cap W^s(p_-)/\R
\end{equation*}
is canonically isomorphic to $\mathfrak{o}_{p_+}\otimes \mathfrak{o}_{p_-}^*\otimes \mathfrak{o}_\R^* \cong \mathfrak{o}_{p_+}\otimes \mathfrak{o}_{p_-}^*$, where we again identify $\mathfrak{o}_\R\cong \Z$ using the standard orientation of $\R$.

If $i_{\mathrm{M}}(p_+) - i_{\mathrm{M}}(p_-) = 1$, then $\MM_{\partial^M}(p_+,p_-)$ is $0$-dimensional and we have a canonical count (cf. Section \ref{flosubsec})
\begin{equation*}
\#\MM_{\partial^M}(p_+,p_-) \in \mathfrak{o}_{p_+}^* \otimes \mathfrak{o}_{p_-} \cong \operatorname{Hom}_\Z(\mathfrak{o}_{p_+},\mathfrak{o}_{p_-}).
\end{equation*}
As in Section \ref{flosubsec}, the absolute value $|\#\MM_{\partial^M}(p_+,p_-)|\in \Z_{\geq 0}$ does not depend on choices. The homomorphism $\partial^M : M_*(\psi_{H^*}) \rightarrow M_{*-1}(\psi_{H^*})$ is defined by the formula
\[
\partial^M := \bigoplus_{\substack{(p_+,p_-) \in \mathrm{crit}(\psi_{H^*})^2 \\  i_{\mathrm{M}}(p_+) -    i_{\mathrm{M}}(p_-)   = 1 }} \;\; \#\mathcal{M}_{\partial^M}(p_+,p_-).
\]
A standard argument proves that $\partial^M$ is a boundary operator and that the $k$-th homology group of the Morse complex $\{M_*(\psi_{H^*}),\partial^M\}$ is isomorphic to the singular homology group $H_k(\mathbb{H}_N,\{\psi_{H^*}<c\})$, where $c$ is a real number which is smaller than the smallest critical value of $\psi_{H^*}$. It is clear from the construction that we get an isomorphic filtered chain complex with a different choice of large $N\in\N$, which corresponds to the isomorphism induced by  $(\mathbb{H}_N,\{\psi_{H^*}^N<c\})\subset (\mathbb{H}_{N'},\{\psi_{H^*}^{N'}<c\})$ for $N<N'$ on the singular homology side, where we added the superscript $N$ to $\psi_{H^*}^N$ for distinction.

\begin{rem}
\rm{Note that we are defining the Morse complex using the negative gradient flow of $\psi_{H^*}$, but we are orienting the flow lines by the positive gradient. The reason is that this choice simplifies the form of the isomorphism with Floer homology, the reason being that the Floer equation appearing in \eqref{modfloer} is a positive gradient equation for the action functional $\mathcal{A}_H$.}
\end{rem}

\subsection{An isomorphism}

Assume that the Hamiltonian $H\in C^{\infty}(\T\times \R^{2n})$ is nondegenerate, non-resonant at infinity and quadratically convex. Fixing a generic time-periodic almost complex structure $J$ on $\R^{2n}$ and a generic Riemann metric $g$ on $M$ as above, both the Floer complex of $(H,J)$ and the Morse complex of $(\psi_{H^*},g)$ are well defined. We then have the following result.
 
\begin{thm}
\label{isom}
There is a chain complex isomorphism
\[
\Theta_* : M_*(\psi_{H^*}) \rightarrow F_{*+n}(H)
\]
preserving the action filtrations.
\end{thm}

The above theorem is proven in \cite{ak22} in the case of $\Z/2$ coefficients. Its upgrade to integer coefficients is proven in Appendix \ref{appB} below, after some preliminaries about orientations of certain Fredholm operators which are discussed in Appendix \ref{appA}.

\begin{rem}
{\rm A consequence of the above theorem is that if $X\subset \R^{2n}$ is a convex body then the capacity from symplectic homology with integer coefficients of $X$ coincides with the minimal action of closed generalized characteristics on $\partial X$. This is proven in \cite[Section 11]{ak22} for the symplectic homology with $\Z/2$ coefficients. The same proof, now building on Theorem \ref{isom}, gives us the desired result.}
\end{rem}

\section{Floer cylinders descending from the systole}
\label{conv-count-sec}

Let $X\subset \R^{2n}$ be a uniformly convex domain. Throughout this section, we assume that $\partial X$ has a unique closed characteristic $\gamma$ of minimal action and that $\gamma$ is nondegenerate. 

Up to a translation, we may assume that the interior of $X$ contains the origin and we denote by $H_X: \R^{2n} \rightarrow \R$ the positively 2-homogeneous function taking the value $1$ on $\partial X$. Since the unique closed characteristic of minimal action is nondegenerate, we can find a number $\eta> \mathcal{A}(\gamma)$ such that no closed characteristics of $\partial X$ have action in the interval $(\mathcal{A}(\gamma),\eta]$. We consider a smooth function $k:[0,+\infty) \rightarrow \R$ such that
\[
k|_{[0,\frac{1}{4}]} = 0, \quad k''\geq 0,     \quad k'(1) = \mathcal{A}(\gamma), \quad k''(1)>0, \quad k'|_{[2,+\infty)} = \eta.
\]
The smooth function $k\circ H_X$ is quadratically convex away from the rescaled domain $\frac{1}{2} X$, on which it vanishes identically. We can modify $k\circ H_X$ in a small neighborhood of $\frac{1}{2} X$ and obtain a smooth function $\tilde{K}$ which is quadratically convex on $\R^{2n}$ and achieves its minimum at the origin, with $\tilde{K}(0)<0$ and Hessian at $0$ with small positive eigenvalues. The fixed point set of the time-$1$ map of the Hamiltonian flow of $\tilde{K}$ consists of the point $z_0=0$, which defines a nondegenerate constant orbit with Conley-Zehnder index $\mathrm{CZ}(z_0)=n$, and the circle $\gamma$, which is a degenerate orbit with nullity 1. Let $h\in \Z$ be the Conley-Zehnder index of $\gamma$, which we recall is extended to degenerate orbits by lower semicontinuity. It follows from Ekeland's work \cite{eke90} that $h=n+1$, and we will recover this fact below.

By \cite[Proposition 2.2]{cfhw96}, we can perturb $\tilde{K}$ near $\gamma$ in a time-dependent way and obtain a smooth function $K: \T \times \R^{2n}\rightarrow \R$ such that:
\begin{enumerate}[(i)]
\item \label{item:floer_descending_from_systole_first} $K(t,z) = \eta H_X(z) + \sigma$ outside of a compact subset of $\T\times \R^{2n}$, where $\sigma\in \R$;
\item $K$ is nondegenerate, non-resonant at infinity (see Remark \ref{exham}), and quadratically convex;
\item $X_K$ has three 1-periodic orbits: the constant orbit $z_0=0$ and two non-constant orbits $\gamma_-$ and $\gamma_+$ near $\gamma$ ;
\item \label{item:floer_descending_from_systole_last} $\mathrm{CZ}(z_0)=n$, $\mathrm{CZ}(\gamma_-)=h$, and $\mathrm{CZ}(\gamma_+)=h+1$.
\end{enumerate}

By (ii), the Clarke dual functional $\Psi_{K^*}$ and its finite dimensional reduction $\psi_{K^*}$ are well-defined, and the latter function has a well-defined Morse complex.
By (iii) and (iv), the function $\psi_{K^*}$ has three critical points, which are given by the projections by $\mathbb{P}$ of $z_0$, $\gamma_-$ and $\gamma_+$ and have Morse index 0, $h-n$ and $h+1-n$, respectively, see \eqref{indici}.

Assumption (i) implies that $\psi_{K^*}$ is unbounded from below. For convenience of the reader, we repeat the argument in \cite[Remark 11.3]{ak22} here. A suitable reparametrization $\widehat{\gamma}: \R \rightarrow \R^{2n}$ of the closed characteristic $\gamma$ satisfies
\[
\widehat{\gamma}'(t) = \mathcal{A}(\gamma) X_{H_X}(\widehat{\gamma}(t)),
\]
and hence is a 1-periodic orbit of the Hamiltonian flow defined by
\[
\widehat{H}:= \mathcal{A}(\gamma) H_X.
\]
The fact that $\widehat{H}$ is 2-homogeneous implies that the critical point $\widehat{\gamma}$ of $\mathcal{A}_{\widehat{H}}$ has vanishing Hamiltonian action and hence
\begin{equation}
\label{zero}
\Psi_{\widehat{H}^*}(\mathbb{P}(\widehat{\gamma})) = \mathcal{A}_{\widehat{H}} (\widehat{\gamma})=0.
\end{equation}
By (i) we have
\[
K(t,z) \geq \eta H_X(z) - c = \theta \widehat{H}(z) - c \qquad \forall (t,z)\in \T\times \R^{2n},
\]
where $\theta:= \frac{\eta}{\mathcal{A}(\gamma)}>1$ and $c$ is suitably large. Then
\[
K^*(t,z) \leq \theta \widehat{H}^*\left(\frac{z}{\theta}\right) + c = \frac{1}{\theta}   \widehat{H}^* (z) + c \qquad \forall (t,z)\in \T\times \R^{2n},
\]
and hence for every $s>0$ we have
\[
\begin{split}
\Psi_{K^*}(s \mathbb{P}(\widehat{\gamma})) &= - \mathcal{A}(s \widehat{\gamma}) + \int_{\T} K^*(t,-s J_0 \widehat{\gamma}'(t)) \, dt \leq - s^2 \mathcal{A}(\gamma) + \frac{1}{\theta} \int_{\T}  \widehat{H}^* (-s J_0 \widehat{\gamma}'(t)) \, dt + c \\ &= \frac{s^2}{\theta} \Psi_{\widehat{H}^*}(\mathbb{P}(\widehat{\gamma})) + \left( \frac{1}{\theta} - 1 \right) \mathcal{A}(\gamma) s^2 + c =  \left( \frac{1}{\theta} - 1 \right) \mathcal{A}(\gamma) s^2 + c,
\end{split}
\]
where in the last equality we have used \eqref{zero}. Since $\theta>1$, the above inequality implies that $\Psi_{K^*}(s \mathbb{P}(\widehat{\gamma}))$ tends to $-\infty$ for $s\rightarrow +\infty$ and hence $\Psi_{K^*}$ is unbounded from below. By \eqref{cfr1}, the reduced function $\psi_{K^*}$ is unbounded from below as well, as claimed.

This implies that the relative homology group $H_0(M, \{\psi_{K^*}< a\})$ vanishes for every $a\in \R$. If $a$ is smaller than the minimum of $\psi_{K^*}$ over its critical set, then the Morse homology of $\psi_{K^*}$ is isomorphic to $H_0(M, \{\psi_{K^*}< a\})$. The fact that the latter group vanishes implies that the degree zero cycle defined by $\mathbb{P}(z_0)$ is a boundary. Therefore, the function $\psi_{K^*}$ must have a critical point of Morse index one and this is necessarily $\mathbb{P}(\gamma_-)$, because if $\mathbb{P}(\gamma_+)$ had index one, then $\mathbb{P}(\gamma_-)$ would be a second critical point of index zero and the Morse homology of $\psi_{K^*}$ in degree zero would not vanish. We deduce that $h=n+1$, as claimed above. The fact that the degree zero cycle defined by $\mathbb{P}(z_0)$ is a boundary implies that $|\#\MM_{\partial^M}(\mathbb{P}(\gamma_-),\mathbb{P}(z_0))|=1$. We also observe that, since the unstable manifold of $\mathbb{P}(\gamma_-)$ is one-dimensional, this implies that one of the two branches of this unstable manifold converges to $\mathbb{P}(z_0)$ while on the other branch $\psi_{K^*}$ tends to $-\infty$.

Let $J$ be a globally bounded time-periodic $\omega_0$-compatible almost complex structure on $\R^{2n}$ such that $(K,J)$ is regular, meaning that all the moduli spaces of Floer cylinders are cut transversally. Since $\mathrm{CZ}(\gamma_-)=n+1$, the moduli space $\mathcal{M}_{\partial^F}(\gamma_-,z_0)$ is $0$-dimensional and compact. From Theorem \ref{isom} and the fact that $|\#\MM_{\partial^M}(\mathbb{P}(\gamma_-),\mathbb{P}(z_0))|=1$, we conclude that $|\# \mathcal{M}_{\partial^F}(\gamma_-,z_0)|=1$. We summarize our findings into the following theorem.

\begin{thm}
\label{convex-count}
Let $K: \T \times \R^{2n} \rightarrow \R$ be a smooth Hamiltonian satisfying the above conditions (i)-(iv) and assume that the pair $(K,J)$ is regular. Then $h=n+1$ and $|\# \mathcal{M}_{\partial^F}(\gamma_-,z_0)|=1$.
\end{thm}

\begin{rem}
{\rm The above proof also shows that in the Morse complex of $\psi_{K^*}$ the boundary of the chain which is associated to the critical point $\mathbb{P}(\gamma_+)$ is zero. This fact can also be deduced from the analogous statement for the chain defined by $\gamma_+$ in the Floer complex of $(K,J)$, which is proven in the above mentioned \cite[Proposition 2.2]{cfhw96}. This leads to a less direct but shorter proof of the fact that the Clarke dual functional is unbounded from below when condition (i) holds. See \cite[Lemma 11.2]{ak22}.}
\end{rem}

\begin{rem}
{\rm Let us also sketch a third way of determining the Morse complex of $\psi_{K^*}$. Consider the finite dimensional reduction $\psi_{\tilde{K}^*}$ of the Clarke dual action functional of the autonomous Hamiltonian $\tilde{K}$. Recall that $K$ was obtained from $\tilde{K}$ by a small time dependent perturbation. The function $\psi_{\tilde{K}^*}$ is a Morse-Bott function on a manifold diffeomorphic to some finite dimensional vector space $V$ and satisfies the Palais-Smale condition. It has a local minimum corresponding to the constant orbit $z_0$ and a $\T$-family of critical points of some index $k\geq 1$ corresponding to $\gamma$. Hence $V$ is obtained from the sublevel set $\left\{ \psi_{\tilde{K}^*} < a \right\}$ for a sufficiently low value of $a$ by attaching a single $0$-handle and a $\T$-family of $k$-handles. If the sublevel set $\left\{ \psi_{\tilde{K}^*} <a \right\}$ were empty, then $V$ would be homotopy equivalent to the one-point compactification of $\T\times \R^k$, contradicting the fact that $V$ is contractible. Hence $\left\{ \psi_{\tilde{K}^*} <a \right\}$ is non-empty. But then the only way to turn the union of $\left\{ \psi_{\tilde{K}^*} <a \right\}$ and the $0$-handle corresponding to $z_0$ into a connected space is to attach a $\T$-family of $1$-handles connecting the two sets. The unstable manifold of the $\T$-family of index $1$ critical points is an infinite cylinder $\R\times \T$. On one side, this cylinder converges to the local minimum corresponding to $z_0$, and on the other side the value of $\psi_{\tilde{K}^*}$ tends to $-\infty$. The function $\psi_{K^*}$ is obtained from $\psi_{\tilde{K}^*}$ by a small perturbation splitting the $\T$-family of index $1$ critical points into two critical points of indices $1$ and $2$, respectively. Clearly, there is a unique gradient flow line connecting the index $1$ critical point and the local minimum.}
\end{rem}

\section{Pseudoholomorphic planes descending from the systole}
\label{ns-stat}

Throughout this section, let $X\subset \R^{2n}$ be a uniformly convex domain containing the origin in its interior. As in Section \ref{conv-count-sec}, we assume that the Reeb flow of $\lambda$ on $Y=\partial X$ has a unique periodic orbit $\gamma$ of minimal action and that $\gamma$ is nondegenerate. Recall that in this situation we must have $\operatorname{CZ}(\gamma) = n+1$.

Let $J$ be an admissible almost complex structure on the symplectic completion $\hat{X}$ and let $\MM_1(\hat{X},\gamma,J)$ denote the moduli space of unparametrized $J$-holomorphic planes with one marked point positively asymptotic to $\gamma$. The virtual dimension of this moduli space is $2n$. Since $\gamma$ is a simple orbit, planes asymptotic to $\gamma$ are automatically somewhere injective. Therefore, $\MM_1(\hat{X},\gamma,J)$ is a smooth $2n$-manifold for generic $J$. The marked point gives rise to an evaluation map
\begin{equation}
\label{eq:eval_map_high_dim}
\operatorname{ev}_{\hat{X}} : \MM_1(\hat{X},\gamma,J) \rightarrow \hat{X}.
\end{equation}
This evaluation map is proper. Indeed, since $\gamma$ has minimal action, the SFT limit of a sequence of planes asypmtotic to $\gamma$ can only have one single layer. It must be a plane asymptotic to $\gamma$ either in the symplectization of $Y$ or in the completion $\hat{X}$. Using this observation, properness of $\operatorname{ev}_{\hat{X}}$ follows from an argument analogous to the one given in the proof of Proposition \ref{prop:fast_planes_properness}.

As a consequence of the above discussion, we see that the evaluation map $\operatorname{ev}_{\hat{X}}$ has a well-defined degree for generic $J$. The sign of this degree depends on a choice of orientation of the moduli space $\MM_1(\hat{X},\gamma,J)$. Note that the orientation local system of $\MM_1(\hat{X},\gamma,J)$ is canonically isomorphic to the orientation line $\mathfrak{o}_\gamma$ of the Reeb orbit $\gamma$ (see Section \ref{subsec:fredholm_theory}). In order to avoid a choice of orientation, we can regard the degree of $\operatorname{ev}_{\hat{X}}$ as an element of $\mathfrak{o}_{\gamma}^*$, i.e.
\begin{equation*}
\operatorname{deg}(\operatorname{ev}_{\hat{X}}) \in \mathfrak{o}_\gamma^*.
\end{equation*}
Note that the absolute value $|\operatorname{deg}(\operatorname{ev}_{\hat{X}})|\in \Z_{\geq 0}$ does not depend on a choice of orientation. The degree of $\operatorname{ev}_{\hat{X}}$ is independent of the choice of generic $J$. This is a consequence of the fact that for a family of almost complex structures $(J_t)_{t\in [0,1]}$, the evaluation map
\begin{equation*}
\operatorname{ev}_{[0,1]\times \hat{X}} : \MM_1(\hat{X},\gamma,(J_t)_{t\in [0,1]}) \rightarrow [0,1]\times \hat{X} \quad (t,[u,z]) \mapsto (t,u(z))
\end{equation*}
is proper (cf. Propostion \ref{prop:degree_invariance}). Again, properness of $\operatorname{ev}_{[0,1]\times \hat{X}}$ relies on $\gamma$ being the Reeb orbit of minimal action.

Let us point out that, in the case $n=2$, we have $\operatorname{CZ}(\gamma)=3$, which implies that all pseudoholomorphic planes positively asymptotic to $\gamma$ are automatically fast. This means that in this case, the evaluation map \eqref{eq:eval_map_high_dim} is the same as the one discussed in Section \ref{sec:fast_planes}.

The goal of the present section is to prove the following result.

\begin{thm}
\label{count-planes}
We have $|\operatorname{deg}(\operatorname{ev}_{\hat{X}})| = 1$.
\end{thm}

The proof of Theorem \ref{count-planes} will occupy the rest of this section. We sketch an outline below.\\[-1ex]

\noindent{\sc Step 1.} The key idea of the proof is to transform Floer cylinders of the Hamiltonian $K$ we studied in Theorem \ref{convex-count} into pseudoholomorphic planes using the neck stretching technique. However, $K$ is not suitable for neck stretching, so we consider another Hamiltonian  $H$. The Hamiltonian vector field $X_H$ of $H$ has $\gamma_-$ and $z_0$ as nondegenerate 1-periodic orbits. Here, $z_0$ is the constant orbit mapping to the origin of $\R^{2n}$, and $\gamma_-$ is the orbit with $\mathrm{CZ}(\gamma_-)=n+1$ that splits off from $\gamma$ as in Section \ref{conv-count-sec}. The Hamiltonian $H$ is zero on $X\setminus U_{z_0}$, where $U_{z_0}$ is a neighborhood of $z_0$, so that Floer cylinders of $H$ are pseudoholomorphic on this region. In Proposition \ref{prop:floer_count_1}, we prove that the absolute value of the algebraic count of Floer cylinders of $H$ descending from $\gamma_-$ to $z_0$ is equal to $1$. This follows from Theorem \ref{convex-count} and a continuation argument. This is the content of Section \ref{neck_stretching_H}. \\[-1ex]

\noindent{\sc Step 2.} We take an irrational ellipsoid $E$ such that $\overline{U}_{z_0}\subset \mathrm{int}(E)$ and $E\subset \mathrm{int}(X)$. We show that if one stretches the neck along the hypersurfaces $F:=\partial E$ and $Y=\partial X$, any Floer cylinder of $H$ descending from $\gamma_-$ to $z_0$ as in Step 1 breaks into three pieces. The middle piece is a pseudoholomorphic cylinder in the symplectic completion $\hat{M}$ of the Liouville cobordism $M:=X\setminus\mathrm{int}(E)$ between $Y$ and $F$. This  cylinder is positively asymptotic to $\gamma$ and negatively asymptotic to the minimal action orbit $\gamma_E$ on $F$. A gluing theorem implies that the absolute value of the algebraic count of such pseudoholomorphic cylinders is equal to $1$. We introduce and investigate the relevant moduli spaces in Sections \ref{top}, \ref{middle}, and \ref{lowest}. The gluing result is discussed in Section \ref{neck-stretching1}.\\[-1ex]

\noindent{\sc Step 3.} Let $\hat{E}$ be the symplectic completion of $E$. For any  point $z\in\hat{E}$, the absolute value of the algebraic count of pseudoholomorphic planes in $\hat{E}$ passing through $z$ and asymptotic to $\gamma_E$ is equal to $1$. By gluing such a plane and a cylinder in $\hat{M}$ as in Step 2, we obtain a pseudoholomorphic plane in the symplectic completion $\hat{X}$ passing through $z$ and asymptotic to $\gamma$. A gluing theorem implies that the absolute value of the algebraic count of such planes is equal to $1$, which proves Theorem \ref{count-planes}. This is carried out in Section \ref{neck-stretching2}.

\subsection{Floer cylinders}
\label{neck_stretching_H}

Let $\eta>\MA(\gamma)=\MA_{\min}(X)$ be a constant such that $\gamma$ is the unique periodic orbit of $\lambda$ with action smaller than or equal to $\eta$. We fix a constant $a\in(1,2)$  sufficiently close to $1$ and consider a smooth function $h:[0,\infty)\to\R$ satisfying 
\begin{equation} 
\label{eq:h}
\begin{split}
&h=0 \mbox{ on a neighborhood of } [0,1], \\
h''\geq 0, &\quad h'(a)=\MA(\gamma),\quad h''(a)>0,\quad   h'|_{[2,+\infty)}=\eta.
\end{split}
\end{equation} 
Then the non-constant $1$-periodic orbits of the Hamiltonian flow of $h\circ H_X$ are precisely given by $\sqrt{a}\gamma(\MA(\gamma)(t+t_0))$ for $t_0 \in \T$, where as before $H_X$ is the positively 2-homogeneous function on $\R^{2n}$ taking the value 1 on $\partial X$. We modify $h\circ H_X$ in a small neighborhood $U_{z_0}$ of the origin $z_0$ by adding a non-positive $C^2$-small perturbation such that the resulting Hamiltonian $\tilde{H}\in C^\infty(\R^{2n})$ has the property that the origin $z_0$ is the unique minimum point of $\tilde{H}$, satisfies $\tilde{H}(z_0)<0$, and is nondegenerate as a 1-periodic orbit of $X_{\tilde{H}}$. Moreover, we can arrange the perturbation so that no new non-constant $1$-periodic orbits are created and $z_0$ is the only constant $1$-periodic orbit of $\tilde{H}$ outside of $\tilde{H}^{-1}(0)$.

We further perturb $\tilde{H}$ to $H\in C^\infty(\T\times \R^{2n})$ in a small neighborhood of $\sqrt{a}\gamma$ as in  Section \ref{conv-count-sec} so that $X_H$ has precisely two non-constant 1-periodic orbits $\gamma_-$ and $\gamma_+$, both of which are nondegenerate and moreover satisfy $\mathrm{CZ}(\gamma_-)=n+1$, $\mathrm{CZ}(\gamma_+)=n+2$, and $\MA_H(\gamma_-)<\MA_H(\gamma_+)$. Note that the 1-periodic orbits of $X_H$ are  $z_0$, $\gamma_-$, $\gamma_+$, and the constant orbits in the region $H^{-1}(0)$. The first three are nondegenerate and the action $\MA_H$ takes positive values on them, while the constant orbits in $H^{-1}(0)$ are degenerate and have vanishing action. Note that, by the convexity of $h$,
\[
\MA_{h\circ H_X}(\sqrt{a}\gamma)= a \mathcal{A}(\gamma) - h(a) > a \mathcal{A}(\gamma) - h'(a)(a-1) = \MA(\gamma),
\]
and that, by taking $a\in(1,2)$ close to $1$, we can make $\MA_{h\circ H_X}(\sqrt{a}\gamma)$ arbitrarily close to $\MA(\gamma)$, and in turn $\MA_H(\gamma_-)$  and $\MA_H(\gamma_+)$ also arbitrarily close to $\MA(\gamma)$. In particular, we can arrange that $\gamma$ is the only periodic Reeb orbit of $\lambda$ with action smaller than $\MA_H(\gamma_-)$.

Consider a compatible time-periodic almost complex structure $J$ on $\R^{2n}$ as in Section \ref{flosubsec}. As defined in \eqref{modfloer}, we let ${\mathcal{M}}_{\partial^F}^\#(\gamma_-,z_0;H,J)$ denote the moduli space of Floer cylinders of $H$ which are positively asymptotic to $\gamma_-$ and negatively asymptotic to $z_0$. Moreover, we let
\[
{\mathcal{M}}_{\partial^F}(\gamma_-,z_0;H,J):={\mathcal{M}}_{\partial^F}^\#(\gamma_-,z_0;H,J)/\R,
\]
denote the quotient space, where the $\R$-action is given by $(\sigma,u)\mapsto u(\sigma+\cdot,\cdot)$. The virtual dimension of this quotient space is given by $\operatorname{CZ}(\gamma_-)-\operatorname{CZ}(z_0)-1 = 0$. As in Section \ref{flosubsec}, we have a well-defined algebraic count $\#\MM_{\partial^F}(\gamma_-,z_0;H,J) \in \mathfrak{o}_{\gamma_-}^*\otimes \mathfrak{o}_{z_0}$ for generic $J$. The goal of this subsection is to show the following:

\begin{prop}
\label{prop:floer_count_1}
For generic $J$, we have $|\#{\mathcal{M}}_{\partial^F}(\gamma_-,z_0;H,J)|=1$.
\end{prop} 

\begin{proof}
Fix a generic $J$. Since $H$ has only finitely many $1$-periodic orbits of positive action, all of which are nondegenerate, we have a well-defined Floer complex $F_*^{(0,\infty)}(H)$ in the positive action window $(0,\infty)$. This complex is supported in degrees $n$, $n+1$ and $n+2$ and we have
\[
F_n^{(0,\infty)}(H)=\mathfrak{o}_{z_0},\quad F_{n+1}^{(0,\infty)}(H)=\mathfrak{o}_{\gamma_{-}},\quad F_{n+2}^{(0,\infty)}(H)=\mathfrak{o}_{\gamma_{+}}.
\]
The statement that $|\#{\mathcal{M}}_{\partial^F}(\gamma_-,z_0;H,J)|=1$ is equivalent to saying that $F_n^{(0,\infty)}(H)$ is contained in the image of the boundary operator.

Let us choose a Hamiltonian $K\in C^\infty(\T\times \R^{2n})$ for which properties \eqref{item:floer_descending_from_systole_first}-\eqref{item:floer_descending_from_systole_last} in Section \ref{conv-count-sec} hold and which in addition satisfies:
\begin{itemize}
\item $K\leq H$,
\item $K(z_0) = H(z_0)$.
\end{itemize}
The condition $K\leq H$ allows for the construction of a continuation chain homomorphism
\begin{equation*}
\mathfrak{c} : F_*^{(0,\infty)}(K)\rightarrow F_*^{(0,\infty)}(H)
\end{equation*}
perserving degree and action filtration, as we recall below. It follows from Theorem \ref{convex-count} that $F_n^{(0,\infty)}(K)=\mathfrak{o}_{z_0}$ is contained in the image of the Floer differential of the complex $F_*^{(0,\infty)}(K)$. Our goal is to show that $\mathfrak{c}$ restricts to an isomorphism
\begin{equation*}
\mathfrak{c}_n: F_n^{(0,\infty)}(K) \rightarrow F_n^{(0,\infty)}(H).
\end{equation*}
Since $\mathfrak{c}$ is a chain homomorphism, this would imply that $F_n^{(0,\infty)}(H)$ is contained in the image of the Floer differential of the complex $F_*^{(0,\infty)}(H)$ as desired.

The chain homomorphism $\mathfrak{c}$ is constructed as follows. We interpolate between $K$ and $H$ via a homotopy $G_s\in C^\infty(\T\times \R^{2n})$ smoothly parametrized by $s\in \R$ such that
\begin{itemize}
\item $\partial_sG_s \leq 0$,
\item $G_s=H$ for $s\leq 0$ and $G_s=K$ for $s\geq 0$,
\item $X_{G_s}$ has linear growth with constants uniform in $s$.
\end{itemize}
Moreover, we choose a uniformly bounded homotopy of almost complex structures $(J_s)_{s\in \R}$ starting and ending at $J$. For every pair of 1-periodic orbits $(x_+,x_-)\in\mathcal{P}(K)\times \mathcal{P}(H)$ of positive action, consider the space of continuation Floer cylinders
\[
\begin{split}
\mathcal{M}_{\mathfrak{c}}(x_+,x_-) := \bigl\{ u\in C^{\infty}(\R\times \T , \R^{2n}) \mid  & \; \partial_s u + J_s(t,u)(\partial_t u - X_{G_s}(t,u)) = 0 , \\ & \lim_{s\rightarrow \pm \infty} u(s,\cdot) = x_{\pm} \bigr\}.
\end{split}
\]
The condition $\partial_sG_s\leq 0$ yields the energy inequality 
\begin{equation}
\label{energy_ineq}
0 \leq E(u)=\int_{\R\times\T}|\partial_su|_{J_{s,t}}^2\,ds\,dt\leq \MA_K(x_+)-\MA_H(x_-)
\end{equation}
for every $u\in\mathcal{M}_{\mathfrak{c}}(x_+,x_-)$. The linear growth of $X_{G_s}$ and the non-resonance of $K$ and $H$ at infinity imply a uniform $C^0$-bound for cylinders in $\MM_{\mathfrak{c}}(x_+,x_-)$ via a similar argument as in \cite[Proposition 3.1]{ak22}. We conclude that the moduli spaces of continuation Floer cylinders are compact up to breaking into continuation Floer cylinders and ordinary Floer cylinders of $K$ and $H$. Since $x_+$ and $x_-$ are assumed to have positive action, it follows from the energy inequality \eqref{energy_ineq} that breaking has to occur at positive action orbits.

If the homotopy $J_s$ is chosen generically, the linearized operator $D_u$ at $u\in \mathcal{M}_{\mathfrak{c}}(x_+,x_-)$ is surjective, and hence $\MM_{\mathfrak{c}}(x_+,x_-)$ is a smooth manifold of dimension $\mathrm{CZ}(x_+)-\mathrm{CZ}(x_-)$. Assume that $\mathrm{CZ}(x_+)=\mathrm{CZ}(x_-)$. As in the definition of the Floer complex in Section \ref{flosubsec}, we have a canonical count
\begin{equation*}
\#\MM_\mathfrak{c}(x_+,x_-)\in \mathfrak{o}_{x_+}^*\otimes \mathfrak{o}_{x_-}\cong \operatorname{Hom}_\Z(\mathfrak{o}_{x_+},\mathfrak{o}_{x_-}).
\end{equation*}
We define
\[
\mathfrak{c}:F_*^{(0,\infty)}(K)\rightarrow F_*^{(0,\infty)}(H),
\qquad 
\mathfrak{c}:=\bigoplus_{(x_+,x_-)} \#\mathcal{M}_{\mathfrak{c}}(x_+,x_-),
\]
where the direct sum ranges over all pairs $(x_+,x_-)\in\mathcal{P}(K)\times \mathcal{P}(H)$ with positive action and satisfying $\mathrm{CZ}(x_+)=\mathrm{CZ}(x_-)$. By the above discussion, this is a well-defined chain map.

The constant cylinder mapping to $z_0$ belongs to $\MM_{\mathfrak{c}}(z_0,z_0)$. Since $\MA_H(z_0)=-H(z_0)=-K(z_0)=\MA_K(z_0)$, the constant cylinder is the only element of $\MM_{\mathfrak{c}}(z_0,z_0)$ by the energy inequality \eqref{energy_ineq}. This implies that $\mathfrak{c}_n$ is an isomorphism.
\end{proof}

\subsection{Hybrid Floer cylinders in the top level}
\label{top}
Consider the strong symplectic cobordism $N:=2X\setminus \operatorname{int}(X)$ between $2Y$ and $Y$ and let $\hat{N}$ denote its symplectic completion. We define a smooth Hamiltonian $H^N$ by 
\[
H^N:\T\times\hat{N}\to\R\qquad  H^N:=
\begin{cases}
0 & \text{on }N^-, \\
H & \text{on }\hat{N}\setminus N^-.
\end{cases}
\] 
Here, $H$ is the Hamiltonian constructed in Section \ref{neck_stretching_H} and we identify $\hat{N}$ with $\R^{2n}\setminus \left\{ 0 \right\}$ via the flow induced by the radial Liouville vector field on $\R^{2n}\setminus\left\{ 0 \right\}$. The 1-periodic orbits $\gamma_-$ and $\gamma_+$ of $X_H$ are contained in $N$ and are therefore also $1$-periodic orbits of $X_{H^N}$. In fact, they are the only non-constant $1$-periodic orbits of $X_{H^N}$.

We call a smooth $\T$-family $J=(J_t)_{t\in\T}$ of almost complex structures on $\hat{N}$ {\it admissible} if each $J_t$ is an admissible almost complex structure on $\hat{N}$ in the sense introduced in Section \ref{subsec:setup_pseudoholomorphic_curves} and the restriction of $J$ to the cylindrical ends $N^\pm$ of $\hat{N}$ is independent of $t$. Let $J$ be such an admissible family. We consider Floer cylinders 
\begin{equation}
\label{hybrid_cylinder}
u:\R\times \mathbb{T}\rightarrow \hat{N},\qquad \partial_su + J(t,u) (\partial_tu - X_{ H^N}(t,u)) = 0.
\end{equation}
Note that the intersection of such a Floer cylinder with the negative cylindrical end $N^-$ is pseudoholomorphic. We are interested in Floer cylinders which, as $s\rightarrow +\infty$, converge to the Hamiltonian periodic orbit $\gamma_-$ and, as $s\rightarrow -\infty$, are negatively asymptotic to the Reeb orbit $\gamma$. More percisely, this means that $u( (-\infty,-s_0)\times \T)\subset N^-$ for $s_0\gg 0$ sufficiently large, which allows us to write $u$ in components
\begin{equation*}
u|_{(-\infty,-s_0)\times \T} =(a,v):(-\infty,-s_0)\times \mathbb{T} \longrightarrow N^-=(-\infty,0]\times Y,
\end{equation*}
and that the following asymptotic conditions hold:
\begin{equation}
\label{hybrid_cylinder_asymp_top}
\begin{split}
&\lim_{s\rightarrow +\infty}u(s,t) =  \gamma_-(t),\\
&\lim_{s\rightarrow -\infty} v(s,t) =  \gamma(\mathcal{A}(\gamma)(t + t_0))  \;\mbox{ for some }\; t_0\in\T,\qquad 
\lim_{s\rightarrow -\infty} a(s,t)=-\infty.
\end{split}
\end{equation}
We define moduli spaces of such hybrid Floer cylinders by
\begin{align*}
\MM^\#(\hat{N},\gamma_-,\gamma, H^N,J) &\coloneqq
\left\{ u:\R\times \mathbb{T}\rightarrow \hat{N} \mid \text{$u$ satisfies \eqref{hybrid_cylinder} and \eqref{hybrid_cylinder_asymp_top}} \right\},\\[1ex]
\MM(\hat{N},\gamma_-,\gamma, H^N,J) &\coloneqq \MM^\#(\hat{N}, \gamma_-, \gamma, H^N,J)/\R,
\end{align*}	
where the $\R$-action is given by $(\sigma,u)\mapsto u(\sigma+\cdot,\cdot)$. Note that the image of any $u\in \MM^\#(\hat{N}, \gamma_-, \gamma, H^N,J)$ is contained in $\hat{N}\setminus N^+$ by a maximum principle (see e.g.\ \cite[Lemma 1.8]{vit99}).

\begin{lem}
\label{lem:compactness_mixed_partialX}
For every admissible $\T$-family of almost complex structures $J=(J_t)_t$ on $\hat{N}$, the moduli space $\MM(\hat{N},\gamma_-,\gamma,H^N,J)$ is compact.
\end{lem}

\begin{proof}
There are two types of degenerations of hybrid Floer cylinders that could potentially result in non-compactness of the moduli space: bubbling and breaking. Our strategy is to rule out both of them via energy considerations.

First, we introduce an appropriate notion of energy of hybrid Floer cylinders. Recall that $N^-=(-\infty,0]\times Y$ is attached to the collar neighborhood $[0,\epsilon)\times Y$ of the concave boundary of $N$. We can therefore regard $(-\infty,\epsilon)\times Y$ as a subset of $\hat{N}$. Let us fix a smooth function $\varphi:(-\infty,\epsilon)\to(1,\infty)$ such that $\varphi'> 0$ everywhere, $\varphi(s)=e^s$ for $s\in(\epsilon/2,\epsilon)$, and $\lim_{s\rightarrow -\infty} \varphi(s) = 1$. We define a smooth symplectic form $\omega_\varphi$ on $\hat{N}$ by 
\begin{equation*}
\omega_{\varphi}:=d\lambda_\varphi,\qquad \lambda_{\varphi}:=
\begin{cases}
\varphi\lambda & \textrm{ on } (-\infty,\epsilon)\times Y,\\
\lambda_0 & \textrm{ on } \hat{N}\setminus ((-\infty,\epsilon)\times Y).
\end{cases}
\end{equation*}
Here we recall that $\lambda_0$ is the standard radial Liouville $1$-form on $\R^{2n}$, that $\lambda$ is the restriction of $\lambda_0$ to $Y$, and that $\hat{N}$ is identified with $\R^{2n}\setminus \left\{ 0 \right\}$ via the flow of the radial Liouville vector field. We assume that $\epsilon>0$ is chosen sufficiently small so that $H^N=0$ on $(-\infty,\epsilon)\times Y$. Then the Hamiltonian vector field $X_{H^N}$ of $H^N$ with respect to the original symplectic form $d\lambda_0$ on $\hat{N}$ agrees with the Hamiltonian vector field of $H^N$ with respect to $\omega_\varphi$.

Let $J$ be an admissible $\T$-family of almost complex structures on $\hat{N}$. Note that $J$ is $\omega_\varphi$-compatible. Let $u\in \MM^\#(\hat{N},\gamma_-,\gamma,H^N,J)$ be a hybrid Floer cylinder. Then the energy of $u$ is given by  
\begin{equation*}
E(u) \coloneqq \int_{\R\times\T} \bigl( u^*\omega_{\varphi}-u^*dH^N\wedge dt \bigr) = \int_{\R\times\T}\omega_{\varphi}(\partial_su,J_t(u)\partial_su)\,ds\,dt \geq 0,
\end{equation*} 
where the equality between the second and the third term follows from the Floer equation \eqref{hybrid_cylinder}.  Consider the action functional 
\begin{equation*}
\MA_{H^N,\varphi}: C^\infty(\T,\hat{N})\to\R,\qquad \MA_{H^N,\varphi}(x)= \int_\T x^*\lambda_\varphi-\int_\T H^N(t,x(t))\,dt.
\end{equation*}
The 1-periodic orbits of $X_{H^N}$ are precisely the critical points of $\MA_{H^N,\varphi}$. The hybrid Floer cylinder $u$ is a positive gradient flow line of $\MA_{H^N,\varphi}$ with respect to the $L^2$-metric induced by $\omega_\varphi(\cdot,J\cdot)$. In particular, for any $-\infty < s_0 < s_1 < \infty$, we have
\begin{equation}
\label{eq:action_mixed_cylinder} 
\MA_{H^N,\varphi}(u(s_1,\cdot))-\MA_{H^N,\varphi}(u(s_0,\cdot)) = \int_{(s_0,s_1)\times \T}\omega_{\varphi}(\partial_su,J(u)\partial_su)\,ds\,dt \geq 0.
\end{equation}
From this, we see that
\begin{equation*}
E(u) = \lim_{s\rightarrow +\infty} \MA_{H^N,\varphi}(u(s,\cdot)) - \lim_{s\rightarrow -\infty} \MA_{H^N,\varphi}(u(s,\cdot)) = \MA_H(\gamma_-) - \MA(\gamma).
\end{equation*}

Consider a sequence $u_k \in \MM^\#(\hat{N},\gamma_-,\gamma,H^N,J)$. Bubbling happens if, for a suitable sequence $(s_k,t_k)\in \R\times\T$, the derivatives $|du_k(s_k,t_k)|$ are unbounded. In our current setting of hybrid Floer cylinders in $\hat{N}$, bubbling can be analyzed exactly in the same way as in the SFT compactness theorem for punctured pseudoholomorphic curves (see \cite[Section 10]{BEHWZ03}). Note that since $\hat{N}$ is an exact symplectic manifold, there do not exist any non-constant punctured finite energy spheres in $\hat{N}\setminus N^+$. Therefore, bubbling can only result in punctured finite energy spheres in the symplectization $\R\times Y$. In fact, each time such a finite energy sphere bubbles off, one obtains an entire pseudoholomorphic building $v$ which consists of a tree of punctured finite energy spheres $v_i:S^2\setminus\Gamma_i\rightarrow \R\times Y$ in possibly multiple symplectization levels. The building $v$ has exactly one positive asymptotic Reeb orbit and no negative asymptotic Reeb orbits. Therefore, the $d\lambda$-energy of $v$, i.e. the quantity
\begin{equation*}
E_{d\lambda}(v) = \sum_i \int_{S^2\setminus \Gamma_i} v_i^*d\lambda,
\end{equation*}
is at least $\MA(\gamma)$, the minimal action of a periodic Reeb orbit of $Y$. Recall from our construction of the Hamiltonian $H$ in Section \ref{neck_stretching_H} that $\MA_H(\gamma_-)$ is arranged to be close to $\MA(\gamma)$. In particular, we can assume that the energy $E(u_k) = \MA_H(\gamma_-) - \MA(\gamma)$ is strictly smaller than $\MA(\gamma)$. It follows that bubbles cannot form.

The absence of bubbles implies that $\|du_k\|_{L^\infty}$ remains bounded. Standard compactness arguments for Floer cylinders and pseudoholomorphic curves then imply that, for every sequence of reparametrizations $u_k(\sigma_k+s,t)$, there exists a subsequence which converges in $C^\infty_{\operatorname{loc}}$ to either a Floer cylinder in $\hat{N}$ satisfying \eqref{hybrid_cylinder} or a pseudoholomorphic cylinder in $\R\times Y$.

Breaking could happen in two different ways: either as a Floer cylinder at a $1$-periodic Hamiltonian orbit or as a pseudoholomorphic cylinder at a periodic Reeb orbit of $Y$. The former type of breaking occurs if a suitable sequence of reparametrizations $u_k(\sigma_k+s,t)$ converges in $C^\infty_{\operatorname{loc}}$ to $\alpha(t)$ for a $1$-periodic orbit $\alpha$ of $X_{H^N}$ with action $\MA_{H^N,\varphi}(\alpha) \in (\MA(\gamma),\MA_H(\gamma_-))$. Since there does not exist any such orbit $\alpha$, we can rule out breaking as a Floer cylinder. Breaking as a pseudoholomorphic cylinder occurs if a sequence of reparametrizations $u_k(\sigma_k+s,t)$ converges in $C^\infty_{\operatorname{loc}}$ to the trivial cylinder $u_\beta(s,t)$ in $\R\times Y$ over some periodic Reeb orbit $\beta$ with action $\MA(\beta) \in (\MA(\gamma), \MA_H(\gamma_-))$. Recall from Section \ref{neck_stretching_H} that $\MA_H(\gamma_-)$ is arranged to be sufficiently close to $\MA(\gamma)$ so that no such orbit $\beta$ exists. Hence, we can also rule out breaking as a pseudoholomorphic cylinder. We deduce that the moduli space $\MM(\hat{N},\gamma_-,\gamma,H^N,J)$ is compact.
\end{proof}

By linearizing the Floer equation \eqref{hybrid_cylinder} at $u\in  \MM^\#(\hat{N}, \gamma_-, \gamma, H^N,J)$, we obtain a real Cauchy-Riemann type operator 
\begin{equation}
\label{top_D_u}
D_u:V\oplus W^{1,p,\delta}(u^*T\hat{N}) \rightarrow L^{p,\delta}(u^*T\hat{N}).	
\end{equation}
Here, $V$ is a $2$-dimensional real vector space spanned by two smooth sections of $u^*T\hat{N}$ which are supported in the negative half cylinder and, near the negative puncture, are equal to $\partial_a$ and $R_\lambda$, respectively. As before, $W^{1,p,\delta}$ and $L^{p,\delta}$ are Sobolev spaces with exponential weight $\delta>0$, which is chosen to be sufficiently small. The operator $D_u$ is positively asymptotic to the asymptotic operator $A_{H^N,\gamma_-}=A_{H,\gamma_-}$ (see Section \ref{flosubsec}). With respect to the splitting $\gamma^* TN^-=\gamma^*\langle \partial_a,R_\lambda\rangle\oplus\gamma^*\xi$, the negative asymptotic operator of $D_u$ can be written as $A_0\oplus A_\gamma$, where $A_0 = -J\partial_t$ and $A_\gamma$ is the asymptotic operator of the Reeb orbit $\gamma$ (see Section \ref{subsec:asymptotics_at_punctures}). Let us point out that since $\gamma_-$ and $\gamma$ are nondegenerate as Hamiltonian and Reeb orbits, respectively, the weight $\delta$ is unnecessary (but harmless) at the positive end of $\R\times\T$ and also in the direction of $\xi$ at the negative end. The Fredholm index of $D_u$ is 
\[
\begin{split}
\mathrm{ind}(D_u) &= \mathrm{CZ}^{-\delta}(A_{H,\gamma_-})-(\mathrm{CZ}^{\delta}(A_0)+\mathrm{CZ}^{\delta}(A_\gamma))+\dim V \\
&= \mathrm{CZ}(A_{H,\gamma_-})-(\mathrm{CZ}^{\delta}(A_0)+\mathrm{CZ}(A_\gamma))+\dim V \\
&=n+1-(1+(n+1))+2=1.
\end{split}
\]
Since every cylinder $u\in \MM^\#(\hat{N}, \gamma_-, \gamma, H^N,J)$ crosses the region on which $J$ is $t$-dependent due to the asymptotic condition in \eqref{hybrid_cylinder_asymp_top}, the operator $D_u$ is surjective for a generic choice of $J$. In this case, $\MM(\hat{N},\gamma_-,\gamma, H^N,J)$ is a compact $0$-dimensional manifold by Lemma \ref{lem:compactness_mixed_partialX}.

The Floer-Hofer kernel gluing operation induces a natural isomorphism $\mathfrak{o}_{D_u}\otimes \mathfrak{o}_{\gamma} \cong \mathfrak{o}_{\gamma_-}$. The orientation local system of the moduli space $\MM^{\#}(\hat{N},\gamma_-,\gamma,H^N,J)$ is therefore naturally isomorphic to $\mathfrak{o}_{\gamma_-}\otimes \mathfrak{o}_{\gamma}^*$. Since $\MM(\hat{N},\gamma_-,\gamma,H^N,J)$ is obtained by quotienting out the $\R$-action, its orientation local system is naturally isomorphic to $\mathfrak{o}_{\gamma_-}\otimes \mathfrak{o}_{\gamma}^*\otimes \mathfrak{o}_{\R}^* \cong \mathfrak{o}_{\gamma_-}\otimes \mathfrak{o}_{\gamma}^*$. Here we use the obvious orientation of $\R$ to identify $\mathfrak{o}_\R\cong \Z$. The algebraic count $\#\MM(\hat{N},\gamma_-,\gamma,H^N,J)$ thus naturally takes values in the dual space $\mathfrak{o}_{\gamma_-}^*\otimes \mathfrak{o}_\gamma$.

\begin{lem}
\label{inv_top}
The count $\#\MM(\hat{N},\gamma_-,\gamma, H^N,J) \in \mathfrak{o}_{\gamma_-}^*\otimes \mathfrak{o}_{\gamma}$ is independent of the choice of generic admissible $\T$-family $J$ of almost complex structures on $\hat{N}$.
\end{lem}

\begin{proof}	
Let $J^0=(J^0_t)_{t\in\T}$ and $J^1=(J^1_t)_{t\in\T}$ be two generic admissible $\T$-families of almost complex structures. We choose a smooth homotopy $(J^\rho)_{\rho\in[0,1]}=(J^\rho_t)_{\rho\in[0,1],t\in\T}$ of admissible families of almost complex structures interpolating between $J^0$ and $J^1$ and consider the moduli space 
\[
\MM(\hat{N},\gamma_-,\gamma, H^N,(J^\rho)_{\rho\in[0,1]}):= \{ (\rho,[u]) \mid \rho\in[0,1],\; [u]\in \MM(\hat{N},\gamma_-,\gamma, H^N,J^\rho) \}.
\]
For a generic choice of homotopy $(J^\rho)$, it is a smooth 1-dimensional manifold. It suffices to show that this moduli space is compact. This can be done in the same way as in the proof of Lemma \ref{lem:compactness_mixed_partialX}. Exactly the  same energy considerations can be used to rule out bubbling and breaking.
\end{proof}

\subsection{Pseudoholomorphic cylinders in the middle level}
\label{middle}

Recall that the Hamiltonian $H$ constructed in Section \ref{neck_stretching_H} vanishes in the region $X\setminus U_{z_0}$, where $U_{z_0}$ is the neighborhood of $z_0$ in which we modified $h\circ H_X$ to obtain $H$. We may assume that $U_{z_0}$ is sufficiently small. Then there exist positive real numbers $a_1 < \dots < a_n$ linearly independent over $\Q$ such that the ellipsoid
\begin{equation*}
E\coloneqq\left\{(x_1,y_1,\dots,x_n,y_n)\in \R^{2n} \,\Big|\, \sum_{j=1}^n\frac{\pi (x_j^2+y_j^2)}{a_j}\leq 1 \right\}
\end{equation*}
satisfies $\overline{U}_{z_0}\subset \operatorname{int}(E)$ and $E\subset \operatorname{int}(X)$. The Reeb flow of the contact form $\lambda_E\coloneqq\lambda_0|_{F}$ on $F=\partial E$ has exactly $n$ simple periodic orbits and they have actions $a_1,\dots,a_n$. Moreover, the periodic orbit of minimal action $a_1$, denoted by $\gamma_E$, has Conley-Zehnder index $\mathrm{CZ}(\gamma_E)=n+1$. All other periodic orbits, including multiply covered ones, have Conley-Zehnder index at least $n+3$.

Let $M\coloneqq X\setminus \operatorname{int}(E)$ denote the strong symplectic cobordism between  $Y=\partial X$ and $F\coloneqq\partial E$. Let $J$ be an almost complex structure on the symplectic completion $\hat{M}$ which is admissible in the sense of Section \ref{subsec:setup_pseudoholomorphic_curves}. In the present subsection, we are interested in the moduli space
\begin{equation*}
\MM(\hat{M},\gamma,\gamma_E,J)
\end{equation*}
of unparametrized $J$-holomorphic cylinders in $\hat{M}$ which are positively asymptotic to $\gamma$ and negatively asymptotic to $\gamma_E$.

\begin{lem} 
\label{lem:compactness_hol_generic}
 For a generic admissible almost complex structure $J$ on $\hat{M}$, the moduli space $\MM(\hat{M}, \gamma, \gamma_E,J)$ is compact. 
\end{lem}

\begin{proof}
Let $(u_k)_{k}$ be a sequence in $\MM(\hat{M},\gamma,\gamma_E,J)$. By the SFT compactness theorem, after possibly passing to a subsequence, it converges to a pseudoholomorphic building. The top level $u$ of the limiting building is a punctured pseudoholomorphic sphere, either in the symplectization of $Y$ or in $\hat{M}$. It has exactly one positive puncture, at which it is asymptotic to $\gamma$, and at least one negative puncture. Since $\gamma$ is the Reeb orbit of $Y$ of minimal action, the only such punctured sphere in the symplectization of $Y$ is the trivial cylinder over $\gamma$. We conclude that $u$ is contained in $\hat{M}$. Let $\Gamma^-$ denote the set of negative punctures of $u$. For $z\in \Gamma^-$, let $\gamma_z$ denote the corresponding negative asymptotic Reeb orbit on $F$. As reviewed in Section \ref{subsec:fredholm_theory}, the Fredholm index of $u$ is given by
\begin{equation}
\label{eq:compactness_hol_generic_proof}
\operatorname{ind}(u) = (n-3)(1-\# \Gamma^-) + \operatorname{CZ}(\gamma) - \sum\limits_{z\in \Gamma^-} \operatorname{CZ}(\gamma_z).
\end{equation}
Note that $u$ is somewhere injective because the Reeb orbit $\gamma$ is simple. Since $J$ is generic, we must have $\operatorname{ind}(u)\geq 0$. We have $\operatorname{CZ}(\gamma) = n+1$ and $\operatorname{CZ}(\gamma_z)\geq n+1$ for all $z\in \Gamma^-$ because the Reeb flow on $F$ is dynamically convex. Combining these facts with \eqref{eq:compactness_hol_generic_proof}, we deduce that $u$ has exactly one negative puncture and that the corresponding asymptotic Reeb orbit must have Conley-Zehnder index $n+1$. The only such Reeb orbit on $F$ is $\gamma_E$. Therefore, $u$ is contained in $\MM(\hat{M},\gamma,\gamma_E,J)$, concluding the proof of compactness.
\end{proof}

Recall from Section \ref{subsec:fredholm_theory} that the Fredholm index of a cylinder $u\in \MM(\hat{M},\gamma,\gamma_E,J)$ is given by $\operatorname{ind}(u) = \operatorname{CZ}(\gamma) - \operatorname{CZ}(\gamma_E) = 0$. The moduli space $\MM(\hat{M},\gamma,\gamma_E,J)$ only contains somewhere injective curves because $\gamma$ is simple and is therefore regular for a generic $J$ admissible on $\hat{M}$. By Lemma \ref{lem:compactness_hol_generic}, the moduli space $\MM(\hat{M},\gamma,\gamma_E,J)$ is then a compact $0$-dimensional manifold and we have a canonical algebraic count $\#\MM(\hat{M},\gamma,\gamma_E,J)$ taking values in $\mathfrak{o}_\gamma^*\otimes \mathfrak{o}_{\gamma_E}$.

\begin{lem}
The count $\#\MM(\hat{M},\gamma,\gamma_E,J) \in \mathfrak{o}_\gamma^*\otimes \mathfrak{o}_{\gamma_E}$ is independent of the choice of a generic admissible almost complex structure $J$ on $\hat{M}$. 
\end{lem}

\begin{proof}
Let $J^0$ and $J^1$ be generic admissible almost complex structures on $\hat{M}$. We choose a smooth homotopy $(J^\rho)_{\rho\in[0,1]}$ of admissible almost complex structures between $J^0$ and $J^1$, and define 
\[
\MM(\hat{M},\gamma,\gamma_E,(J^\rho)_{\rho\in[0,1]})\coloneqq \{(\rho,u)\mid \rho\in[0,1],\; u\in \MM(\hat{M},\gamma,\gamma_E,J^\rho)\}.
\]
We choose $(J^\rho)_\rho$ to be generic so that the above moduli space is a smooth $1$-dimensional manifold. It suffices to show that this moduli space
is compact for generic $(J^\rho)_\rho$. We prove this via a Fredholm index argument similar to the one in the proof of Lemma \ref{lem:compactness_hol_generic}. This argument will make crucial use of the fact that $\gamma_E$ is the only periodic Reeb orbit on $F$ of Conley-Zehnder index at most $n+2$.

Let $(\rho_k,u_k)_{k}$ be  a sequence  in $\MM(\hat{M},\gamma,\gamma_E,(J^\rho)_{\rho\in[0,1]})$. After passing to a subsequence, $\rho_k$ converges to some $\rho\in[0,1]$ and $u_k$ converges to a $J^\rho$-holomorphic building by the SFT compactness theorem. By the same reasoning as in the proof of Lemma \ref{lem:compactness_hol_generic}, the top level $u$ of the limiting building must be a punctured pseudoholomorphic sphere in $\hat{M}$ because $\gamma$ is the orbit of least action on $Y$. It has exactly one positive puncture, at which it is asymptotic to $\gamma$ and at least one negative puncture. Let $\Gamma^-$ denote the set of negative punctures of $u$ and let $\gamma_z$ be the asymptotic orbit of $z\in \Gamma^-$. Let $\MM$ denote the moduli space consisting of all tuples $(\sigma,v)$ where $\sigma\in [0,1]$ and $v$ is a punctured (unparametrized) $J^\sigma$-holomorphic sphere in $\hat{M}$ with the same asymptotic orbits as $u$. The virtual dimension of $\MM$ is given by
\begin{equation*}
1 + (n-3)(1-\#\Gamma^-) + \operatorname{CZ}(\gamma) - \sum\limits_{z\in \Gamma^-} \operatorname{CZ}(\gamma_z).
\end{equation*}
Since the homotopy $(J^\rho)_\rho$ was chosen generically and curves in $\MM$ are automatically somewhere injective, this virtual dimension must be non-negative. Since $\operatorname{CZ}(\gamma) = \operatorname{CZ}(\gamma_E) = n+1$ and all other periodic Reeb orbits on $F$ have Conley-Zehnder index at least $n+3$, the only possibility for this to happen is that $\gamma_E$ is the only negative asymptotic orbit of $u$. This concludes the proof of compactness.
\end{proof}

\subsection{Hybrid Floer cylinders in the bottom level}
\label{lowest}
\label{cylinder_E}

We consider the symplectic completion $\hat{E}$ of the ellipsoid $E$ and the smooth Hamiltonian
\begin{equation*}
H^E:\hat{E}\rightarrow\R,\qquad H^E(z):=
\begin{cases}
H(z) & z\in E, \\
0 & z \in E^+.
\end{cases}
\end{equation*}
Here we recall that the function $H\in C^\infty(\T\times\R^{2n})$ defined in Section \ref{neck_stretching_H} is autonomous on $E$. The Hamiltonian vector field $X_{H^E}$ has only constant $1$-periodic orbits, namely the constant orbit mapping to the unique minimum point $z_0$ of $H^E$ and the constant orbits mapping to $(H^E)^{-1}(0)$. For an admissible almost complex structure $J$ on $\hat{E}$, we consider Floer cylinders
\begin{equation}
\label{eq:mixed_moduli_space_E}
u:\R\times \mathbb{T}\rightarrow \hat{E},\qquad \partial_su + J(u) (\partial_tu - X_{H^E}(u)) = 0, 
\end{equation}
which are, as $s\rightarrow +\infty$,  positively asymptotic to the periodic Reeb orbit $\gamma_E$ on $F$ and, as $s\rightarrow -\infty$, converge to the Hamiltonian orbit $z_0$. Writing $u|_{(s_0,\infty)\times \T} = (a,v)$ in components for $s_0\gg 0$ sufficiently large, we can express these asymptotic conditions as
\begin{equation}
\label{eq:mixed_moduli_space_E_asymptotic}
\begin{split}
&\lim_{s\rightarrow -\infty} u(s,t) =  z_0,\\
&\lim_{s\rightarrow +\infty} v(s,t) =  \gamma_E(\mathcal{A}(\gamma_E)(t + t_0))  \;\mbox{ for some }\; t_0\in\T  , \qquad 
\lim_{s\rightarrow +\infty} a(s,t)=+\infty.
\end{split} 
\end{equation}
We are interested in the moduli spaces
\begin{align*}
\MM^\#(\hat{E}, \gamma_E, z_0,H^E,J) &\coloneqq \left\{ u:\R\times \mathbb{T}\rightarrow \hat{E} \mid \text{$u$ satisfies \eqref{eq:mixed_moduli_space_E} and \eqref{eq:mixed_moduli_space_E_asymptotic}} \right\},\\
\MM(\hat{E},\gamma_E,z_0,H^E,J) &\coloneqq \MM^\#(\hat{E}, \gamma_E, z_0,H^E,J)/(\R\times\T),
\end{align*}
where the $(\R\times\T)$-action is given by $(\sigma,\tau,u)\mapsto u(\sigma+\cdot,\tau+\cdot)$.

\begin{lem}
\label{lem:compactness_mixed_E}
The moduli space $\MM(\hat{E},\gamma_E,z_0,H^E,J)$ is compact for every admissible almost complex structure $J$  on $\hat{E}$.
\end{lem}

\begin{proof}
The proof of Lemma \ref{lem:compactness_mixed_E} is similar to the proof of Lemma \ref{lem:compactness_mixed_partialX}. We use energy considerations to rule out all possible degenerations of hybrid Floer cylinders in $\MM(\hat{E},\gamma_E,z_0,H^E,J)$.

The appropriate notion of energy of hybrid Floer cylinders is defined as follows. Recall that $E^+=[0,\infty)\times F$ is attached to the collar neighborhood $(-\epsilon,0]\times F$ of the  boundary of $E$. Fix a smooth function $\varphi:(-\epsilon,\infty)\to(0,1)$ such that $\varphi'> 0$ everywhere, $\varphi(s)=e^s$ for $s\in(-\epsilon,-\epsilon/2)$, and $\lim_{s\rightarrow +\infty}\varphi(s) = 1$. We define a smooth symplectic form $\omega_\varphi$ on $\hat{E}$ by 
\[
\omega_{\varphi}:=d\lambda_\varphi,\qquad \lambda_{\varphi}:=
\begin{cases}
\lambda_0 & \textrm{ on } E\setminus ((-\epsilon,0]\times  F)\\
\varphi\lambda_E & \textrm{ on } (-\epsilon,\infty)\times F.
\end{cases}
\]
Let $\epsilon>0$ be so  small that $H^E=0$ on $(-\epsilon,\infty)\times F$. Then the Hamiltonian vector field $X_{H^E}$ of $H^E$ with respect to $d\lambda_0$ agrees with the Hamiltonian vector field of $H^E$ with respect to $\omega_\varphi$.

Let $J$ be an admissible almost complex structure on $\hat{E}$. Note that $J$ is $\omega_\varphi$-compatible. The energy of $u\in\MM^\#(\hat{E}, \gamma_E, z_0,H^E,J)$ is defined by
\begin{equation*}
E(u) \coloneqq \int_{\R\times \T} \bigl( u^*\omega_{\varphi}-u^*dH^E\wedge dt \bigr) = \int_{\R\times\T}\omega_{\varphi}(\partial_su,J(u)\partial_su)\,ds\,dt\geq 0.
\end{equation*}
The critical points of the action functional 
\begin{equation*}
\MA_{H^E,\varphi}: C^\infty(\T,\hat{E})\to\R,\qquad \MA_{H^E,\varphi}(x)= \int_\T x^*\lambda_\varphi-\int_\T H^E(t,x(t))\,dt
\end{equation*}
are precisely the $1$-periodic orbits of $X_{H^E}$. Cylinders in $\MM^\#(\hat{E}, \gamma_E, z_0,H^E,J)$ are positive gradient flow lines of $\MA_{H^N,\varphi}$ with respect to the $L^2$-metric induced by $\omega_\varphi(\cdot,J\cdot)$. Using this we compute, for any $-\infty< s_0< s_1<\infty$,
\begin{equation}
\label{eq:action_mixed_cylinder_bottom} 
\MA_{H^E,\varphi}(u(s_1,\cdot))-\MA_{H^E,\varphi}(u(s_0,\cdot)) = \int_{(s_0,s_1)\times \T}\omega_{\varphi}(\partial_su,J(u)\partial_su)\,ds\,dt \geq 0.
\end{equation}
Moreover, we have
\begin{equation*}
E(u) = \lim_{s\rightarrow +\infty}\MA_{H^E,\varphi}(u(s,\cdot)) - \lim_{s\rightarrow -\infty} \MA_{H^E,\varphi}(u(s,\cdot)) = \MA(\gamma_E) - \MA_H(z_0).
\end{equation*}

Consider a sequence of hybrid Floer cylinders $u_k\in \MM^\#(\hat{E},\gamma_E,z_0,H^E,J)$. Note that every punctured finite energy sphere in $\hat{E}$ or $\R\times F$ must have at least one positive asymptotic Reeb orbit. Arguing similarly as in the proof of Lemma \ref{lem:compactness_mixed_partialX}, we therefore see that the formation of a bubble requires energy at least $\MA(\gamma_E)$, the minimal action of a periodic Reeb orbit of $F$. However, the action $\MA_H(z_0)$ is strictly positive, implying that the energy $E(u_k) = \MA(\gamma_E) - \MA_H(z_0)$ is strictly smaller than $\MA(\gamma_E)$. This rules out bubbling. The remaining potential degenerations of $(u_k)_k$ are breaking as a pseudoholomorphic cylinder at a periodic Reeb orbit of $F$ and breaking as a Floer cylinder at a $1$-periodic Hamiltonian orbit. The former type of breaking can be ruled out because there does not exist any periodic Reeb orbit $\alpha$ on $F$ with action $\MA(\alpha) \in (\MA_H(z_0),\MA(\gamma_E))$. The latter type of breaking cannot occur because there does not exist any $1$-periodic orbit $\beta$ of $X_{H^E}$ with action $\MA_{H^E,\varphi}(\beta) \in (\MA_H(z_0),\MA(\gamma_E))$. We conclude that $\MM(\hat{E},\gamma_E,z_0,H^E,J)$ is compact.
\end{proof}

By linearizing the Floer equation in \eqref{eq:mixed_moduli_space_E} at $u\in\MM^\#(\hat{E},\gamma_E,z_0,H^E,J)$, we obtain a real Cauchy-Riemann type operator
\[
D_u:V\oplus W^{1,p,\delta}(u^*T\hat{E}) \rightarrow L^{p,\delta}(u^*T\hat{E}).
\]
Here, $V$ is a $2$-dimensional real vector space spanned by two smooth sections of $u^*T\hat{E}$ which are supported in the positive half cylinder and, near the positive puncture, are constant equal to $\partial_a$ and $R_{\lambda_E}$, respectively.  As before, the weight $\delta>0$ in the Sobolev and Lebesgue spaces $W^{1,p,\delta}$ and $L^{p,\delta}$ is chosen sufficiently small. With respect to the decomposition $\gamma_E^*\langle \partial_a, R_{\lambda_E}\rangle\oplus \gamma_E^*\operatorname{ker}\lambda_E$, the positive asymptotic operator of $D_u$ is given by $A_0\oplus A_{\gamma_E}$, where $A_0 = -J\partial_t$ and $A_{\gamma_E}$ is the asymptotic operator of the periodic Reeb orbit $\gamma_E$. The negative asymptotic operator of $D_u$ is $A_{H^E,z_0}=A_{H,z_0}$. Similarly to the situation in Section \ref{top}, the small asymptotic weight $\delta$ is really only necessary in the direction $\gamma_E^*\langle \partial_a,R_{\lambda_E}\rangle$ at the positive end, but it does no harm to impose it everywhere. The operator $D_u$ is Fredholm with index 
\begin{equation*}
\begin{split}
\operatorname{ind}(D_u)&=\mathrm{CZ}^{-\delta}(A_0)+\mathrm{CZ}^{-\delta}(A_{\gamma_E})-\mathrm{CZ}^{\delta}(A_{H,z_0})+\dim V\\
&=\mathrm{CZ}^{-\delta}(A_0)+\mathrm{CZ}(A_{\gamma_E})-\mathrm{CZ}(A_{H,z_0})+\dim V\\
&=-1+(n+1)-n+2=2.
\end{split}
\end{equation*}
Note that the hybrid cylinder $u$ is pseudoholomorphic on the complement of $U_{z_0}$. Since the asymptotic Reeb orbit $\gamma_E$ of $u$ is simple, $u$ must have an injective point in $\operatorname{int}(E) \setminus \overline{U}_{z_0}$, i.e. there exists $(s,t)\in \R\times \T$ such that $u(s,t) \in \operatorname{int}(E)\setminus \overline{U}_{z_0}$, the linearization $du(s,t)$ is injective, and $u^{-1}(u(s,t)) = \left\{ (s,t) \right\}$. It therefore follows from standard transversality theory (see e.g. \cite[Chapter 7]{wen16b}) that for a generic choice of admissible almost complex structure $J$ on $\hat{E}$, every hybrid cylinder $u\in \MM^\#(\hat{E},\gamma_E,z_0,H^E,J)$ is regular, i.e. has a surjective linearized operator $D_u$. For such a $J$, the moduli space $\MM(\hat{E},\gamma_E,z_0,H^E,J)$ is a smooth compact $0$-dimensional manifold. The algebraic count $\#\MM(\hat{E},\gamma_E,z_0,H^E,J)$ canonically takes values in $\mathfrak{o}_{\gamma_E}^*\otimes \mathfrak{o}_{z_0}$.

\begin{lem}
The count $\#\MM(\hat{E},\gamma_E,z_0,H^E,J) \in \mathfrak{o}_{\gamma_E}^*\otimes \mathfrak{o}_{z_0}$ is independent of the choice of generic admissible almost complex structure $J$ on $\hat{E}$.
\end{lem}

\begin{proof}
This follows from an easy adaptation of the proofs of Lemmas \ref{inv_top} and \ref{lem:compactness_mixed_E}.
\end{proof}

\subsection{Stretching the neck}
\label{neck-stretching1}
The hypersurfaces $F$ and $Y$ divide $\R^{2n}$ into three pieces $E$, $M=X\setminus\mathrm{int}(E)$, and $\hat{N}\setminus \operatorname{int}(N^-)$. We want to stretch the neck along $F$ and $Y$. To this end, we fix a sufficiently small constant $\epsilon>0$ and consider a family of smooth manifolds 
\[
\hat{W}_R := E\underset{\alpha_R^-}{\cup}\big((-R-\epsilon,\epsilon)\times F\big)\underset{\alpha_R^+}{\cup} M \underset{\beta_R^-}{\cup} \big((-R-\epsilon,\epsilon)\times Y\big) \underset{\beta_R^+}{\cup} (\hat{N}\setminus \operatorname{int}(N^-))
\]
parametrized by $R\geq0$, where the attaching maps are given by
\begin{alignat*}{2}
&\alpha_R^-:(-R-\epsilon,-R]\times F\to E,\qquad &&(t,x)\mapsto \phi_{Z}^{t+R}(x), \\
&\alpha_R^+:[0,\epsilon)\times F\to M,\quad 	&&(t,x)\mapsto \phi_{Z}^{t}(x),\\
&\beta_R^-:(-R-\epsilon,-R]\times Y\to M,\qquad &&(t,x)\mapsto \phi_{Z}^{t+R}(x), \\
&\beta_R^+:[0,\epsilon)\times Y\to \hat{N}\setminus \operatorname{int}(N^-),\quad 	&&(t,x)\mapsto \phi_{Z}^{t}(x).
\end{alignat*}
Here, $\phi_Z^t$ denotes the flow of the radial Liouville vector field $Z$. 
Let $J^E$, $J^M$, and $J^N=(J^N_t)_{t\in \T}$ be admissible almost complex structures on $\hat{E}$, $\hat{M}$, and $\hat{N}$, respectively. We assume that they are generic in the sense of Sections \ref{top}, \ref{middle}, and \ref{lowest}, and that there exist admissible almost complex structures $J^F$ and $J^Y$ on the symplectizations $\R\times F$ and $\R\times Y$, respectively, such that
\[
\begin{split}
&J^E=J^F \mbox{ on } E^+=[0,\infty)\times F,\qquad J^M=J^F \mbox{ on } M^-=(-\infty,0]\times F,\\
&J^M=J^Y \mbox{ on } N^+=[0,\infty)\times Y,\qquad J^N=J^Y \mbox{ on } N^-=(-\infty,0]\times Y.
\end{split}
\]
Gluing such a triple $(J^E,J^M,J^N)$ defines an almost complex structures $J_R$ on $\hat{W}_R$. Similarly, we glue the Hamiltonians $H^E$ on $\hat{E}$ and $H^N$ on $\hat{N}$ to obtain a Hamiltonian $H_R$ on $\hat{W}_R$. To be precise, we set
\[
(H_R,J_R):=
\begin{cases}
(H^E,J^E) & \mbox{ on } 	E,\\
(\;0\;,J^M) & \mbox{ on }  ((-R,\epsilon)\times F) \cup M  \cup ((-R-\epsilon,0)\times Y),\\
(H^N,J^N) & \mbox{ on } \hat{N}\setminus \operatorname{int}(N^-).
\end{cases}
\]

\begin{lem}
\label{lem:neckstretching_compactness}
Let $(R_k)_k$ be a sequence of positive real numbers diverging to $+\infty$ and let $u_k\in  {\mathcal{M}}_{\partial^F}(\hat{W}_{R_k},\gamma_-,z_0;H_{R_k},J_{R_k})$. After passing to a subsequence, the sequence $(u_k)_k$ converges in the sense of SFT to a hybrid Floer/pseudoholomorphic building with three levels 
\[
\big(u^N,u^M,u^E\big)\in\MM(\hat{N},\gamma_-,\gamma, H^N,J^N) \times  \MM(\hat{M},\gamma, \gamma_E,J^M) \times \MM(\hat{E},\gamma_E,z_0,H^E,J^E).
\]
\end{lem}

\begin{proof}
Lemma \ref{lem:neckstretching_compactness} follows from essentially the same energy and index considerations we already used in the proofs of Lemmas \ref{lem:compactness_mixed_partialX}, \ref{lem:compactness_hol_generic} and \ref{lem:compactness_mixed_E}.

By the SFT compactness theorem and standard compactness results for Floer cylinders, we may pick a subsequence such that $u_k$ converges to a hybrid Floer/pseudoholomorphic building. A priori, the $\hat{N}$-level of the limiting building could consist of a broken chain of Floer cylinders, possibly with punctures negatively asymptotic to periodic Reeb orbits of $Y$. There are no $1$-periodic orbits $\alpha$ of $X_{H^N}$ in $\hat{N}$ with action $\MA_{H^N,\varphi}(\alpha)\in (\MA_H(z_0),\MA_H(\gamma_-))$, where $\MA_{H^N,\varphi}$ is the action functional introduced in the proof of Lemma \ref{lem:compactness_mixed_partialX}. Hence we can rule out breaking of $u_k$ at a Hamiltonian orbit in $\hat{N}$ and conclude that the $\hat{N}$-level of the limiting building must be a single punctured Floer cylinder $u^N:(\R\times \T)\setminus\Gamma \rightarrow \hat{N}$ which, as $s\rightarrow +\infty$, is asymptotic to $\gamma_-$ and, as $s\rightarrow -\infty$, is negatively asymptotic to some periodic Reeb orbit $\gamma_{-\infty}$ of $Y$. At each puncture $z\in \Gamma$, the cylinder $u^N$ is negatively asymptotic to some periodic Reeb orbit $\gamma_z$. The energy of hybrid Floer cylinders introduced in the proof of Lemma \ref{lem:compactness_mixed_partialX} also makes sense for the punctured Floer cylinder $u^N$:
\begin{equation*}
E(u^N) = \int_{(\R\times \T)\setminus \Gamma} \bigl( (u^N)^*\omega_\varphi - (u^N)^*dH^N\wedge dt \bigl) = \int_{(\R\times \T)\setminus\Gamma} \omega_\varphi(\partial_su^N,J(u^N)\partial_su^N) \,ds \,dt \geq 0.
\end{equation*}
It satisfies
\begin{equation}
\label{eq:neckstretching_compactness_proof_punctured_energy}
E(u^N) = \MA_H(\gamma_-) - \MA(\gamma_{-\infty}) - \sum_{z\in \Gamma} \MA(\gamma_z).
\end{equation}
Recall that $\gamma$ is the only periodic Reeb orbit of $Y$ of action at most $\MA_H(\gamma_-)$ and that $\MA_H(\gamma_-)$ is arranged to be very close to $\MA(\gamma)$. Thus it follows from \eqref{eq:neckstretching_compactness_proof_punctured_energy} and non-negativity of the energy $E(u^N)$ that $\Gamma = \emptyset$ and $\gamma_{-\infty} = \gamma$. In other words, the $\hat{N}$-level of the limiting building is a hybrid Floer cylinder $u^N\in \MM(\hat{N},\gamma_-,\gamma,H^N,J^N)$.

The limiting building cannot have a level in the symplectization $\R\times Y$. The reason is that because $\gamma$ is the Reeb orbit of minimal action, the only punctured pseudoholomorphic curve in $\R\times Y$ which is positively asymptotic to $\gamma$ and has at least one negative puncture is the trivial cylinder over $\gamma$.

The level of the limiting building contained in $\hat{M}$ must be a punctured pseudoholomorphic sphere positively asymptotic to $\gamma$ and with at least one negative puncture. As shown in the proof of Lemma \ref{lem:compactness_hol_generic}, the only possibility for such a punctured sphere to have non-negative Fredholm index is to be a cylinder negatively asymptotic to $\gamma_E$. Since the almost complex structure is chosen generically, the Fredholm index must be non-negative. Thus the second level of the limiting building is a pseudoholomorphic cylinder $u^M\in \MM(\hat{M},\gamma,\gamma_E,J^M)$.

Since $\gamma_E$ is the Reeb orbit of $F$ of smallest action, the same reasoning we used to rule out a level in $\R\times Y$ also rules out a level in $\R\times F$.

Finally, we observe that we cannot have breaking at $1$-periodic Hamiltonian orbits in $\hat{E}$ because $z_0$ is the only Hamiltonian orbit in this region with positive action. This shows that the $\hat{E}$-level of the limiting building is a hybrid Floer cylinder $u^E\in \MM(\hat{E},\gamma_E,z_0,H^E,J^E)$.
\end{proof}

\begin{lem}
\label{lem:gluing1}
For every sufficiently large $R> 0$, there exists a gluing homeomorphism
\[
\MM(\hat{N},\gamma_-,\gamma, H^N,J^N) \times  \MM(\hat{M},\gamma, \gamma_E,J^M) \times \MM(\hat{E},\gamma_E,z_0,H^E,J^E) \to {\mathcal{M}}_{\partial^F}(\hat{W}_R,\gamma_-,z_0;H_R,J_R).
\]
The isomorphism between orientation local systems induced by this gluing homeomorphism agrees with the natural isomorphism $(\mathfrak{o}_{\gamma_-}\otimes\mathfrak{o}_\gamma^*)\otimes (\mathfrak{o}_{\gamma}\otimes \mathfrak{o}_{\gamma_E}^*)\otimes (\mathfrak{o}_{\gamma_E} \otimes \mathfrak{o}_{z_0}^*) \cong \mathfrak{o}_{\gamma_-}\otimes \mathfrak{o}_{z_0}^*$. In particular, 
\begin{multline*}
|\#\MM(\hat{N},\gamma_-,\gamma, H^N,J^N)|\cdot 
|\#\MM(\hat{M},\gamma, \gamma_E,J^M)|\cdot
|\#\MM(\hat{E},\gamma_E,z_0,H^E,J^E)| \\
=
|\# {\mathcal{M}}_{\partial^F}(\hat{W}_R,\gamma_-,z_0;H_R,J_R)|.
\end{multline*}
\end{lem}

\begin{proof}
Lemma \ref{lem:gluing1} follows from Lemma \ref{lem:neckstretching_compactness} and a gluing construction for punctured pseudoholomorphic curves asymptotic to periodic Reeb orbits. Gluing results in this setting have appeared in \cite{par19} in the context of contact homology and in \cite{ht07, ht09} in the context of embedded contact homology. The gluing technique described in \cite[Section 5]{par19} can be applied to our present setting without significant changes. The fact that the hybrid Floer cylinders under consideration are not pseudoholomorphic away from the gluing region does not affect the gluing construction.
\end{proof}

\begin{cor}
\label{cor:cylinder_count}
We have $|\#\MM(\hat{M},\gamma,\gamma_E,J^M)|=1$.
\end{cor}

\begin{proof}
Let $\psi_R:\hat{W}_R\to \R^{2n}$ be the diffeomorphism satisfying $H_R=H\circ \psi_R$ mentioned above. This induces a bijection
 \[
 {\mathcal{M}}_{\partial^F}(\gamma_-,z_0;H,(\psi_R)_*J_R) \cong  {\mathcal{M}}_{\partial^F}(\hat{W}_R,\gamma_-,z_0;H_R,J_R),
 \]
where elements in the the moduli space on the left are Floer cylinders in $\R^{2n}$, see Section \ref{flosubsec}. The induced map of orientation local systems is simply the identity map of $\mathfrak{o}_{\gamma_-}\otimes \mathfrak{o}_{z_0}^*$. By Proposition \ref{prop:floer_count_1}, the absolute value of the algebraic count of the moduli space on the left is equal to $1$. Hence Lemma \ref{lem:gluing1} implies that
\begin{equation*}
|\#\MM(\hat{N},\gamma_-,\gamma, H^N,J^N)|\cdot
|\#\MM(\hat{M},\gamma, \gamma_E,J^M)|\cdot
|\#\MM(\hat{E},\gamma_E,z_0,H^E,J)|=1,
\end{equation*}
from which we conclude that $|\#\MM(\hat{M},\gamma, \gamma_E,J^M)|=1$.
\end{proof}

\subsection{Proof of Theorem \ref{count-planes}}
\label{neck-stretching2}

Adapting the neck stretching argument in \cite{hwz03}, one can prove Theorem \ref{count-planes} for the ellipsoid $E$.

\begin{lem}
\label{lem:deg_ev_E}
For a generic admissible almost complex structure $J^E$ on $\hat{E}$, the absolute value of the degree of $\operatorname{ev}_{\hat{E}} : \MM_1(\hat{E},\gamma_E,{J}^E) \rightarrow \hat{E}$ is equal to $1$.
\end{lem}

\begin{proof}
It is well known from Gromov-Witten theory that the count of pseudoholomorphic spheres in $\C \mathrm{P}^n$ representing the homology class of a complex line and passing through two prescribed points is equal to $1$. Our strategy is to deduce the lemma from this fact via a neck-stretching argument. To this end, we symplectically embed $(E,\omega_0|_E)$ into $(\C \mathrm{P}^n,c\cdot \omega_{\mathrm{FS}})$, where $\omega_{\mathrm{FS}}$ is the Fubini-Study form and $c>0$ is sufficiently large. We set $V\coloneqq \C \mathrm{P}^n \setminus \operatorname{int}(E)$ and fix two points $p\in\mathrm{int}(E)$ and $q\in \operatorname{int}(V)$. We pick generic admissible almost complex structures $J^E$ and $J^V$ on the completions $\hat{E}$ and $\hat{V}$, respectively, which, on the cylindrical ends, agree with the same admissible almost complex structure $J^F$ on $\R\times F$. For $R\geq 0$, let $Q_R$ be the smooth manifold obtained from $\C \mathrm{P}^n$ by replacing a neighborhood $(-\epsilon,\epsilon)\times F$ of $F$ by $(-R-\epsilon,\epsilon)\times F$. Let $J_R$ be the almost complex structure on $Q_R$ obtained by gluing $J^E$ and $J^V$. Note that, for every $R\geq 0$, the manifold $Q_R$ is diffeomorphic to $\C \mathrm{P}^n$ and that the pushforward of $J_R$ via a suitable diffeomorphism $\psi_R:Q_R\rightarrow \C \mathrm{P}^n$ compressing the neck is a compatible almost complex structure on $\C \mathrm{P}^n$.

Let $A \in H_2(\C \mathrm{P}^n)$ denote the homology class represented by a complex line. Let 
\[
\MM_2(Q_R,A,J_R)
\] 
denote the moduli space of unparametrized $J_R$-holomorphic spheres in $Q_R$ with two marked points and representing the homology class $A$. Since the almost complex structures are chosen generically, this moduli space is a closed $4n$-manifold for all $R>0$ in the complement of some discrete subset of $\R$. It comes equipped with a canonical complex orientation. We have an evaluation map
\begin{equation*}
\operatorname{ev}_{Q_R} : \MM_2(Q_R,A,J_R) \rightarrow Q_R \times Q_R
\end{equation*}
and the algebraic count of $\operatorname{ev}_{Q_R}^{-1}(p,q)$ is equal to $1$ (see e.g.\ \cite[Example 7.1.14]{ms12}).

Note that we have isomorphisms $H_2(\C \mathrm{P}^n) \cong H_2(\C \mathrm{P}^n,E) \cong H_2(\hat{V},V^-)$ and we abbreviate the image of $A$ under these isomorphisms by the same symbol. Let $\MM_1(\hat{V},A,\gamma_E,J^V)$ denote the moduli space of unparametrized $J^V$-holomorphic planes with one marked point representing the homology class $A$ and negatively asymptotic to $\gamma_E$. This moduli space is a smooth manifold, since $J^V$ is generic and all planes contained in it are necessarily somewhere injective. While the evaluation map $\operatorname{ev}_{\hat{V}}$ at the marked point is not necessarily proper, we argue that $\operatorname{ev}_{\hat{V}}^{-1}(q)$ is nevertheless a compact $0$-manifold. Indeed, suppose that $u_k$ is a sequence in $\operatorname{ev}_{\hat{V}}^{-1}(q)$ and pick a subsequence which converges to a pseudoholomorphic building in the sense of SFT. Note that each non-constant punctured pseudoholomorphic curve in $\hat{V}$ must represent a positive multiple of the homology class $A$ for action reasons. Therefore, the top level of the limiting building must consist of a single punctured pseudoholomorphic sphere $u$ in $\hat{V}$ representing $A$ and passing through $q$. Since $A$ is indivisible, $u$ must be somewhere injective. Since $J^V$ is generic and $u$ passes through the prescribed point $q$, it must have Fredholm index at least $2n-2$. Let $\Gamma^-$ denote the set of negative punctures and, for $z\in \Gamma^-$, let $\gamma_z$ be the corresponding negative asymptotic orbit. Then the Fredholm index is given by
\begin{equation*}
(n-3)(2-\#\Gamma^-) + 2(n+1) - \sum\limits_{z\in \Gamma^-} \operatorname{CZ}(\gamma_z).
\end{equation*}
Recall that the orbit $\gamma_E$ has Conley-Zehnder index equal to $n+1$ and that all other Reeb orbits on $F$ have index at least $n+3$. This allows us to conclude that $u$ has exactly one negative puncture, at which it is asymptotic to $\gamma_E$. This proves compactness of $\operatorname{ev}_{\hat{V}}^{-1}(q)$. Note that the orientation local system of the moduli space $\MM_1(\hat{V},A,\gamma_E,J^V)$ is canonically isomorphic to $\mathfrak{o}_{\gamma_E}^*$. Therefore, we have a canonical count $\# \operatorname{ev}_{\hat{V}}^{-1}(q) \in \mathfrak{o}_{\gamma_E}$.

For sufficiently large $R\gg0$, there is a gluing homeomorphism
\begin{equation}
\label{eq:deg_ev_E_proof_gluing_homeo}
\operatorname{ev}_{\hat{V}}^{-1}(q) \times \operatorname{ev}_{\hat{E}}^{-1}(p) \cong \operatorname{ev}_{Q_R}^{-1}(p,q).
\end{equation}
Similarly to the discussion of gluing in Section \ref{neck-stretching1}, the proof of this claim splits into two parts. The first part is an analogue of Lemma \ref{lem:neckstretching_compactness}. One shows that, for any sequence $R_k$ diverging to $+\infty$ and for any sequence of pseudoholomorphic spheres $u_k \in \operatorname{ev}_{Q_{R_k}}^{-1}(p,q)$, there exists a subsequence converging to a $2$-level pseudoholomorphic building whose top level is an element of $\operatorname{ev}_{\hat{V}}^{-1}(q)$ and whose bottom level is an element of $\operatorname{ev}_{\hat{E}}^{-1}(p)$. In order to see this, note that the Fredholm index consideration used above to show that $\operatorname{ev}_{\hat{V}}^{-1}(q)$ is compact also implies that the top level of a limiting building of $u_k$ must be an element of $\operatorname{ev}_{\hat{V}}^{-1}(q)$. Using that each $u_k$ passes through the point $p$ and $\gamma_E$ is the systole of the Reeb flow on $F$, it follows that there must be exactly two levels and that the bottom level must be an element of $\operatorname{ev}_{\hat{E}}^{-1}(p)$. The second part of the construction of \eqref{eq:deg_ev_E_proof_gluing_homeo} is a gluing result for punctured pseudoholomorphic curves. As in Lemma \ref{lem:gluing1}, the construction given in \cite[Section 5]{par19} can be carried out in the present setting and we omit further details.

The isomorphism of orientation local systems induced by \eqref{eq:deg_ev_E_proof_gluing_homeo} agrees with the natural isomorphism $\mathfrak{o}_{\gamma_E}^* \otimes \mathfrak{o}_{\gamma_E} \cong \Z$. We conclude that
\begin{equation*}
|\# \operatorname{ev}_{\hat{V}}^{-1}(q)| \cdot |\# \operatorname{ev}_{\hat{E}}^{-1}(p)| = |\# \operatorname{ev}_{Q_R}^{-1}(p,q)| = 1.
\end{equation*}
Clearly, this implies that $|\operatorname{deg}(\operatorname{ev}_{\hat{E}})| = |\# \operatorname{ev}_{\hat{E}}^{-1}(p)|=1$.
\end{proof}

Recall that the ellipsoid $E$ is contained in the uniformly convex domain $X$. The hypersurface $F=\partial E$ divides $\hat{X}$ into $E$ and $\hat{M}\setminus \operatorname{int}(M^-)$. We consider pseudoholomorphic planes in $\hat{X}$ asymptotic to $\gamma$ passing through a point $p\in E$ and stretch the neck along $F$. For $R\geq 0$, let $\hat{X}_R$ denote the manifold obtained from $\hat{X}$ by replacing a neighborhood $(-\epsilon,\epsilon)\times F$ of $F$ by the neck $(-R-\epsilon,\epsilon)\times F$. Let $J^E$ and $J^M$ be generic admissible almost complex structures on $\hat{E}$ and $\hat{M}$, respectively, which, on the cylindrical ends, agree with the same admissible almost complex structure $J^F$ on the symplectization $\R\times F$. Gluing $J^E$ and $J^M$ induces an almost complex structure $J_R$ on $\hat{X}_R$.

Let $\MM_1(\hat{X}_R,\gamma,J_R)$ denote the moduli space of unparametrized $J_R$-holomorphic planes in $\hat{X}_R$ with one marked point and positively asymptotic to $\gamma$. We are interested in the moduli space of such planes passing through a point $p\in E$, i.e., in the preimage $\operatorname{ev}_{\hat{X}_R}^{-1}(p)$ under the evaluation map at the marked point.

\begin{lem}
\label{lem:neckstretching_compactness2}
Let $(R_k)_k$ be a sequence of positive real numbers diverging to $+\infty$ and let $(u_k,z_k)\in \operatorname{ev}_{\hat{X}_R}^{-1}(p)$. After passing to a subsequence, the sequence $(u_k,z_k)$ converges in the sense of SFT to a pseudoholomorphic building with two levels 
\[
(u^M,(u^E,z))\in  \MM(\hat{M},\gamma,\gamma_E,J^M) \times \operatorname{ev}_{\hat{E}}^{-1}(p).
\]
\end{lem}

\begin{proof}
Pick a subsequence such that $u_k$ converges to some limiting building. The top level of this building is a punctured pseudoholomorphic sphere, either in $\hat{M}$ or in the symplectization of $Y$, which is positively asymptotic to $\gamma$. It must have at least one negative puncture because each $u_k$ passes through the point $p$. Since $\gamma$ is the orbit of minimal action on $Y$, we conclude that the top level is contained in $\hat{M}$. It follows from the same Fredholm index argument that was used in the proof of Lemma \ref{lem:compactness_hol_generic} that the top level is in fact a cylinder $u^M \in \MM(\hat{M},\gamma,\gamma_E,J^M)$. Using that $\gamma_E$ is the Reeb orbit of minimal action on $F$, deduce that there is only one remaining level, which is a plane $u^E$ in $\hat{E}$ positively asymptotic to $\gamma_E$ and with marked point $z$ mapping to $p$.
\end{proof}

\begin{lem}
\label{lem:gluing2}
For every sufficiently large $R> 0$, there exists a gluing homeomorphism
\[
\MM(\hat{M},\gamma,\gamma_E,J^M) \times \operatorname{ev}_{\hat{E}}^{-1}(p) \to \operatorname{ev}_{\hat{X}_R}^{-1}(p).
\]
The isomorphism of orientation local systems induced by this homeomorphism agrees with the natural isomorphism
\begin{equation*}
\mathfrak{o}_{\gamma}\otimes \mathfrak{o}_{\gamma_E}^* \otimes \mathfrak{o}_{\gamma_E} \cong \mathfrak{o}_{\gamma}.
\end{equation*}
In particular, 
\[
|\#\MM(\hat{M},\gamma,\gamma_E,J^M)| \cdot |\#\operatorname{ev}_{\hat{E}}^{-1}(p)| = |\# \operatorname{ev}_{\hat{X}_R}^{-1}(p)|.
\]	
\end{lem}
\begin{proof}
Similarly to Lemma \ref{lem:gluing1}, the assertion of Lemma \ref{lem:gluing2} follows from Lemma \ref{lem:neckstretching_compactness2} and the gluing construction for punctured pseudoholomorphic curves in \cite[Section 5]{par19}.	
\end{proof}

Finally, we are in a position to prove Theorem \ref{count-planes}.

\begin{proof}[Proof of Theorem \ref{count-planes}]
Let $\psi_R:\hat{X}_R\to\hat{X}_0=\hat{X}$ be a diffeomorphism given by shrinking $(-R-\epsilon,\epsilon)\times F$ to $(-\epsilon,\epsilon)\times F$ and equip $\hat{X}$ with the almost complex structure $(\psi_R)_*J_R$. This induces a bijection
\[
\operatorname{ev}_{\hat{X}}^{-1}(p) \cong \operatorname{ev}_{\hat{X}_R}^{-1}(p),
\]
whose induced map of orientation local systems is simply the identity map of $\mathfrak{o}_\gamma$. In combination with Lemma \ref{lem:gluing2}, this implies
\[
|\operatorname{deg}(\operatorname{ev}_{\hat{X}})| = |\#\MM(\hat{M},\gamma,\gamma_E,J^M)| \cdot |\operatorname{deg}(\operatorname{ev}_{\hat{E}})|.
\]
Both factors appearing in the product on the right hand side are equal to $1$, as proved in Corollary \ref{cor:cylinder_count} and Lemma \ref{lem:deg_ev_E}. This proves that $|\operatorname{deg}(\operatorname{ev}_{\hat{X}})| = 1$.
\end{proof}

\section{Proofs of Theorem \ref{main} and Corollary \ref{cor:hopf_systole_on_convex_domain}}
\label{sec:proof_main_theorem}

\begin{proof}[Proof of Theorem \ref{main}]
Up to a translation, we can assume that the origin is contained in the interior of $X$. If the Reeb flow of $\lambda=\lambda_0|_{\partial X}$ has a unique periodic orbit $\gamma$ of minimal action and $\gamma$ is nondegenerate, then Theorems \ref{thm:degree_characterizes_hopf} and \ref{count-planes} imply that $\gamma$ bounds a disk-like global surface of section.

The general case can be proved by the following perturbation argument. Using the radial projection $\rho: S^3 \rightarrow \partial X$, we pull the contact form $\lambda$ back to $S^3$ and obtain the contact form $\lambda_X := \rho^* \lambda$ on $S^3$ inducing the standard contact structure $\xi_{\mathrm{st}}=\ker \lambda_0|_{S^3}$. Let $\gamma$ be a periodic Reeb orbit of $\lambda_X$ of minimal action $a_0:= \MA_{\min}(X)$. Let $K$ denote its knot type and $k$ its self-linking number. Our goal is to prove that $K$ is the unknot and $k=-1$. Once we know this, it follows from \cite[Theorem 1.7]{hry14} that $\gamma$ bounds a disk-like global surface of section.

Note that $\gamma$ might not be isolated among periodic orbits of action $a_0$ and that the number $a_0$ might not be isolated in the action spectrum of $\lambda_X$. We consider the compact set $\Gamma\subset S^3$ consisting of all periodic Reeb orbits of $\lambda_X$ of action not exceeding $\frac{5}{3} a_0$. Note that all of these Reeb orbits are simple because $\frac{5}{3}<2$. We define $\Gamma_0$ to be the subset of $\Gamma$ consisting of those periodic orbits which have knot type $K$ and self-linking number $k$. The set $\Gamma_0$ is not empty because it contains $\gamma$.

In the space of embedded closed curves which are transverse to $\xi_{\mathrm{st}}$, having a given knot type and a given self-linking number is a $C^1$-open condition. We deduce that the compact set $\Gamma$ is partitioned into finitely many open and closed subsets corresponding to different pairs (knot type, self-linking number), with $\Gamma_0$ being one of these subsets. Therefore, we can find open neighborhoods $U$ of $\Gamma_0$ and $V$ of $\Gamma\setminus \Gamma_0$ such that $\overline{U} \cap \overline{V} = \emptyset$. If $\Gamma\setminus \Gamma_0$ is empty, we can take $U=S^3$ and $V=\emptyset$.

Let $\chi$ be a smooth real function on $S^3$ such that $\chi=0$ on $U$ and $\chi=1$ on $V$. For every $\epsilon>0$, we consider the contact form
\begin{equation*}
\lambda_{\epsilon} := e^{\epsilon \chi} \lambda_X.
\end{equation*}
If a contact form is $C^1$-close to $\lambda_X$, then its Reeb vector field is $C^0$-close to the Reeb vector field of $\lambda_X$ and hence its periodic orbits of action not exceeding $\frac{4}{3}a_0$ are $C^1$-close to some periodic orbits of $\lambda_X$ of action not exceeding $\frac{5}{3} a_0$. Therefore, if $\epsilon>0$ is small enough, then every periodic orbit of $\lambda_{\epsilon}$ of action not exceeding $\frac{4}{3}a_0$ is contained in either  $U$ or $V$. By the form of $\lambda_{\epsilon}$, the periodic orbits of $\lambda_{\epsilon}$ which are fully contained in either $U$ or $V$ coincide with the periodic orbits of $\lambda_X$ contained in these sets. Their action is unchanged in the case of orbits in $U$ and gets multiplied by the factor $e^{\epsilon}>1$ in the case of orbits in $V$. We conclude that the following assertions hold for some fixed small $\epsilon>0$:
\begin{enumerate}[(i)]
\item $\gamma$ is a periodic orbit of $\lambda_{\epsilon}$ of minimal action and its action is equal to $a_0$.
\item All periodic orbits of $\lambda_{\epsilon}$ of action strictly less than $e^\epsilon a_0$ have knot type $K$ and self-linking number $k$.
\end{enumerate}
Taking a generic $C^{\infty}$-small perturbation $f_{\epsilon}$ of the function $e^{\epsilon\chi}$, we find a contact form $\tilde{\lambda}_{\epsilon} = f_{\epsilon} \lambda_X$ all of whose periodic Reeb orbits are nondegenerate and such that there is just one periodic orbit $\tilde{\gamma}_{\epsilon}$ of minimal action. It is possible to arrange $f_{\epsilon}$ so that the action of $\tilde{\gamma}_{\epsilon}$ is at most $a_0$. Indeed, simply choose a perturbation $f_{\epsilon}$ which is generic subject to the conditions that $f_{\epsilon}(p) = 1$ and $df_{\epsilon}(p)=0$ for every point $p$ in the image of $\gamma$. Then $\gamma$ is a periodic orbit of $\tilde{\lambda}_{\epsilon}$ of action $a_0$, implying that $\tilde{\gamma}_{\epsilon}$ has action at most $a_0$.

Since $\tilde{\lambda}_{\epsilon}$ is $C^1$-close to $\lambda_{\epsilon}$ and $\tilde{\gamma}_{\epsilon}$ has action at most $a_0$, it follows that $\tilde{\gamma}_{\epsilon}$ is $C^1$-close to some periodic orbit of $\lambda_{\epsilon}$ of action strictly less than $e^{\epsilon}a_0$. By assertion (ii), the orbit $\tilde{\gamma}_{\epsilon}$ must have knot type $K$ and self-linking number $k$. Note that the contact form $\tilde{\lambda}_{\epsilon}$ is the radial pullback to $S^3$ of the restriction of $\lambda_0$ to the boundary of some uniformly convex domain $\tilde{X}_{\epsilon}$ because $f_{\epsilon}$ is $C^2$-close to the constant function $1$. By the special case considered at the beginning of this proof, the orbit $\tilde{\gamma}_{\epsilon}$ therefore bounds a disk-like global surface of section and in particular is a Hopf orbit. We conclude that $K$ is the unknot and $k=-1$, as we wished to prove.
\end{proof}

\begin{proof}[Proof of Corollary \ref{cor:hopf_systole_on_convex_domain}]
Let $X\subset \R^4$ be a smooth convex domain which is not necessarily uniformly convex. Pick a sequence $X_k$ of uniformly convex domains converging to $X$ with respect to the $C^\infty$ topology. For each $k$, pick a closed characteristic $\gamma_k$ of $\partial X_k$ of minimal action. By Theorem \ref{main}, each $\gamma_k$ bounds a disk-like global surface of section and is therefore a Hopf orbit. After passing to a subsequence, the sequence $\gamma_k$ smoothly converges to a periodic orbit $\gamma$ on $\partial X$. Since $\MA_{\min}$ is a capacity on convex bodies, the sequence $\MA_{\min}(X_k)$ converges to $\MA_{\min}(X)$. It follows that the orbit $\gamma$ must have minimal action and therefore be simple. Since each $\gamma_k$ is Hopf, it is easy to see that $\gamma$ is Hopf as well.
\end{proof}

\appendix

\section{Orientations of certain Fredholm operators}
\label{appA}

\subsection{The determinant bundle} 

Given two real Banach spaces $X$ and $Y$, we denote by $\Phi(X,Y)$ the space of Fredholm operators from  $X$ and $Y$. The \textit{determinant line} over $T\in \Phi(X,Y)$ is the real line
\[
\Det(T):= \Lambda^{\max}(\ker T) \otimes (\Lambda^{\max}(\coker T))^*,
\]
where $\Lambda^{\max}$ denotes the top exterior power functor on the category of finite dimensional real vector spaces, with the usual convention $\Lambda^{\max}((0)):=\R$. The union of these lines defines the total space $\Det(\Phi(X,Y))$ of an analytic real line bundle over $\Phi(X,Y)$ called \textit{determinant bundle}, where the base is equipped with the operator norm topology. See e.g.\ \cite{qui85}, \cite{fh93} or \cite[Section 7]{ama09}. 

It is useful to recall the definition of the analytic structure on this line bundle, as an ingredient of this construction is used also in the proof of Theorem \ref{isom} in Appendix \ref{appB}. We fix a finite dimensional subspace $V\subset Y$ and consider the open set $\mathcal{T}_V$ of those operators $T\in \Phi(X,Y)$ which are transverse to $V$, meaning that $V+T(X)=Y$. For $T$ in $\mathcal{T}_V$, we have the exact sequence
\begin{equation}
\label{exact}
0 \rightarrow  \Ker T \stackrel{i}{\rightarrow} T^{-1} (V) \stackrel{\tau}{\rightarrow} V \stackrel{\pi}{\rightarrow} \coker T \rightarrow 0,
\end{equation}
where $i$ denotes the inclusion, $\tau$ the restriction of $T$ and $\pi$ the restriction of the quotient projection onto $\coker T = Y/T(X)$. This exact sequence induces the canonical isomorphism
\begin{equation}
\label{isodet}
\Det(T) = \Lambda^{\max}(\Ker T) \otimes (\Lambda^{\max}(\coker T))^*  \cong \Lambda^{\max}(T^{-1}(V)) \otimes  \Lambda^{\max} (V)^*,
\end{equation}
given by
\[
\alpha \otimes (\pi_* \gamma)^*  \mapsto (i_* \alpha \wedge \beta) \otimes (\gamma \wedge \tau_* \beta)^*,
\]
where $\alpha$ generates $\Lambda^{\max}(\ker T)$, $i_*\alpha \wedge \beta$ generates $\Lambda^{\max}(T^{-1}(V))$, $\gamma \wedge \tau_*\beta$ generates $\Lambda^{\max}(V)$, and the superscript $*$ indicates the dual generator. 

By the above isomorphism, we can identify the determinant line bundle over $\mathcal{T}_V$ with the line bundle whose fiber at $T$ is $ \Lambda^{\max}(T^{-1}(V)) \otimes  \Lambda^{\max}(V)^*$. Since the map $T\mapsto T^{-1}(V)$ is analytic on $\mathcal{T}_V$, the latter line bundle is analytic. The sets $\mathcal{T}_V$ for $V$ in the finite dimensional Grassmannian of $Y$ build an open cover of $\Phi(X,Y)$ and it suffices to check that the transition mappings are analytic. This verification uses the particular choice of the sign of the isomorphism \eqref{isodet} (see \cite[Sections 5]{ama09} and \cite[Section 11]{sei08b}).

Denote by $\mathcal{T}$ the open subset of $\Phi(X,Y) \times \mathrm{Gr}_{\mathrm{fin}}(Y)$ consisting of the pairs $(T,V)$ such that $T$ is transverse to $V$, where $\mathrm{Gr}_{\mathrm{fin}}(Y)$ denotes the Grassmannian of finite dimensional vector subspaces of $Y$. The above construction and the analyticity of the map  
\begin{equation}
\label{invim}
\mathcal{T} \rightarrow \mathrm{Gr}_{\mathrm{fin}}(X), \qquad (T,V) \mapsto T^{-1}(V)
\end{equation}
show that \eqref{isodet} defines an analytic lift of this map to an analytic bundle morphism 
\[
\bigl(\Det(\Phi(X,Y))\otimes \Lambda^{\max}(\mathrm{Gr}_{\mathrm{fin}}(Y))\bigr)|_{\mathcal{T}} \longrightarrow \Lambda^{\max}(\mathrm{Gr}_{\mathrm{fin}}(X))
\]
which is a fiberwise isomorphism.

\subsection{The orientation local system} 

The \textit{orientation line} of $T\in \Phi(X,Y)$ is the free $\Z$-module of rank 1
\begin{equation}
\label{orfred}
\mathfrak{o}_T := \mathfrak{o}_{\ker T} \otimes \mathfrak{o}_{\coker T}^* \cong \mathfrak{o}_{\Lambda^{\max}(\ker T)} \otimes \mathfrak{o}_{\Lambda^{\max}(\coker T)}^* \cong \mathfrak{o}_{\Det(T)},
\end{equation}
where we recall that the orientation line of the $n$-dimensional real vector space $V$, $0\leq n < \infty$, is defined as
\[
\mathfrak{o}_V := H_n(V,V\setminus \{0\};\Z),
\]
and we have used the canonical isomorphism
\[
\mathfrak{o}_V \cong \mathfrak{o}_{\Lambda^{\max}(V)}.
\]
The same isomorphism and the exterior product induce a canonical isomorphism
\begin{equation}
\label{or-sum}
\mathfrak{o}_V \otimes \mathfrak{o}_W \cong \mathfrak{o}_{V\oplus W},
\end{equation}
where $V$ and $W$ are finite dimensional real vector spaces. 
The isomorphism \eqref{orfred} and the vector bundle structure of the determinant bundle imply that the orientation lines of Fredholm operators build a local system of  free $\Z$-modules of rank 1 over $\Phi(X,Y)$. 

If an operator $T\in \Phi(X,Y)$ is transverse to the finite dimensional subspace $V\subset Y$, then the isomorphism \eqref{isodet} determines a canonical isomorphism
\begin{equation}
\label{isoor}
\mathfrak{o}_T \cong  \mathfrak{o}_{T^{-1}(V)} \otimes  \mathfrak{o}_V^*,
\end{equation}
which defines a lift of the map \eqref{invim} to the orientation local systems.

When $X$ and $Y$ are separable infinite dimensional Hilbert spaces, the connected components of $\Phi(X,Y)$ are labeled by the Fredholm index and each of them has fundamental group $\Z_2$. In this case, the orientation local system over $\Phi(X,Y)$ is, up to isomorphism, the unique non-trivial local system of  free $\Z$-modules of rank 1 over this space.

\subsection{Orientations of Cauchy-Riemann type linear operators on cylinders}
 
In the remaining part of this appendix, we discuss the gluing of orientations over certain spaces of real Cauchy Riemann type operators. Unlike in Section \ref{CRsec}, we here work in a Hilbert setting and consider only trivial symplectic vector bundles over a Riemann surface $\Sigma$ which is either a cylinder or a half cylinder. The complex structure on these trivial vector bundles is allowed to be non-constant. We start with the case of the cylinder, recalling results from \cite{fh93}.

Denote by $\mathcal{J}(\R^{2n},\omega_0)$ the space of $\omega_0$-compatible complex structures on $\R^{2n}$, i.e., linear isomorphisms $J: \R^{2n} \rightarrow \R^{2n}$ such that $J^2=-\mathrm{id}$ and $\omega_0(\cdot,J \cdot)$ is a positive definite scalar product. The standard complex structure 
\[
J_0(x,y) := (-y,x),
\]
which corresponds to the multiplication by $i$ in the identification $\R^{2n} \cong \C^n$ given by $(x,y)\mapsto x+iy$, is $\omega_0$-compatible.

Set $\T:= \R/\Z$ and denote by $\mathcal{A}(2n)$ the space of pairs $(j,a)$ where $j\in C(\T, \mathcal{J}(\R^{2n},\omega_0))$ and $a\in C(\T,\mathrm{sp}(2n))$ 
 is a loop into the Lie algebra of the linear symplectic group $\mathrm{Sp}(2n)$ such that the path $w:[0,1] \rightarrow \mathrm{Sp}(2n)$ defined by
\begin{equation}
\label{symp-path}
w'(t) = a(t) w(t), \qquad w(0)=\mathrm{id},
\end{equation}
is nondegenerate, meaning that $1$ is not an eigenvalue of $w(1)$. We denote by $\mathrm{CZ}(a):= \mathrm{CZ}(w)$ the Conley-Zehnder index of the nondegenerate symplectic path $w$.

Given two elements $\alpha_+=(j_+,a_+)$ and $\alpha_-=(j_-,a_-)$ in $\mathcal{A}(2n)$, we denote by $\mathcal{D}(\alpha_+,\alpha_-)$ the space of operators of the form
\[
H^1(\R\times \T,\R^{2n}) \rightarrow L^2 (\R\times \T,\R^{2n}), \quad u \mapsto \partial_s u + J(s,t) ( \partial_t u - A(s,t) u), \quad (s,t)\in \R\times \T,
\]
such that, denoting by $\overline{\R}:= \R\cup\{\pm \infty\}$ the extended real line and by $\mathcal{L}(\R^{2n})$ the space of linear endomorphisms of $\R^{2n}$, we have:
\begin{enumerate}[(i)]
\item $J\in C( \overline{\R}\times \T, \mathcal{J}(\R^{2n},\omega_0))$ satisfies $J(\pm \infty,t) = j_{\pm}(t)$ for every $t\in \T$;
\item $A\in C( \overline{\R}\times \T,\mathcal{L}(\R^{2n}))$ satisfies $A(\pm \infty,t) = a_{\pm}(t)$ for every $t\in \T$.
\end{enumerate}
It is well known that these operators are Fredholm of index $\mathrm{CZ}(a_+) - \mathrm{CZ}(a_-)$, see \cite{sz92}. 

The fact that $\mathcal{J}(\R^{2n},\omega_0)$ is contractible and the space of maps $A$ satisfying (ii) is convex implies that $\mathcal{D}(\alpha_+,\alpha_-)$ is a contractible subspace of $\Phi(H^1(\R\times \T,\R^{2n}),L^2 (\R\times \T,\R^{2n}))$, see \cite[Proposition 7]{fh93}. Therefore, the orientation local system is trivial on $\mathcal{D}(\alpha_+,\alpha_-)$, meaning that we have a canonical isomorphism
\[
\mathfrak{o}_{D_1} \cong \mathfrak{o}_{D_2}
\]
for every pair of operators $D_1, D_2\in \mathcal{D}(\alpha_+,\alpha_-)$. We can then unambiguously denote by
\[
\mathfrak{o}_{\mathcal{D}(\alpha_+,\alpha_-)}
\]
the orientation line of an arbitrary $D\in \mathcal{D}(\alpha_+,\alpha_-)$.

Now let $D\in \mathcal{D}(\alpha_1,\alpha_2)$ and $D'\in \mathcal{D}(\alpha_2,\alpha_3)$ for some $\alpha_1,\alpha_2,\alpha_3\in \mathcal{A}(2n)$. The linear gluing construction which is described in \cite[Section 3]{fh93}) defines a 1-parameter family of operators $\{D\#_{\sigma} D'\}_{\sigma>0}$ in $\mathcal{D}(\alpha_1,\alpha_3)$ and isomorphisms
\[
\mathfrak{o}_D \otimes \mathfrak{o}_{D'} \cong \mathfrak{o}_{D\#_{\sigma} D'}.
\]
This induces a canonical isomorphism
\[
\mathfrak{o}_{\mathcal{D}(\alpha_1,\alpha_2)} \otimes \mathfrak{o}_{\mathcal{D}(\alpha_2,\alpha_3)} \cong \mathfrak{o}_{\mathcal{D}(\alpha_1,\alpha_3)}.
\]

\subsection{Orientations of Cauchy-Riemann type linear operators on negative half-cylinders} 

Any $\gamma\in L^2(\T,\R^{2n})$ has the Fourier representation
\begin{equation}
\label{fourier}
\gamma(t) = \sum_{k\in \Z} e^{2\pi k t J_0} \widehat{\gamma}(k),
\end{equation}
where $\widehat{\gamma}: \Z \rightarrow \R^{2n}$ is square-summable. The Hilbert space
\[
\begin{split}
H^{\frac{1}{2}}(\T,\R^{2n}) &:= \Bigl\{ \gamma \in L^2(\T,\R^{2n}) \mid \sum_{k\in \Z} |k| |\widehat{\gamma}(k)|^2 < +\infty \Bigr\}, \\ (\gamma,\beta)_{\frac{1}{2}} &:= \widehat{\gamma}(0) \cdot \widehat{\beta}(0) + 2\pi \sum_{k\in \Z} |k| \, \widehat{\gamma}(k) \cdot \widehat{\beta}(k), 
\end{split}
\]
has the orthogonal splitting
\[
H^{\frac{1}{2}}(\T,\R^{2n}) = \mathbb{H}_+ \oplus \mathbb{H}_0 \oplus \mathbb{H}_-,
\]
where
\[
\mathbb{H}_+ := \Bigl\{  \gamma \in H^{\frac{1}{2}}(\T,\R^{2n}) \mid \widehat{\gamma}(k) = 0 \; \forall k\leq 0 \Bigr\}, \quad
 \mathbb{H}_- := \Bigl\{  \gamma \in H^{\frac{1}{2}}(\T,\R^{2n}) \mid \widehat{\gamma}(k) = 0 \; \forall k\geq 0 \Bigr\},
\]
and $\mathbb{H}_0\cong \R^{2n}$ denotes the space of constant loops. Denote by $\mathbb{P}_+$, $\mathbb{P}_0$ and $\mathbb{P}_-$ the corresponding projectors.

Given an element $\alpha=(j,a)\in \mathcal{A}(2n)$, we denote by $\mathcal{D}_-(\alpha)$ the space of operators of the form
\[
\begin{split}
H^1((-\infty,0)\times \T,\R^{2n}) \longrightarrow  L^2 ((-\infty,0)\times \T,\R^{2n}) \times \mathbb{H}^+, \\
\qquad u \mapsto \bigl(\partial_s u + J(s,t) (\partial_t u - A(s,t) u), \mathbb{P}_+ u(0,\cdot) \bigr)
\end{split}
\]
where
\begin{enumerate}[(i)]
\item $J\in C( [-\infty,0] \times \T, \mathcal{J}(\R^{2n},\omega_0))$ satisfies $J(- \infty,t) = j(t)$ for every $t\in \T$ and $J(s,t)=J_0$ for every $(s,t)\in [-1,0]\times \T$;
\item $A\in C(  [-\infty,0] \times \T,\mathcal{L}(\R^{2n}))$ satisfies $A(- \infty,t) = a(t)$ for every $t\in \T$.
\end{enumerate}
Here, $u(0,\cdot)$ denotes the trace of the $H^1$ map $u$ on the boundary $\{0\}\times \T$ of the cylinder, which is indeed an element of $H^{\frac{1}{2}}(\T,\R^{2n})$. The next result is a simple variant of \cite[Proposition 8.2]{ak22}.

\begin{prop}
\label{hybridop}
Assume that $\alpha=(j,a) \in \mathcal{A}(2n)$. Then every operator in $\mathcal{D}_-(\alpha)$ is Fredholm of index $n - \mathrm{CZ}(a)$.
\end{prop}

Again, the space $\mathcal{D}_-(\alpha)$ is contractible, and hence the orientation local system is trivial on $\mathcal{D}_-(\alpha)$. This allows us to define the orientation line
\[
\mathfrak{o}_{\mathcal{D}_-(\alpha)}
\]
as $\mathfrak{o}_D$ for an arbitrary $D\in \mathcal{D}_-(\alpha)$.

Arguing as in \cite[Section 3]{fh93}, an operator $D_1$ in $\mathcal{D}_-(\alpha_1)$ defined by the data $(J_1,A_1)$ and an operator $D_2$ in $\mathcal{D}(\alpha_1,\alpha_2)$ defined by the data $(J_2,A_2)$ and can be glued producing a 1-parameter family of operators $\{D_1 \#_{\sigma} D_2\}_{\sigma>0}$ in $\mathcal{D}_-(\alpha_2)$ of the form
\[
(D_1 \#_{\sigma} D_2)(u) =   \bigl(\partial_s u + J(s,t) (\partial_t u - A(s,t) u), \mathbb{P}_+ u(0,\cdot) \bigr)
\]
where $(J,A)=(J_1,A_1)$ for $s\in [-\sigma,0]$ and $(J,A)=(J_2,A_2)(\cdot+3\sigma,\cdot)$ for $s\in (-\infty,-2\sigma]$, together with isomorphisms
\[
\mathfrak{o}_{D_1} \otimes \mathfrak{o}_{D_2} \cong \mathfrak{o}_{D_1 \#_{\sigma} D_2}.
\]
A variant of \cite[Theorem 10]{fh93} now gives us a canonical isomorphism
\begin{equation}
\label{isocan}
\mathfrak{o}_{\mathcal{D}_-(\alpha_1)} \otimes \mathfrak{o}_{\mathcal{D}(\alpha_1,\alpha_2)} \cong \mathfrak{o}_{\mathcal{D}_-(\alpha_2)},
\end{equation}
for every $\alpha_1,\alpha_2\in \mathcal{A}(2n)$. 

\section{Proof of Theorem \ref{isom}}
\label{appB}

Let $H\in C^{\infty}(\T\times \R^{2n})$ be a quadratically convex, nondegenerate and non-resonant at infinity time-periodic Hamiltonian. In this Appendix, we recall the definition of the isomorphism from the Morse complex of $\psi_{H^*}$ to the Floer complex of $H$ from \cite{ak22} and upgrade it to integer coefficients, hence proving Theorem \ref{isom}. We use the notation introduced in Section \ref{floclasec} and in Appendix \ref{appA}.

\subsection{A second look at the Floer complex of $H$} 

Given $\gamma\in \mathcal{P}(H)$, we set
\[
j_{\gamma}(t):= J(t,\gamma(t)), \qquad a_{\gamma}(t):= dX_H(t,\gamma(t)), \qquad \alpha_{\gamma}:= (j_{\gamma},a_{\gamma})\in \mathcal{A}(2n),
\]
where $J$ is the time-periodic almost complex structure on $\R^{2n}$ which is used in the definition of the Floer complex of $H$. We also simplify the notation by setting for every $\gamma_+,\gamma_-\in \mathcal{P}(H)$
\[
\mathcal{D}(\gamma_+,\gamma_-):= \mathcal{D}(\alpha_{\gamma_+},\alpha_{\gamma_-}), \qquad \mathcal{D}_-(\gamma_-):= \mathcal{D}_-(\alpha_{\gamma_-}).
\]
Here, $\mathcal{D}$ and $\mathcal{D}_-$ are the spaces of real Cauchy-Riemann type operators from Appendix \ref{appA}.

Following a standard approach, the orientation line $\mathfrak{o}_{\gamma}$ of $\gamma\in \mathcal{P}(H)$ is defined in Section \ref{flosubsec} as the orientation line of any real Cauchy-Riemann type operator $D$ on $\C$ which is positively asymptotic to the asymptotic operator associated to $\alpha_{\gamma}$. We start by noticing that we have a canonical isomorphism
\begin{equation}
\label{altdefn}
\mathfrak{o}_{\gamma} \cong \mathfrak{o}_{\mathcal{D}_-(\gamma)}^*.
\end{equation}
Indeed, by gluing the above operator $D$ with an operator $D'$ in $\mathcal{D}_-(\gamma)$ we obtain a real Cauchy-Riemann type operator $D'\# D$ on the closed disk $\D$. The space of these operators on $\D$ is contractible and contains a connected subset consisting of complex linear operators. Therefore, $\mathfrak{o}_{D'\# D}$ is canonically isomorphic to $\Z$, and the gluing isomorphism
\[
\mathfrak{o}_{\mathcal{D}_-(\gamma)} \otimes \mathfrak{o}_{\gamma}= \mathfrak{o}_{D'} \otimes \mathfrak{o}_{D} \cong \mathfrak{o}_{D'\# D} \cong \Z
\]
induces the isomorphim \eqref{altdefn}.

We shall hereafter identify $\mathfrak{o}_{\gamma}$ with $\mathfrak{o}_{\mathcal{D}_-(\gamma)}^*$. The space of Floer cylinders $\mathcal{M}_{\partial^F}^{\#}(\gamma_+,\gamma_-)$ coincides with the set of zeroes of the nonlinear map
\[
\tilde{u} + H^1(\R\times \T,\R^{2n}) \rightarrow L^2(\R\times \T,\R^{2n}), \qquad u\mapsto \partial_s u + J(t,u) (\partial_t u - X_H(t,u)),
\]
where $\tilde{u}: \R\times \T \rightarrow \R^{2n}$ is a smooth map such that $\tilde{u}(-s,t) = \gamma_-(t)$ and $\tilde{u}(s,t) = \gamma_+(t)$ for every $s\geq 1$ and $t\in \T$ (see \cite[Proposition A.1]{ak22} for the regularity result behind this fact). The differential of this map at $u\in \mathcal{M}_{\partial^F}^{\#}(\gamma_+,\gamma_-)$ is an operator $D_u$ in the space $\mathcal{D}(\gamma_+,\gamma_-)$. For a generic $J$, $D_u$ is surjective with kernel $T_u \mathcal{M}_{\partial^F}^{\#}(\gamma_+,\gamma_-)$, and the canonical isomorphism \eqref{isocan} reads  
\[
\mathfrak{o}_{\mathcal{D}_-(\gamma_+)} \otimes \mathfrak{o}_{T_u \mathcal{M}_{\partial^F}^{\#}(\gamma_+,\gamma_-)} \cong \mathfrak{o}_{\mathcal{D}_-(\gamma_-)},
\]
or, equivalently,
\begin{equation}
\label{isocan2}
\mathfrak{o}_{T_u \mathcal{M}_{\partial^F}^{\#}(\gamma_+,\gamma_-)} \cong  
\mathfrak{o}_{\gamma_+} \otimes \mathfrak{o}_{\gamma_-}^*.
\end{equation}
As discussed in Section \ref{flosubsec}, in the case $\mathrm{CZ}(\gamma_+) - \mathrm{CZ}(\gamma_-)=1$  the quotient of $\mathcal{M}_{\partial^F}^{\#}(\gamma_+,\gamma_-)$  by the $\R$-action given by translations of the $s$ variable is a finite set, and \eqref{isocan2} defines a canonical count
\[
\#  \mathcal{M}_{\partial^F}(\gamma_+,\gamma_-) \in ( \mathfrak{o}_{\gamma_+} \otimes \mathfrak{o}_{\gamma_-}^*)^* \cong \mathrm{Hom}_{\Z}(\mathfrak{o}_{\gamma_+} , \mathfrak{o}_{\gamma_-}),
\]
which defines the boundary operator in the Floer complex of $(H,J)$:
\[
\partial^F =  \bigoplus_{\substack{(\gamma_+,\gamma_-)\in \mathcal{P}(H)^2\\ \mathrm{CZ}(\gamma_+) - \mathrm{CZ}(\gamma_-) = 1}} \#  \mathcal{M}_{\partial^F}(\gamma_+,\gamma_-): F_*(H) \rightarrow F_{*-1}(H).
\]

\subsection{The isomorphism}

Let $J\in C^{\infty}((-\infty,0] \times \T \times \R^{2n},\mathcal{J}(\R^{2n},\omega_0))$ be uniformly bounded map such that $J(s,t,z)=J_0$ for every $(s,t,z)\in [-1,0]\times \T \times \R^{2n}$ and $J(s,t,z)$ is independent of $s$ for $s$ small enough. Let $g$ be a Riemannian metric on the submanifold 
\[
M= \mathop{\mathrm{graph}}(Y_N : \mathbb{H}_N \rightarrow \mathbb{H}^N) \subset H^1_0(\T,\R^{2n})
\]
which is uniformly equivalent to the one induced by the ambient Hilbert space.

Recall that $\mathbb{P}_+$ denotes the orthogonal projector onto the subspace $\mathbb{H}_+$ of $H^{\frac{1}{2}}(\T,\R^{2n})$ consisting of loops $\gamma$ such that $\widehat{\gamma}(k)=0$ for every $k\leq 0$. We denote by the same symbol the projector on $H^1_0(\T,\R^{2n})$ having the same definition and note that the restriction of $\mathbb{P}_+$ to $M$ is a diffeomorphism onto the finite dimensional vector space $\mathbb{H}_N\subset \mathbb{H}_+$. In particular, for every critical point $p\in M$ of $\psi_{H^*}$ the set $\mathbb{P}_+ W^u(p)$ is a submanifold of dimension $i_{\mathrm{M}}(p)$ of $\mathbb{H}_+$.

Given a critical point $p\in M$ of $\psi_{H^*}$ and a 1-periodic orbit $\gamma\in \mathcal{P}(H)$, we consider the following space of Floer negative half-cylinders:
\[
\begin{split}
\mathcal{M}_{\Theta}(p,\gamma) := \{ u\in C^{\infty}((-\infty,0]\times \T,\R^{2n}) \mid & \; \partial_s u + J(s,t,u)(\partial_t u - X_H(t,u)) = 0, \\ & \lim_{s\rightarrow -\infty} u(s,\cdot) = \gamma, \; u(0,\cdot) \in \mathbb{P}_+^{-1}(\mathbb{P}_+W^u(p)) \}.
\end{split}
\]
The last condition says that the loop $u(0,\cdot)$ has the form
\[
u(0,t) = q(t) + \beta(t) \qquad \mbox{where } q\in W^u(p)\subset M \mbox{ and } \beta\in \mathbb{H}_0 \oplus \mathbb{H}_-.
\]
The space $\mathcal{M}_{\Theta}(p,\gamma)$ is the inverse image of the submanifold $\{0\}\times \mathbb{P}_+W^u(p)$ by the smooth nonlinear map
\begin{equation}
\label{maphybrid}
\begin{split}
\gamma + H^1((-\infty,0)\times \T,\R^{2n}) \rightarrow L^2((-\infty,0)\times \T,\R^{2n}) \times \mathbb{H}_+,  \\
u \mapsto \bigl( \partial_s u + J(s,t,u)(\partial_t u - X_H(t,u)) , \mathbb{P}_+ u(0,\cdot) \bigr).
\end{split}
\end{equation}
The differential of the above map at an element $u\in \mathcal{M}_{\Theta}(p,\gamma)$ is an operator $D_u$ belonging to the space $\mathcal{D}_-(\gamma)$. By Proposition \ref{hybridop}, $D_u$ is Fredholm of index $n-\mathrm{CZ}(\gamma)$. 

If $J$ and $g$ are generic, the above map is transverse to the submanifold $\{0\}\times \mathbb{P}_+W^u(p)$ and hence $\mathcal{M}_{\Theta}(p,\gamma)$ is a smooth manifold of dimension
\[
\dim \mathcal{M}_{\Theta}(p,\gamma) = \ind D_u + \dim  W^u(p) = n - \mathrm{CZ}(\gamma) + i_{\mathrm{M}}(p).
\]
If $\gamma_p$ denotes the 1-periodic orbit corresponding to $p$ in the bijective correspondence between critical points of $\Phi_H$ and $\psi_{H^*}$, i.e.\ $\mathbb{P}(\gamma_p)=p$, identity \eqref{indici} implies that
\[
\dim \mathcal{M}_{\Theta}(p,\gamma) = \mathrm{CZ}(\gamma_p) - \mathrm{CZ}(\gamma).
\]
The tangent space of $\mathcal{M}_{\Theta}(p,\gamma)$ at $u$ is given by
\[
T_u \mathcal{M}_{\Theta} (p,\gamma) = D_u^{-1} \bigl( (0) \times \mathbb{P}_+ T_q W^u(p) \bigr),
\]
where
\[
q := (\mathrm{id} \times Y_N) ( \mathbb{P}_+ u(0,\cdot)) \in W^u(p).
\]
The canonical isomorphism \eqref{isoor} reads
\[
\mathfrak{o}_{D_u} \cong \mathfrak{o}_{T_u \mathcal{M}_{\Theta} (p,\gamma)} \otimes \mathfrak{o}_{T_q W^u(p)}^*.
\]
Since $\mathfrak{o}_p = \mathfrak{o}_{T_p W^u(p)} \cong \mathfrak{o}_{T_q W^u(p)} $, $\mathfrak{o}_{\gamma}$ is identified with $\mathfrak{o}_{\mathcal{D}_-(\gamma)}^*$, and $D_u\in\mathcal{D}_-(\gamma)$, we obtain the canonical isomorphism
\begin{equation}
\label{orMTheta}
\mathfrak{o}_{T_u \mathcal{M}_{\Theta}(p,\gamma)} \cong \mathfrak{o}_p \otimes \mathfrak{o}_{\gamma}^*.
\end{equation}

The compactness up to breaking of the spaces $\mathcal{M}_{\Theta}(p,\gamma)$ follows from the inequality
\begin{equation}
\label{ines}
\mathcal{A}_H(\beta+\eta) \leq \Psi_{H^*}(\mathbb{P}(\beta)) - \frac{1}{2} \|\mathbb{P}_- \eta\|^2_{H^{\frac{1}{2}}} \qquad \forall \beta\in H^1(\T,\R^{2n}), \; \forall \eta\in \mathbb{H}_0 \oplus \mathbb{H}_-,
\end{equation}
with equality holding if and only if $\beta' = X_H(\beta+\eta)$ almost everywhere, which gives the uniform energy estimates for the elements of $\mathcal{M}_{\Theta}(p,\gamma)$ and hence the starting point for proving compactness, see \cite[Section 9]{ak22}.

Now assume that $p\in \mathrm{crit}(\psi_{H^*})$ and $\gamma\in \mathcal{P}(H)$ satisfy
\[
\mathrm{CZ} (\gamma) = i_{\mathrm{M}}(p) + n.
\]
In this case, $\mathcal{M}_{\Theta}(p,\gamma)$ is $0$-dimensional and compact, and hence a finite set. By \eqref{orMTheta}, the choice of a generator of $\mathfrak{o}_p \otimes \mathfrak{o}_{\gamma}^*$ determines an orientation of this finite set, and hence a $\Z$-valued count of it. This gives us a canonical count
\[
\# \mathcal{M}_{\Theta}(p,\gamma) \in ( \mathfrak{o}_p \otimes \mathfrak{o}_{\gamma}^*)^* \cong \mathrm{Hom}_{\Z}( \mathfrak{o}_p, \mathfrak{o}_{\gamma}).
\]
We can then define the homomorphism
\[
\Theta: M_*(\psi_{H^*}) \rightarrow F_{*+n}(H)
\]
as
\[
\Theta := \bigoplus_{\substack{(p,\gamma) \in \mathrm{crit} (\psi_{H^*}) \times \mathcal{P}(H) \\ \mathrm{CZ}(\gamma) - i_{\mathrm{M}}(p) = n}} \# \mathcal{M}_{\Theta}(p,\gamma).
\]
The arguments from \cite[Section 10]{ak22} show that $\Theta$ is an isomorphism preserving the action filtrations. This uses the inequality \eqref{ines}. In order to prove Theorem \ref{isom}, we just have to show that $\Theta$ is a chain map. This is done in the next section.

\subsection{Proof of the chain map property}

It is convenient to fix generators $o_p$ and $o_{\gamma}$ of the orientation lines $\mathfrak{o}_p= \mathfrak{o}_{T_p W^u(p)}$ and $\mathfrak{o}_{\gamma}= \mathfrak{o}_{\mathcal{D}_-(\gamma)}^*$, for every $p\in \mathrm{crit}(\psi_{H^*})$ and $\gamma\in \mathcal{P}(H)$.
Thanks to \eqref{canMorse}, \eqref{isocan2}, and \eqref{orMTheta}, these generators induce:
\begin{enumerate}[(i)]
\item An orientation of $W^u(p_+)\cap W^s(p_-)$ for every pair of critical points $p_+,p_-$ of $\psi_{H^*}$. In particular, if ${i}_M(p_+) - {i}_M(p_-)=1$ we have
\[
\# \mathcal{M}_{\partial^M}(p_+,p_-) \, o_{p_+} =  \sum_{\widehat{x}\in \mathcal{M}_{\partial^M}(p_+,p_-)} \epsilon(\widehat{x}) \, o_{p_-},
\]
where $\epsilon(\widehat{x})\in \{\pm1\}$ takes the value  $+1$ if and only if along the orbit $\widehat{x}\subset W^u(p_+)\cap W^s(p_-)$ the gradient of $\psi_{H^*}$ is positively oriented.
\item An orientation of $\mathcal{M}_{\partial^F}^{\#}(\gamma_+,\gamma_-)$ for every pair of 1-periodic orbits $\gamma_+,\gamma_-$ of $X_H$. In particular, if $\mathrm{CZ}(\gamma_+) - \mathrm{CZ}(\gamma_-)=1$  we have
\[
\#\mathcal{M}_{\partial^F}(\gamma_+,\gamma_-) \, o_{\gamma_+} = \sum_{\widehat{v} \in \mathcal{M}_{\partial^F}(\gamma_+,\gamma_-)} \epsilon(\widehat{v})\, o_{\gamma_-},
\]
where $\epsilon(\widehat{v})\in \{\pm1\}$ takes the value  $+1$ if and only if  the $\R$-action given by positive translations of $s$ is positively oriented on the connected component of $\mathcal{M}^{\#}_{\partial^F}(\gamma_+,\gamma_-)$ defined by $\widehat{v}$.
\item An orientation of $\mathcal{M}_{\Theta}(p,\gamma)$ for every $p\in \mathrm{crit}(\psi_{H^*})$ and $\gamma\in \mathcal{P}(H)$. In particular, if $\mathrm{CZ}(\gamma) - i_{\mathrm{M}}(p)=n$ we have
\[
\# \mathcal{M}_{\Theta}(p,\gamma) \, o_{p} = \sum_{u\in \mathcal{M}_{\Theta}(p,\gamma)} \epsilon(u) \, o_{\gamma},
\]
where $\epsilon(u)\in \{\pm1\}$ gives us the orientation of the singleton $\{u\}$ in the $0$-dimensional manifold $\mathcal{M}_{\Theta}(p,\gamma)$.
\end{enumerate}
With these choices, the homomorphisms $\partial^M$, $\partial^F$ and $\Theta$ have the form:
\[
\begin{split}
\partial^M o_p &= \sum_{\substack{q\in \mathrm{crit}(\psi_{H^*}) \\ i_\mathrm{M}(q) = i_\mathrm{M}(p) - 1}} \Bigl( \sum_{\widehat{x}\in \mathcal{M}_{\partial^M}(p,q)} \epsilon(\widehat{x}) \Bigr) o_q, \\
\partial^F o_{\gamma} &= \sum_{\substack{\beta\in \mathcal{P}(H) \\ \mathrm{CZ}(\beta) = \mathrm{CZ}(\gamma) - 1}} \Bigl( \sum_{\widehat{v}\in \mathcal{M}_{\partial^F}(\gamma,\beta)} \epsilon(\widehat{v}) \Bigr) o_{\beta}, \\
\Theta \, o_p &= \sum_{\substack{\gamma \in \mathcal{P}(H) \\ \mathrm{CZ}(\gamma) = i_{\mathrm{M}}(p)+n}} \Bigl( \sum_{u\in \mathcal{M}_{\Theta}(p,\gamma)} \epsilon(u) \Bigr) o_{\gamma}.
\end{split} 
\]
Let $p$ be a critical point of $\psi_{H^*}$ with $i_\mathrm{M}(p) = k-n$. From the above formulas, we find
\begin{eqnarray*}
\partial^F \Theta \, o_p &=& \sum_{\substack{\gamma\in \mathcal{P}(H) \\ \mathrm{CZ}(\gamma) = k-1}} \Bigl( \sum_{\substack{\beta\in \mathcal{P}(H) \\ \mathrm{CZ}(\beta) = k}} \sum_{\substack{u\in \mathcal{M}_{\Theta}(p,\beta) \\ \widehat{v} \in \mathcal{M}_{\partial^F} (\beta,\gamma)}} \epsilon(u) \epsilon(\widehat{v}) \Bigr) o_{\gamma}, \\
\Theta \partial^M \, o_p &=& \sum_{\substack{\gamma\in \mathcal{P}(H) \\ \mathrm{CZ}(\gamma) = k-1}} \Bigl( \sum_{\substack{q\in \mathrm{crit}(\psi_{H^*}) \\ i_{\mathrm{M}}(q) = k-n-1}} \sum_{\substack{\widehat{x} \in \mathcal{M}_{\partial^M}(p,q) \\ u \in \mathcal{M}_{\Theta} (q,\gamma)}} \epsilon(\widehat{x}) \epsilon(u) \Bigr) o_{\gamma},
\end{eqnarray*}
and the identity $\partial^F \Theta o_p=\Theta \partial^M o_p$ is equivalent to
\begin{equation}
\label{tobeproven}
\sum_{\substack{\beta\in \mathcal{P}(H) \\ \mathrm{CZ}(\beta) = k}} \sum_{\substack{u\in \mathcal{M}_{\Theta}(p,\beta) \\ \widehat{v} \in \mathcal{M}_{\partial^F} (\beta,\gamma)}} \epsilon(u) \epsilon(\widehat{v})  = \sum_{\substack{q\in \mathrm{crit}(\psi_{H^*}) \\ i_{\mathrm{M}}(q) = k-n-1}} \sum_{\substack{\widehat{x} \in \mathcal{M}_{\partial^M}(p,q) \\ u \in \mathcal{M}_{\Theta} (q,\gamma)}} \epsilon(\widehat{x}) \epsilon(u),
\end{equation}
for every $\gamma\in \mathcal{P}(H)$ with $\mathrm{CZ}(\gamma)=k-1$. The sums appearing in the above expression are indexed by the set $A \cup B$, where
\[
A:= \bigcup_{\substack{\beta\in \mathcal{P}(H) \\ \mathrm{CZ}(\beta) = k}} \mathcal{M}_{\Theta}(p,\beta) \times \mathcal{M}_{\partial^F}(\beta,\gamma), \qquad B:=  \bigcup_{\substack{q\in \mathrm{crit}(\psi_{H^*}) \\ i_{\mathrm{M}}(q) = k-n-1}} \mathcal{M}_{\partial^M} (p,q) \times \mathcal{M}_{\Theta}(q,\gamma).
\]
By a standard gluing argument, the finite set $A\cup B$ is precisely the set of boundary points of the compactification of the 1-dimensional manifold $\mathcal{M}_{\Theta}(p,\gamma)$: the non-compact connected components of $\mathcal{M}_{\Theta}(p,\gamma)$ are arcs connecting two distinct elements from $A\cup B$, and each element of $A\cup B$ is an end-point of precisely one arc in $\mathcal{M}_{\Theta}(p,\gamma)$. By (iii), the arcs in $\mathcal{M}_{\Theta}(p,\gamma)$ are oriented, and hence each of its end-points carries an orientation sign $\pm 1$. The identity \eqref{tobeproven} is then an easy consequence of the next two claims:
\begin{enumerate}[(a)]
\item The orientation sign of every $(u,\widehat{v})\in \mathcal{M}_{\Theta}(p,\beta) \times \mathcal{M}_{\partial^F}(\beta,\gamma)$ is $\epsilon(u)\epsilon(\widehat{v})$.
\item The orientation sign of every $(\widehat{x},u)\in \mathcal{M}_{\partial^M}(p,q) \times \mathcal{M}_{\Theta}(q,\gamma)$ is $-\epsilon(\widehat{x})\epsilon(u)$.
\end{enumerate}

\paragraph{\sc Proof of ($\mathrm{a}$).} 
We here consider an end-point of the form $(u,\widehat{v})\in \mathcal{M}_{\Theta}(p,\beta) \times \mathcal{M}_{\partial^F} (\beta,\gamma)$. Let $r\mapsto w_r$, $r\in (0,1)$, be a smooth parametrisation of a connected component of  $\mathcal{M}_{\Theta}(p,\gamma)$ such that $w_r$ converges to $(u,\widehat{v})$ for $r\rightarrow 0$. This means that $w_r$ has the following asymptotic behaviour for $r\rightarrow 0$: $w_r\rightarrow u$ in $C^{\infty}_{\mathrm{loc}}((-\infty,0]\times \T,\R^{2n})$ and there exists a smooth positive function $\sigma$ on $(0,1)$ such that $\sigma'<0$, $\sigma(r)\rightarrow +\infty$ for $r\rightarrow 0$ and $\tau_{\sigma(r)} w_r$ converges  in $C^{\infty}_{\mathrm{loc}}(\R\times \T,\R^{2n})$ to an element $v\in \mathcal{M}_{\partial_F}^{\#}(\beta,\gamma)$ in the equivalence class $\widehat{v}$. Here, $\tau_{\sigma}$ denotes the translation operator
\[
\tau_{\sigma} w (s,t) := w(s-\sigma,t).
\]
The smooth curve 
\[
x(r) := (\mathrm{id}\times Y_N) \bigl( \mathbb{P}_+ w_r(0,\cdot) \bigr) \in W^u(p), \qquad r\in (0,1),
\]
converges to $x(0):= (\mathrm{id}\times Y_N) (\mathbb{P}_+ u(0,\cdot)) \in W^u(p)$ for $r\rightarrow 0$. 

We denote by $D_u\in \mathcal{D}_-(\beta)$, $D_v\in \mathcal{D}(\beta,\gamma)$, and $D_{w_r}\in \mathcal{D}_-(\gamma)$ the linearized operators at $u$, $v$, and $w_r$ respectively. The operator $D_u$ is Fredholm of index $n-k$ and is transverse to $(0) \times \mathbb{P}_+ T_{x(0)} W^u(p)$ with
\[
D_u^{-1}  \bigl( (0) \times \mathbb{P}_+ T_{x(0)} W^u(p) \bigr) = (0).
\]
By \eqref{isoor} we obtain the isomorphism
\begin{equation}
\label{Du}
\mathfrak{o}_{D_u} \cong \mathfrak{o}_{T_{x(0)} W^u(p)}^* \cong  \mathfrak{o}_{T_p W^u(p)}^* = \mathfrak{o}_p^*, \quad o_{\beta}^* \mapsto \epsilon(u) o_p^*.
\end{equation}
where the superscript $*$ indicates dual generators. The operator $D_v$ is Fredholm of index 1, surjective and with kernel given by
\[
\ker D_v = T_v \mathcal{M}_{\partial^F}^{\#} (\beta,\gamma) =  \mathrm{Span} \{\partial_s v \}.
\]
Denoting by $[y]$ the canonical generator of the orientation line $\mathfrak{o}_{\R y}$ induced by the non-vanishing vector $y$, we see that
\begin{equation}
\label{Dv}
\mathfrak{o}_{D_v} = \mathfrak{o}_{T_v \mathcal{M}_{\partial^F}^{\#} (\beta,\gamma)} \mbox{ has the positive generator } \epsilon(\widehat{v}) [\partial_s v],
\end{equation}
with respect to the orientation of $\mathcal{M}_{\partial^F}^{\#} (\beta,\gamma)$ chosen in (ii). By a simple variant of \cite[Proposition 9]{fh93}, when $r$ is small enough the glued operator $D_u \#_{\sigma(r)} D_v\in \mathcal{D}_-(\gamma)$ is transverse to $(0) \times \mathbb{P}_+ T_{x(r)} W^u(p)$ and the subspace
\[
(D_u \#_{\sigma(r)} D_v)^{-1}((0) \times \mathbb{P}_+ T_{x(r)} W^u(p))
\]
is 1-dimensional and close to the line
\[
V_r:= \mathrm{Span}\{v_r\}, \quad \mbox{where} \quad v_r:= \tau_{-\sigma(r)} \partial_s v|_{(-\infty,0)\times \T}.
\]
Thanks to \eqref{Du}, \eqref{Dv}, and the canonical isomorphism $\mathfrak{o}_{D_u} \otimes  \mathfrak{o}_{D_v}  \cong \mathfrak{o}_{D_u \#_{\sigma(r)} D_v}$, the isomorphism \eqref{isoor} reads
\begin{equation}
\label{DuDv}
\mathfrak{o}_{D_u \#_{\sigma(r)} D_v} \cong \mathfrak{o}_{V_r} \otimes \mathfrak{o}_{T_{x(r)} W^u(p)}^* \cong  \mathfrak{o}_{V_r} \otimes \mathfrak{o}_p^*, \quad o_{\gamma}^* \mapsto \epsilon(u) \epsilon(\widehat{v}) [v_r] \otimes o_p^*.
\end{equation}
The operator $D_{w_r}$ is Fredholm of index $n-k+1$ and is transverse to $(0) \times \mathbb{P}_+ T_{x(r)} W^u(p)$ with
\[
D_{w_r}^{-1} \bigl( (0) \times \mathbb{P}_+ T_{x(r)} W^u(p) \bigr) = T_{w_r} \mathcal{M}_{\Theta}(p,\gamma) = \mathrm{Span} \{\partial_r w_r \}.
\]
Therefore, \eqref{isoor} gives us the isomorphism
\begin{equation}
\label{Dwr}
\mathfrak{o}_{D_{w_r}} \cong \mathfrak{o}_{\mathrm{Span} \{\partial_r w_r \}} \otimes \mathfrak{o}_{T_{x(r)} W^u(p)}^*\cong 
\mathfrak{o}_{\mathrm{Span} \{\partial_r w_r \}} \otimes \mathfrak{o}_p^*.
\end{equation}
Since $D_{w_r}$ is close to $D_u \#_{\sigma(r)} D_v$ for $r$ small, we deduce that the lines $\mathrm{Span} \{\partial_r w_r \}$ and $V_r$
are close. Therefore, when $r$ is small it makes sense to compare the orientations of two generators of these two lines.

When $r$ is small, the function $w_r: (-\infty,0]\times \T \rightarrow \R^{2n}$ is close to the function $u$ on $(-\frac{\sigma(r)}{2},0] \times \T$ and close to the function $\tau_{-\sigma(r)} v$ on $(-\infty,-\frac{\sigma(r)}{2})\times \T$. Actually, $w_{\sigma(r)}$ is obtained from a Newton iteration scheme starting from a function which is obtained by glueing $u|_{(-\frac{\sigma(r)}{2},0] \times \T}$ and $\tau_{-\sigma(r)} v|_{(-\infty,\frac{\sigma(r)}{2}]\times \T}$. 
The analysis of this iteration scheme shows that the derivative with respect to $r$ of $w_r$ is close to the derivative with respect to $r$ of this glued function, i.e.\ with 
\[
\partial_r \tau_{-\sigma(r)} v = \sigma'(r) \tau_{-\sigma(r)} \partial_s v.
\]
Since $\sigma'<0$, this shows that the natural generators of the lines $\mathrm{Span} \{\partial_r w_r \}$ and $V_r$ have opposite orientations:
\[
[\partial_r w_r] = - [ v_r].
\]
Together with \eqref{DuDv} and \eqref{Dwr}, the above identity implies that the orientation of the line
\[
T_{w_r} \mathcal{M}_{\Theta}(p,\gamma) = \mathrm{Span} \{\partial_r w_r\} 
\]
from (iii) is given by $-\epsilon(u) \epsilon(\widehat{v}) [\partial_r w_r]$. 

Since $[\partial_r]$ induces the orientation $-1$ on the singleton $\{0\}$ seen as boundary of $[0,1)$, we conclude that the orientation of the singleton $(u,\widehat{v})$ as end-point of an arc in the oriented 1-dimensional manifold $\mathcal{M}_{\Theta}(p,\gamma)$ is $\epsilon(u) \epsilon(\widehat{v})$, as claimed.

\paragraph{\sc Proof of ($\mathrm{b}$).} We here consider an end-point of the form $(\widehat{x},u)\in \mathcal{M}_{\partial^M} (p,q) \times \mathcal{M}_{\Theta}(q,\gamma)$. Let $r\mapsto w_r$, $r\in (0,1)$, be a smooth parametrisation of a connected component of  $\mathcal{M}_{\Theta}(p,\gamma)$ such that $w_r$ converges to $(\widehat{x},u)$ for $r\rightarrow 0$. This means the following: for $r\rightarrow 0$, $w_r \rightarrow u$ in $C^{\infty}((-\infty,0]\times \T,\R^{2n})$ and, setting 
\[
y(r):= (\mathrm{id} \times Y_N)(\mathbb{P}_+ w_r(0,\cdot)) \in W^u(p), \qquad r\in (0,1),
\]
 we have that 
 \[
 \lim_{r\rightarrow 0} y(r) = y(0):= (\mathrm{id} \times Y_N)(\mathbb{P}_+ u(0,\cdot))\in W^u(q),
 \] 
 and there exists a smooth positive function $\sigma$ on $(0,1)$ such that 
 \[
 \lim_{r\rightarrow 0} \sigma(r) = +\infty, \qquad \lim_{r\rightarrow 0} \phi^{-\sigma(r)}_{-\nabla \psi_{H^*}}(y(r)) =x, 
 \]
 where $x\in W^u(p)\cap W^s(q)$ is a point on the orbit $\widehat{x}$.

Let $\Sigma_0$ be a small embedded open $(k-n-1)$-dimensional ball in $W^u(p)$ such that $\Sigma_0 \cap W^s(q) = \{x\}$ with transverse intersection. Then $\Sigma_0$ has codimension 1 in $W^u(p)$, is transverse to $\nabla \psi_{H^*}$ at $x$ and hence everywhere, up to reducing its size. This implies that
\[
\Sigma_t := \phi^t_{-\nabla \psi_{H^*}}(\Sigma_0), \quad t\in \R,
\]
is a smooth foliation of codimension 1 of an open subset $U$ of $W^u(p)$ which is invariant under the negative gradient flow. The closure of $U$ contains $W^u(q)$ and the leaves $\Sigma_t$ $C^1$-converge to $W^u(q)$ for $t\rightarrow +\infty$ in the following sense: if $t_j\rightarrow +\infty$ and $z_j\in \Sigma_{t_j}$ converges to $z\in W^u(q)$ then $T_{z_j} \Sigma_{t_j}$ converges to $T_z W^u(q)$. Indeed, by dynamical continuation it is enough to check this for sequences $(z_j)$ which are contained in a small neighborhood of $q$ and the claim follows from the properties of the graph transform, see \cite[Proposition 2.1 (iv)]{ama23}.

By the $C^1$-convergence property, the leaves $\Sigma_t$ inherit an orientation from the orientation $o_q$ of $W^u(q)$, and we use the same symbol $o_q$ to denote this orientation of a leaf. 

Denote by $\tau: U \rightarrow \R$ the smooth function taking the value $t$ on $\Sigma_t$. The gradient of $\tau$ is transverse to the leaves and at $x$ is contained in the open half-space of $T_x W^u(p)$ defined by the hyperplane $T_x\Sigma_0$ and containing $-\nabla \psi_{H^*}(x)$. By the definition of $\epsilon(\widehat{x})$, we have, using the isomorphism \eqref{or-sum},
\[
o_p = \epsilon(\widehat{x}) \,[\nabla \psi_{H^*}(x)] \otimes o_q \quad \mbox{in the splitting} \quad T_x W^u(p) = \mathrm{Span} \{ \nabla \psi_{H^*}(x) \} \oplus T_x \Sigma_0,
\]
and hence 
\[
o_p = -\epsilon(\widehat{x})\, [\nabla \tau(x)] \otimes o_q \quad \mbox{in the splitting} \quad T_x W^u(p) = \mathrm{Span} \{ \nabla \tau(x) \} \oplus T_x \Sigma_0.
\]
By continuation, we have for every $z\in U$:
\begin{equation}
\label{fol}
o_p = -\epsilon(\widehat{x})\,[\nabla \tau(z)]  \otimes o_q \quad \mbox{in the splitting} \quad T_z W^u(p) = \mathrm{Span} \{ \nabla \tau(z) \} \oplus T_z \Sigma_{\tau(z)}.
\end{equation}

The fact that $\phi^{-\sigma(r)}_{-\nabla \psi_{H^*}}(y(r))$ converges to $x\in U$ for $r\rightarrow 0$ implies that there exists $r_0\in (0,1]$ such that $y(r)\in U$ for every $r\in (0,r_0)$. Since $\tau(y(r)) \rightarrow +\infty$ for $r\rightarrow 0$, we can find a sequence $(r_j)\subset (0,r_0)$ such that $r_j\rightarrow 0$ and
\[
\frac{d}{dr} \tau\circ y (r_j) = d\tau(y(r_j))[y'(r_j)] < 0.
\]
Therefore, the vector $y'(r_j)$ is transverse to $\Sigma_{\tau(y(r_j))}$ and from \eqref{fol} we obtain
\begin{equation}
\label{fol-y}
o_p = \epsilon(\widehat{x})\, [y'(r_j)] \otimes o_q \quad \mbox{in the splitting} \quad T_{y(r_j)} W^u(p) = \mathrm{Span} \{ y'(r_j) \} \oplus T_{y(r_j)} \Sigma_{\tau(y(r_j))}.
\end{equation}

Let $D_u$ and $D_{w_r}$ be the operators in $\mathcal{D}_-(\gamma)$ which are obtained by differentiating the map \eqref{maphybrid} at $u$ and $w_r$.

The Fredholm operator $D_u$ is injective and its image is transverse to $(0)\times \mathbb{P}_+ T_{y(0)} W^u(q)$. Then the dual of the isomorphism \eqref{isoor} reads
\begin{equation}
\label{oru}
\mathfrak{o}_{D_u}^* \cong \mathfrak{o}_{T_{y(0)} W^u(q))} \cong \mathfrak{o}_q, \qquad o_{\gamma} \mapsto \epsilon(u) o_q.
\end{equation}
The Fredholm operator $D_{w_{r_j}}$ is transverse to the subspace $(0)\times \mathbb{P}_+ T_{y(r_j)} W^u(p)$ and the isomorphism \eqref{isoor} gives us
\begin{equation}
\label{orw}
\mathfrak{o}_{D_{w_{r_j}}}^* \cong \mathfrak{o}_{T_{w_{r_j}} \mathcal{M}_{\Theta}(p,\gamma)}^* \otimes \mathfrak{o}_{T_{y(r_j)} W^u(p)}, \quad
o_{\gamma} \mapsto o_{p,\gamma}^* \otimes o_p = \epsilon(\widehat{x}) \, o_{p,\gamma}^* \otimes ( [y'(r_j)] \otimes o_q),
\end{equation}
where $o_{p,\gamma}$ is the orientation of $\mathcal{M}_{\Theta}(p,\gamma)$ from (iii) and we have used \eqref{fol-y}. Note that the
middle homomorphism in the exact sequence \eqref{exact} inducing the above isomorphism is
\[
0 \times (\mathbb{P}_+ \mathrm{tr}_0):T_{w_r} \mathcal{M}_{\Theta}(p,\gamma)  \rightarrow (0) \times \mathbb{P}_+T_{y(r)} W^u(p) \cong T_{y(r)} W^u(p), \qquad \partial_r w_r \mapsto y'(r),
\]
where $\mathrm{tr}_0$ is the trace operator at $s=0$. Since the Fredholm operators $D_u$ and $D_{w_{r_j}}$ are close to each other,  a comparison of \eqref{oru} and \eqref{orw} gives us the identity
\[
o_{p,\gamma} = \epsilon(\widehat{x}) \epsilon(u) [\partial_r w_r].
\]
Since $[\partial_r]$ induces the orientation $-1$ on the singleton $\{0\}$ seen as boundary of $[0,1)$, we conclude that the orientation of the singleton $(\widehat{x},u)$ as end-point of an arc in the oriented 1-dimensional manifold $\mathcal{M}_{\Theta}(p,\gamma)$ is $-\epsilon(\widehat{x}) \epsilon(u)$, as claimed.


\providecommand{\bysame}{\leavevmode\hbox to3em{\hrulefill}\thinspace}
\providecommand{\MR}{\relax\ifhmode\unskip\space\fi MR }
\providecommand{\MRhref}[2]{%
  \href{http://www.ams.org/mathscinet-getitem?mr=#1}{#2}
}
\providecommand{\href}[2]{#2}

\end{document}